\documentclass[12pt]{iopart} 
\usepackage[nolists]{endfloat}

\usepackage[top=2.5cm,bottom=2.5cm,left=2cm,right=2cm]{geometry}

\usepackage{graphicx,subfigure}
\usepackage{amssymb,verbatim,amsthm}
\newtheorem{theorem}{Theorem}[section]
 
\newtheorem{proposition}[theorem]{Proposition}
\newtheorem{corollary}[theorem]{Corollary}

\newtheorem{assumption}[theorem]{Assumption}
\newtheorem{remark}[theorem]{Remark}

\newcommand{\N}{\mathbb{N}}
\newcommand{\Z}{\mathbb{Z}}

\newcommand{\R}{\mathbb{R}}

\newcommand{\bbP}{\mathbb{P}}

\newcommand{\E}{\mathbb{E}}
\newcommand{\T}{\mathbb{T}}

\newcommand{\eps}{\sigma}
\newcommand{\Laplace}{\Delta}

\newcommand{\Pl}{P_{\lambda}}
\newcommand{\Ql}{Q_{\lambda}}
\newcommand{\Bl}{B_0({\eta})}
\newcommand{\Wl}{W_{\lambda}}
\newcommand{\Wc}{W_{\lambda}^c}

\newcommand{\cA}{\mathcal{A}}
\newcommand{\cB}{\mathcal{B}}
\newcommand{\cC}{\mathcal{C}}
\newcommand{\ls}{\lambda^{\star}}
\newcommand{\Be}{B_j}
\newcommand{\hu}{\widehat {m}}
\newcommand{\hC}{\widehat {C}}

\newcommand{\rd}{\mathrm{d}}
\newcommand{\pd}{\partial}
\newcommand{\abs}[1]{| #1 |}

\newcommand{\norm}[1]{\| #1 \|}

\newcommand{\inner}[2]{\langle #1 , #2 \rangle}

\begin{document}

\title{Stability of Filters for the Navier-Stokes Equation}
\author{C.\ E.\ A.\ Brett, 
K.\ F.\ Lam, K.\ J.\ H.\ Law, 
D.\ S.\ McCormick, M.\ R.\ Scott 
and A.\ M.\ Stuart}
\address{
Warwick Mathematics Institute, University of Warwick,
  Coventry CV4 7AL, UK}

\ead{a.m.stuart@warwick.ac.uk, k.j.h.law@warwick.ac.uk}

\begin{abstract}
Data assimilation methodologies are designed 
to incorporate noisy observations
of a physical system into an underlying model 
in order to infer the properties of the state 
of the system. 
Filters refer to a class of data assimilation
algorithms designed to update the estimation of the state
in a on-line fashion, as data is acquired sequentially.
For linear problems subject
to Gaussian noise filtering can be performed exactly using
the Kalman filter. For nonlinear systems it can be
approximated in a systematic way by particle filters.
However in high dimensions these particle filtering methods
can break down. Hence, for the large nonlinear systems 
arising in applications such as weather forecasting, 
various {\em ad hoc} filters are used,
mostly based on making Gaussian approximations. The purpose
of this work is to study the properties of these 
{\em ad hoc} filters, working in the context of the 2D 
incompressible Navier-Stokes equation.
By working in this infinite dimensional setting we provide an
analysis which is useful for the understanding
of high dimensional filtering, and is robust to mesh-refinement.
We describe theoretical results showing that, in the
small observational noise limit, the filters
can be tuned to perform accurately in tracking the signal
itself (filter stability), 
provided the system is observed in a sufficiently large
low dimensional space; roughly speaking this space
should be large enough to contain the unstable modes of the
linearized dynamics. Numerical results are given which
illustrate the theory. In a simplified scenario we also
derive, and study numerically,
a stochastic PDE which determines filter
stability in the limit of frequent observations, subject
to large observational noise. 
The positive results herein concerning
filter stability complement recent numerical
studies which demonstrate that the {\em ad hoc} filters 
perform poorly in reproducing statistical variation about the true
signal.
\end{abstract}

\pacs{}
\date{\today}

\maketitle

\section{Introduction}
\label{sec:intro}

Assimilating large  
data sets into mathematical models of time-evolving
systems presents a major challenge in a wide range of 
applications. Since the data and the model are often 
uncertain, a natural overarching framework
for the formulation of such problems is that of Bayesian
statistics. However, for high dimensional models,
investigation of the Bayesian posterior distribution
of model state given data is not computationally feasible
in on-line situations. For this reason various {\em ad hoc} 
filters are used. The purpose of this paper is to provide an
analysis of such filters. 

The paradigmatic example of data assimilation is weather 
forecasting: we have many complex models to predict the state 
of the atmosphere, currently involving 
on the order of ${\mathcal O}(10^{8})$ 
unknowns, but we cannot know the exact state of the 
system at any one time; we thus have to contend with an initial 
condition which is only known incompletely. This is compensated
for by a large number, currently on the order of 
${\mathcal O}(10^{6})$, partial observations of the
atmosphere at subsequent times. Filters are widely used to
make forecasts which combine the mathematical model
of the atmosphere and the data to make predictions. 
Indeed the particular method of data assimilation which we 
study here includes, as a special case, the algorithm 
commonly known as 3DVAR. This method originated in weather 
forecasting. It was first proposed 
at the UK 
Met Office in 1986 \cite{article:Lorenc1986}, and was developed 
by the US National Oceanic and Atmospheric Administration 
(NOAA) soon thereafter; see \cite{article:Parrish1992}.
More details of the implementation of 3DVAR by the UK Met 
Office can be found in \cite{article:Lorenc2000}, and by 
the European Centre for Medium-Range Weather Forecasts (ECMWF) 
in \cite{article:Courtier1998}. The 3DVAR algorithm is 
prototypical of the many more sophisticated filters which are 
now widely used in practice  and 
it is thus natural to study it. The reader should be aware,
however, that the development of new filters is a very
active area and that the analysis here constitutes an initial
step towards the analyses required for these  more
recently developed algorithms. For insight into some
of these more sophisticated filters see 
\cite{toth1997ensemble,evensen2009data,VL09,
harlim2008filtering,majda2010mathematical} and the references therein.

Filtering can be performed exactly
for linear systems subject to Gaussian noise: 
the Kalman filter \cite{harvey1991forecasting}.
For nonlinear or non-Gaussian scenarios the particle
filter \cite{doucet2001sequential}
may be used and provably approximates the desired
probability distribution as the number of particles
is increased \cite{bain2008fundamentals}. However
in practice this method performs poorly in high
dimensional systems \cite{Bickel}. Whilst there
is considerable research activity
aimed at overcoming this degeneration 
\cite{van2010nonlinear, chorin2010implicit},
the methodology cannot yet be viewed
as a provably accurate tool within the context of the
high dimensional problems arising in geophysical
data assimilation.  In order to circumvent problems associated
with the representation of high dimensional probability
distributions some form of {\em ad hoc}
Gaussian approximation is typically
used to create practical filters, and the 3DVAR method
which we analyze here is perhaps the simplest example of this.

In the paper \cite{lawstuart} a wide range of Gaussian
approximate filters, including 3DVAR,
are evaluated by comparing the
distributions they produce with a highly accurate
(and impractical in realistic online scenarios) 
MCMC simulation of the
desired distributions. The conclusion of that work
is that the Gaussian approximate filters perform well
in tracking the mean of the desired distribution, but poorly
in terms of statistical variability about the mean. In
this paper we provide a theoretical analysis of the
ability of the filters to estimate the mean state accurately.
Although filtering is widely used in practice, much of the 
analysis of it, in the context of fluid mechanics, works with 
finite-dimensional dynamical models. Our aim is to work directly
with a PDE model relevant in fluid mechanics, the Navier-Stokes 
equation, and thereby confront the high-dimensional nature of 
the problem head-on. 
Study of the stability of filters for data assimilation has
been a developing research area over the last few years and
the paper \cite{carrassi2008data} contains a finite
dimensional  theoretical result, 
numerical experiments in 
a variety of finite and (discretized) infinite dimensional
systems not covered by the theory, and references to
relevant applied literature. 
Our analysis will build in a concrete fashion
on the approach in 
\cite{olson2003determining} and \cite{hayden2011discrete}
which were amongst the first to study data assimilation directly
through PDE analysis, using ideas from the theory of determining
modes in infinite dimensional dynamical systems.
However, in contrast to those papers,
we will allow for noisy observations in our analysis. 
Nonetheless the estimates in \cite{hayden2011discrete} form
an important component of our analysis.

The presentation will be organized as follows: in Section
\ref{sec:NSE} we introduce the Navier-Stokes equation as the forward
model; in Section \ref{sec:filters} we formulate data
assimilation as a Bayesian inverse problem and show
how approximate Gaussian filters can be derived,
leading to 3DVAR and generalizations;  in 
Section \ref{sec:stability} we introduce notions of
stability and prove the Main Theorems \ref{t:m}
and \ref{t:mz} concerning filter stability for
sufficiently small observational noise; 
section \ref{sec:stability} also contains derivation
of the stochastic PDE (SPDE) and deterministic PDE
which may be used to
study filter stability for 3DVAR
when data is aquired frequently
in time  and the observational noise is large;
in Section \ref{sec:numerics}
we present numerical results which corroborate the analysis;
and finally, in Section \ref{sec:conclusions} we present conclusions.
The mathematical tools required to follow the arguments
in this paper comprise basic ideas from probability
on Hilbert space, properties of the Navier-Stokes
equation, analysis of nonautonomous/random
dynamical systems, and properties of stochastic PDEs. 
Together these tools facilitate an analysis which allows
mathematical study of filtering and, we believe, 
can be built upon to study other problems arising in 
the filtering of complex systems.

\section{Forward Model: Navier-Stokes equation}
\label{sec:NSE}

In this section we establish a notation for,
and describe the properties of, the 
Navier-Stokes equation. This is the forward model which
underlies the inverse problem which we study in this paper.
We consider the 2D Navier-Stokes equation on the torus
$\T^{2} := [0,L) \times [0,L)$ with periodic boundary conditions:
\begin{eqnarray*}
\begin{array}{cccc}
\pd_{t}u - \nu \Laplace u + u \cdot \nabla u + \nabla p &=& f 
& {\rm for~ all ~} 
(x, t) \in \T^{2} \times (0, \infty), \\
\nabla \cdot u &=& 0 &{\rm for~ all~ }
(x, t) \in \T^{2} \times (0, \infty), \\
u(x, 0) &=& u_{0}(x) &{\rm for~ all ~}
x \in \T^{2}.
\end{array}
\label{eq:NSE}
\end{eqnarray*}

Here $u \colon \T^{2} \times (0, \infty) \to \R^{2}$ is a time-dependent vector field representing the velocity, $p \colon \T^{2} \times (0,\infty) \to \R$ is a time-dependent scalar field representing the pressure, $f \colon \T^{2} \to \R^{2}$ is a vector 
field representing the forcing (which we assume to be
time-independent for simplicity), and $\nu$ is the viscosity. 
In numerical simulations (see section~\ref{sec:numerics}), we
typically represent the solution via  
the vorticity $w$ and stream function $\zeta$; these 
are related through $u = \nabla^{\perp} \zeta$ and $w=\nabla^{\perp} \cdot u$, 
where $\nabla^{\perp} = (\pd_{2}, -\pd_{1})^{\mathrm{T}}$.
We define
$${\mathcal H}:= \left\{ L{\rm {-periodic\,trigonometric\,polynomials\,}} 
u:[0,L)^2 \to {\mathbb R}^2\,\Bigl|\, \nabla \cdot u = 0, \,\int_{\T^{2}} u(x) \, \rd x = 0 \right\}
$$
and $H$ as the closure of ${\mathcal H}$ with respect to the
$(L^{2}(\T^{2}))^{2}$ norm. We define $P:(L^{2}(\T^{2}))^{2} 
\to H$ to be the Leray-Helmholtz orthogonal projector.

Given $k = (k_{1}, k_{2})^{\mathrm{T}}$, define $k^{\perp} := (k_{2}, -k_{1})^{\mathrm{T}}$. Then an orthonormal basis for
$H$ is given by $\psi_{k} \colon \R^{2} \to \R^{2}$, where 
$$\psi_{k} (x) := \frac{k^{\perp}}{|k|} \exp\Bigl(\frac{2 \pi i k \cdot x}{L}\Bigr)$$ for $k \in \Z^{2} \setminus \{0\}$. 
Thus for $u \in H$ we may write
$$u = \sum_{k \in \Z^{2} \setminus \{0\}} u_{k}(t) \psi_{k}(x)$$
where, since $u$ is a real-valued function, we have the 
reality constraint $u_{-k} = - \overline{u_{k}}.$
We then define the projection operators $\Pl: H \to H$ 
and $\Ql:H \to H$ by
$$\Pl u = \sum_{|2\pi k|^2 <\lambda L^2} u_{k}(t) \psi_{k}(x),
\quad \Ql=I-\Pl.$$
We let $\Wl=\Pl H$ and $\Wc=\Ql H.$

Using the Fourier decomposition of $u$, we define the
fractional Sobolev spaces
\begin{equation}
\label{eq:Hs} 
H^{s}:= \left\{ u \in H : \sum_{k \in \Z^{2} \setminus \{0\}} (4\pi^{2}\abs{k}^{2})^{s}\abs{u_{k}}^{2} < \infty \right\}
\end{equation}
with the norm $\norm{u}_{s}:= \bigl(\sum_{k} (4\pi^{2}\abs{k}^{2})^{s}
\abs{u_{k}}^{2}\bigr)^{1/2}$, where $s \in \R$. We use
the abbreviated notation $\norm{u}$ for the norm on $H^1$,
and $|\cdot|$ for the norm on $H=H^0$. 

Applyng the projection $P$ to the Navier-Stokes
equation we may write it as an ODE (ordimary differential
equation) in $H$; see
\cite{constantin1988navier, temam1995navier, book:Robinson2001} 
for details. This ODE takes the form
\begin{equation}
\frac{\rd u}{\rd t} + \nu Au + \cB(u, u) = f, \quad u(0)=u_0.
\label{eq:nse}
\end{equation}
Here, $A = -P \Laplace$ is the Stokes operator, 
the term $\cB(u,u) = P(u \cdot \nabla u)$ is the bilinear form 
found by projecting the nonlinear term $u \cdot \nabla u$ into 
$H$ and finally, with abuse of notation, $f$ is the original 
forcing, projected into $H$. 
We note that $A$ is diagonalized in the basis comprised
of the $\{\psi_k\}_{k \in {\Z}^2\backslash\{0\}}$, 
on $H$, and the smallest eigenvalue of $A$
is $\lambda_1=4\pi^2/L^2.$
The following proposition
is a classical result which implies the existence
of a dissipative semigroup for the ODE \eref{eq:nse}. 
See Theorems 9.5 and 12.5 in \cite{book:Robinson2001}
for a concise overview and \cite{temam1995navier,book:Temam1997}
for further details.

\begin{proposition} 
\label{prop:1}
Assume that $u_0 \in H^1$ and $f \in H$.
Then \eref{eq:nse} has a unique strong solution on $t
\in [0,T]$ for any $T>0:$
$$u \in L^{\infty}\bigl((0,T);H^1\bigr)\cap L^{2}\bigl((0,T);D(A)\bigr),\quad \frac{du}{dt} \in L^{2}\bigl((0,T);H\bigr).$$
Furthermore the equation has a global attractor $\cA$
and there is $K>0$ such that, 
if $u_0 \in \cA$, then $\sup_{t \ge 0}\|u(t)\|^2 \le K.$
\end{proposition}

We let $\{\Psi(\cdot,t): H^1 \to H^1\}_{t \ge 0}$ 
denote the semigroup of 
solution operators for the equation \eref{eq:nse} through $t$
time units. We note that by
working with weak solutions,
$\Psi(\cdot,t)$ can be extended to act on larger spaces $H^s$,
with $s \in [0,1)$, under the same assumption on $f$;
see Theorem 9.4 in \cite{book:Robinson2001}.  
We will, on occasion, use this extension of $\Psi(\cdot,t)$
to act on larger spaces. In particular we use
it in the following proposition which follows from
Lemmas 5.3 and 5.5 in \cite{article:Cotter2009}.
This proposition is key to enabling us to show that the 
Bayesian inverse problem, which underlies the
filtering problem of interest, is well-posed.

\begin{proposition} 
\label{prop:2}
Let $s \in (0,1]$ and $t_0>0.$ Then,
for all $t>t_0$, and $u,v \in H$,
\begin{eqnarray*}
\begin{array}{ll}
\|\Psi(u,t)\|_{s}^2 &\le  t_0^{-s}C\bigl(|f|^2+|u|^2)\\
\|\Psi(u,t)-\Psi(v,t)\|_{s}^2 &\le  t_0^{-s}C\bigl(|u|,|u-v|,|f|\bigr)
|u-v|^2.
\end{array}
\end{eqnarray*}
\end{proposition}

The key properties of the Navier-Stokes
equation that drive our analysis of the filters are summarized
in the following proposition, taken from the paper
\cite{hayden2011discrete}.
To this end, define $\Psi(\cdot)=\Psi(\cdot,h)$
for some fixed $h>0.$ 
Note that the statement here is closely related
to the {\em squeezing property} \cite{book:Robinson2001}
of the Navier-Stokes equation, a property 
employed in a wide range of applied contexts.

\begin{proposition} 
\label{prop:3}
Let $u \in \cA$ and $v \in H^1$. There is
$\beta=\beta(|f|,L,\nu)>0$ such that
\begin{equation}
\label{eq:1}
\|\Psi(u)-\Psi(v)\|^2 \le \exp(\beta h)\|u-v\|^2.
\end{equation}
Now let $\|u-v\| \le R$ and assume that
$$\lambda>\ls:=\frac{9c^{8/3}}{\lambda_1^{\frac13}}\Bigl(
\frac{2K^\frac12+R^{\frac12}}{\nu}\Bigr)^{8/3}$$
where $c$ is a dimensionless positive constant.
Then there exists $t^*=t^*(|f|,L,\nu,\lambda,R)$
with the property that, for all $h \in (0,t^*],$ 
there is $\gamma \in (0,1)$ such that
\begin{equation}
\label{eq:2}
\|Q_{\lambda}\bigl(\Psi(u)-\Psi(v)\bigr)\|^2 \le \gamma^2\|u-v\|^2.
\end{equation}
\end{proposition}

\begin{proof}
The first statement is simply Theorem 3.8 from 
\cite{hayden2011discrete}. The second statement
follows from the proof of Theorem 3.9 in the same paper, modified at the
end to reflect the fact that, in our setting, 
$\Pl \delta(0) \ne 0.$
Note also that the constant $\lambda$ appearing on the right
hand side of the lower bound for $\lambda$ in the 
statement of Theorem 3.9 in \cite{hayden2011discrete}
should be $\lambda_1$ (as the
proof that follows in \cite{hayden2011discrete} shows) 
and that use of definition of $K$ 
(see Theorem 3.6 of that paper) 
allows rewrite in terms of $K$ -- indeed
the proof in that paper is expressed in terms of $K$.
\end{proof}

\section{Inverse Problem: Filtering}
\label{sec:filters}

In this section we describe the basic problem of {\em filtering}
the Navier Stokes equation \eref{eq:nse}: estimating
properties of the state of the system sequentially
from partial, noisy, sequential observations of the state. 
We first set up the inverse
problem of interest, in subsection \ref{ssec:setup}.
Then, in subsection \ref{ssec:fild}, we describe the 
full statistical filtering distribution. We 
describe the approximate Gaussian filters, which are
the focus of the remainder of the paper, in subsection
\ref{ssec:approx}; in subsection \ref{ssec:fully}
we make a brief remark concerning the extension to
completely observed systems. 
We conclude with subsection \ref{ssec:3DVar}
which introduces 3DVAR as an example of the approximate
Gaussian filter.

\subsection{Setup}
\label{ssec:setup}

Throughout the following we write $\N$ for the natural 
numbers $\{ 1, 2, 3, \dots \}$, and $\Z^{+} := \N \cup \{ 0 \}$ for the non-negative integers $\{ 0, 1, 2, 3, \dots \}$.
Recall that we have defined $\Psi(\cdot)=\Psi(\cdot,h)$
for some fixed $h>0.$ 
Our interest is in determining $u$ from noisy
observations of $\Pl u$.
We let $X$ denote $H^s$ for any $s \geq 1$ and
define $\{ u_j \}_{j \in \Z^{+}}$, $u_j \in X$ by 
\footnote{With abuse of notation, subscripts $j$ will indicate
times, while subscripts $k$ will denote Fourier coefficients as before
in order to avoid confusion. The meaning should also be clear in context.}
\begin{equation}
\label{eqn:ModelWithoutNoise}
u_{j+1} := \Psi(u_j).
\end{equation}
Thus $u_j=u(jh)$ where $u$ solves \eref{eq:nse}.
We now let $\{ {\xi}_j \}_{j \in \N}$ be a noise sequence in $\Wl$ 
which perturbs the sequence $\{\Pl u_j \}_{j \in \N}$
to generate the observation sequence $\{ y_j \}_{j \in \N}$ 
in $\Wl$ given by
\begin{equation}
\label{eqn:Observation}
y_{j+1} := \Pl u_{j+1} + {\xi}_{j+1}, \quad j \in \Z^{+}.
\end{equation}
We let $Y_j=\{y_i\}_{i=1}^j$, the accumulated data 
up to time $t=jh.$
We assume that $u_0$ is not known exactly. The {\em goal}
of filtering is to determine information about the state
$u_j$ from the data $Y_j.$ Mathematically, it is natural 
to formulate this as a Bayesian inverse problem, 
and this viewpoint is described in subsection \ref{ssec:fild}. 

\subsection{Filtering Distribution}
\label{ssec:fild}

In the statistical formulation of the filtering problem
we assume that $(u_0,Y_j)$ is a random variable on $(X,\Wl^j)$,
defined by specifying the distributions of $u_0$ (the prior)
and $=\{\xi_i\}_{1 \le i \le j}$ (the observational noises) which
is assumed to be an i.i.d. sequence. The
aim is to find the conditional probability distribution on $u_0$
given a single realization of $Y_j$, and we denote
this conditional distribution by $\bbP(u_0|Y_j).$ 
The filtering distribution, namely the probability distribution 
on $u_j$ given $Y_j$, denoted by $\bbP_{j}(u_j|Y_j)$, 
is then found 
as the image of $\bbP(u_0|Y_j)$
under the map $\Psi(\cdot;jh).$
In finite dimensions carrying out this 
program is a straightforward application of Bayes' theorem.
Here we show how similar ideas may be applied in the
infinite dimensional setting. We make the following assumption.

\begin{assumption}
\label{a:z} 
The prior distribution on $u_0$ is
a Gaussian $\bbP_0(u_0)=N({\widehat m}_0,{\widehat \cC}_0)$, with
the property that $\bbP_0(H)=1.$
The observational noise sequence 
$\{\xi_j\}_{j \in \N}$ is
an i.i.d sequence in $\Wl$, independent of $u_0$, 
with $\xi_1$ distributed according to a 
Gaussian measure $N(0,\Gamma)$ on $\Wl$, with 
$\Gamma$ strictly positive on $\Wl$.
\end{assumption}

Under this assumption the Bayesian inverse problem has
a well-defined solution, and exhibits well-posedness
with respect to the data, in the Hellinger metric
(see \cite{article:Cotter2009} for a definition).

\begin{theorem}
Let Assumption \ref{a:z} hold.
The measure $\bbP(u_0|Y_j)$ is 
absolutely continuous with respect to $\bbP_0(u_0)$
with Radon-Nikodym derivative given by
\begin{equation}
\label{eq:rnd}
\frac{d\bbP}{d\bbP_0}(u_0|Y_j) \propto 
\exp\Bigl(-\Phi^{(j)}(u_0)\Bigr).
\end{equation}
Here
\begin{equation}
\Phi^{(j)}(w)=\frac12\sum_{i=1}^j \Bigl|\Gamma^{-\frac12}
\Bigl(y_i-\Pl\Psi(w;ih)\Bigr)\Bigr|^2. 
\label{eq:phin}
\end{equation}
Furthermore, $\bbP(u_0|Y_j)$ is Lipschitz in $Y_j$
with respect to the Hellinger metric.
\label{thm:z}
\end{theorem}

\begin{proof} Note that $\Wl$ is a finite
dimensional space. The result can thus
be deduced from Corollary 2.2 and
Theorem 2.5 in \cite{article:Cotter2009}. To apply
Corollary 2.2 it suffices to show that $\Psi(\cdot;ih)\colon H
\to H$ is continuous for any $i \in {\mathbb Z}$
since then $\Pl \Psi(\cdot;ih) \colon H \to \Wl$ is continuous; 
the required
continuity follows from the second item of Proposition
\ref{prop:2}. To apply Theorem 2.5 it suffices to
show that $\Psi(\cdot;ih)\colon H \to H$ is polynomially
bounded, since then $\Phi^{(j)}\colon H \to \mathbb{R}^+$,
and its Lipschitz constant with respect to data $Y_j$, are
both polynomially bounded. This polynomial boundedness of
$\Psi$ follows from the first item of Proposition \ref{prop:2}. 
\end{proof}

The measure $\bbP_j(u_j|Y_j)$ is then defined by
the push-forward under the semigroup $\Psi(\cdot;jh)$ and 
Proposition \ref{prop:2} gives the following corollary:

\begin{corollary}
Let Assumption \ref{a:z} hold. Then
the sequence of measures $\{\bbP_j(u_j|Y_j)\}_{j \ge 0}$
is well-defined by
$$\bbP_j(\cdot|Y_j)=\Psi(\cdot;jh) \star \bbP(\cdot |Y_j)$$
where $\bbP(\cdot|Y_j)$ is given in Theorem \ref{thm:z}.
Furthermore, for $v \sim \bbP_j(\cdot|Y_j)$, $v \in H^1$ 
with probability one. 
\label{cor:z}
\end{corollary}

\subsection{Approximate Gaussian Filters; Partial Observations}
\label{ssec:approx}

The measures $\bbP(u_0|Y_j)$
and $\bbP_j(u_j|Y_j)$ determined in Theorem \ref{thm:z}
and Corollary \ref{cor:z} are, in practice,
very hard to compute by statistical sampling methods.
Sequential Monte Carlo methods (SMC) can in principle
be applied to determine the $\bbP_j(u_j|Y_j)$, but
new ideas are required to extend them to problems
in high dimensions \cite{Bickel}.
The measure $\bbP(u_0|Y_j)$ has the advantage of being
specified via density with respect to a Gaussian
and it is shown in \cite{lawstuart} that,
building on this fact,
Markov chain-Monte Carlo (MCMC) methods can be used successfully
to approximate $\bbP(u_0|Y_j)$
when the dynamics of \eref{eq:nse} are mildly turbulent;
but these methods are very expensive indeed, and do not exploit
the sequential nature of the problem as $j$ is incremented.
Sequential methods are attractive for online applications and
for this reason various {\em ad hoc} approximation methods
are used which lead to tractable sequential algorithms. 
A commonly made approximation is to impose 
a Gaussian structure on the filtering distributions so that
\begin{equation}
\bbP_j(u_j|Y_j) \approx N(\widehat{m}_{j},\widehat{\cC}_j).
\label{eq:appg}
\end{equation}
The key question in designing an approximate Gaussian filter, 
then, is to find an update rule of the form
\begin{equation}
\label{eq:keym}
(\widehat{m}_{j},\widehat{\cC}_j) \mapsto (\widehat{m}_{j+1},\widehat{\cC}_{j+1}).
\end{equation}
Because of the linear form of the observations 
in \eref{eqn:Observation}, together with the fact that 
the noise is mean zero-Gaussian, this update rule
is determined directly if we use the Gaussian
{\em assumption} 
that the distribution of $u_{j+1}$ given $Y_j$ is Gaussian:
\begin{equation}
u_{j+1}|Y_j \sim N(m_{j+1},\cC_{j+1}).
\label{eq:appg2}
\end{equation}
In general, even if the approximation \eref{eq:appg} is
a good one, there is no reason to expect \eref{eq:appg2}
to be a good approximation since the distribution
on $u_{j+1}|Y_j$ is
the pushforward of $\bbP_j(u_{j}|Y_j)$, assumed Gaussian,
under the {\em nonlinear} map $\Psi(\cdot;h)$ and in general
only linear transformations preserve Gaussianity.  
However, in this paper, we will simply {\em impose} the
approximation \eref{eq:appg2} with 
\begin{equation}
m_{j+1}=\Psi(\widehat{m}_{j};h)
\label{eq:mapm}
\end{equation} 
and with $\cC_{j+1}$ specified exogenously.
This then defines the map \eref{eq:keym},
as we now show.

\begin{assumption} Assume that $u_{j+1}|Y_j$ is
specified by \eref{eq:appg2} for
$m_{j+1}$ given by \eref{eq:mapm} and covariance
operator $\cC_{j+1}$ on $H$ which is strictly
positive on $\Wl$ and which commutes
with $A$.\footnote{Note that commuting with $A$ is equivalent
to being diagonalizable in the same basis as $A$, namely
the $\{\psi_k\}_{k \in {\mathcal Z}^2\backslash\{0\}}$.} 
Assume further that $y_{j+1}|u_{j+1}$
is given by \eref{eqn:Observation} where
the random variable $\xi_{j+1}$ is
a mean zero Gaussian in $\Wl$
with covariance $\Gamma$ a
strictly positive operator on $\Wl$ which
commutes with $A$.
\label{a:zz}
\end{assumption}

Note that the Gaussian $N(m_{j+1},\cC_{j+1})$
factors as the product of two independent Gaussians
on $\Wl$ and $\Wc$, because $\cC_{j+1}$ 
commutes with $A$; to avoid proliferation of
notation, we also denote by $\cC_{j+1}$ the
covariance operator restricted to $\Wl$ and $\Wc$.
The factoring as independent
products is inherited by the resulting Gaussian approximation
for $u_{j+1}|Y_{j+1}$ as the following charactertization shows.

\begin{theorem} Let Assumption \ref{a:zz} hold.
Then $u_{j+1}|Y_{j+1}$ is 
Gaussian on $H$ and factors as the product of
two indepenent Gaussians on $\Wl$ and $\Wc$. 
Denoting the mean and covariance of
both of these independent Gaussians by $\hu_{j+1}$ 
and $\hC_{j+1}$ respectively we have
 

\begin{eqnarray*} \begin{array}{ccccc}
\hC_{j+1}&= \cC_{j+1}, \quad \quad ~~ 
&\quad  \hu_{j+1}&= \Ql \Psi(\hu_{j}),  
~~~ \quad \quad \quad \quad \quad
&\quad {\rm {on}} \,\,\Wc\\
{\widehat {\cC}}_{j+1}^{-1}&=\cC_{j+1}^{-1}+\Gamma^{-1}, &\quad 
{\widehat {\cC}}_{j+1}^{-1}\widehat{m}_{j+1} &= \cC_{j+1}^{-1} \Pl\Psi(\widehat{m}_{j}) + \Gamma^{-1}y_{j+1},
&\quad {\rm {on}} \,\,\Wl.
\end{array}
\end{eqnarray*}
\label{t:app}
\end{theorem}

\begin{proof} For economy of notation,
let $v$ denote the random variable $u_{j+1}|Y_j$
and $y$ the random variable $y_{j+1}$.
Under the stated assumptions, $(v,y)$ is a jointly 
Gaussian random variable and we are interested 
in finding the conditional distribution of 
$v|y$ (which is the desired 
distribution of $u_{j+1}|Y_{j+1}.$) Let $\nu(dv,dy)$ denote
the Gaussian distribution of $(v,y)$ and let
$\nu_0(dv,dy)$ denote the Gaussian distribution specified
as the independent product of the Gaussian
measures $\mu_0=N(m_{j+1},\cC_{j+1})$ and
$N(0,\Gamma).$ From the properties of finite
dimensional Gaussians it is immediate that $\nu$ has
density with respect to $\nu_0$ and that
$$\frac{d\nu}{d\nu_0}(v,y)=\exp\Bigl(-\frac12\bigl|
\Gamma^{-\frac12}(y-\Pl v)\bigr|^2+\frac12\bigl|\Gamma^{-\frac12}
y\bigr|^2\Bigr).$$
If $\mu$ is the desired conditional distribution on $v|y$,
then Lemma 2.3 in \cite{article:Cotter2009} gives
$$\frac{d\mu}{d\mu_0}(v) \propto \exp\Bigl(-\frac12\bigl|
\Gamma^{-\frac12}(y-\Pl v)\bigr|^2\Bigr)$$
with constant of proportionality depending only on $y$.
Note that $\mu_0$ factors in the desired fashion on $\Wl$
and $\Wc$, and that the change of measure depends 
only on $\Pl v$; hence $\mu$ factors in the same fashion.
Furthermore, since the change of meausure depends only on (the
finite dimensional) $\Pl v \in \Wl$, completing the square
gives the expression for the measure in $\Wl$, and
in $\Wc$ we have $\mu=\mu_0$. This
completes the proof.
\end{proof}

It is demonstrated numerically in \cite{lawstuart} that 
Gaussian approximations of the filtering distribution
such as the one characterized in the
preceding theorem are, in general, not good approximations.
More precisely, they fail to accurately capture
covariance information. However, the same numerical experiments
reveal that the methodology can perform well in
replicating the mean, if parameters are chosen correctly,
even if it is initially in error. Indeed this accurate
tracking of the mean is often achieved by means of
{\em variance inflation} -- increasing the model uncertainty,
here captured in the exogenously imposed $\cC_j$, 
in comparison with the 
data uncertainty, here captured in the $\Gamma$.
The purpose of the remainder of the paper is to
explain, and illustrate, this phenomenon by means of
analysis and numerical experiments. To this 
end we introduce a compact notation for  the mean update $\hu_{j}
\mapsto \hu_{j+1}.$

We define an operator $B_j \colon \Wl \times \Wc \mapsto
\Wl \times \Wc$ by
\begin{eqnarray}
B_j=\left(
\begin{array}{cc}
{\widehat \cC}_{j+1}\cC_{j+1}^{-1} & 0\\
0 & {\widehat \cC}_{j+1}\cC_{j+1}^{-1}
\end{array}
\right)
\label{eq:oldB}
\end{eqnarray}
where ${\widehat \cC}_{j+1}$ and $\cC_{j+1}$
denote the covariances in $\Wl$ (resp. $\Wc$)
in the top left (resp. bottom right) entries of $B_j$. 
We also
extend $y_{j+1}$ from an element of $\Wl$ to an
element of $H$ in the
canonical fashion, by defining it to be zero in $\Wc$. Then
Theorem \ref{t:app} yields the key equation
\begin{equation}
\label{eq:meanu}
\widehat{m}_{j+1}=B_j\Psi(\widehat{m}_j;h)+(I-B_j)y_{j+1},
\end{equation}
which demonstrates that the estimate of
the mean at time $j+1$ is found as an operator-convex
combination of the true dynamics applied to the estimate
of the mean at time $j$, and the data at time $j+1$.
Note that, in the case of a linear dynamical system where
$\Psi(\cdot;h)$ is a linear map, the matrix $I-B_j$ is
the {\em Kalman gain} matrix \cite{harvey1991forecasting}.

\subsection{Approximate Gaussian Filters; Complete Observations}
\label{ssec:fully}

We will also study the situation where
complete observations are made, obtained by taking $\lambda
\to \infty$ in the preceding analyses.
The observations are given by
\begin{equation}
\label{eqn:Observation2}
y_{j} := u_{j} + {\xi}_{j}, \quad j \in \Z^{+}
\end{equation}
where now $y_j, \xi_j \in H$ and the i.i.d. mean
zero Gaussian sequence
$\{\xi_j\}_{j \in \N}$ 
is defined by a covariance operator $\Gamma$ on $H$.

It is possible to characterize the full filtering
distribution in this setting, but rather technical
to do so; the example of the inverse problem for the heat
equation in \cite{article:Stuart2010} illustrates the
technicalities involved.
Because of this we concentrate on studying only the
equation for the mean update in the approximate Gaussian
filter. This takes the form \eref{eq:meanu} where, on the
whole of $H$, we have
\begin{equation}
\label{eq:meanf}
B_j={\widehat \cC}_{j+1}\cC_{j+1}^{-1}, \quad
{\widehat {\cC}}_{j+1}^{-1}=\cC_{j+1}^{-1}+\Gamma^{-1}. 
\end{equation}

\subsection{Example of an Approximate Gaussian Filter: 3DVAR}
\label{ssec:3DVar}

The algorithm described in the previous section yields
the well-known 3DVAR method, discussed in the introduction,
when $\cC_j \equiv \cC$ for some fixed operator $\cC$.
To impose commutativity with $A$, we assume that 
the operators $\Gamma$ and $\cC$
are both fractional powers of the Stokes operator $A$, in
$\Wl$ and $H$ respectively.
Note that fixing $\cC_j \equiv \cC$ implies that
${\widehat \cC}_j \equiv {\widehat \cC}$ where
\begin{eqnarray*}
\begin{array}{cc}
{\widehat \cC}^{-1}&=\cC^{-1} + \Gamma^{-1},\quad {\rm in}\,\,\Wl\\
{\widehat \cC}^{-1}&=\cC^{-1},\quad \quad \quad ~~ {\rm in}\,\,\Wc.
\end{array} \end{eqnarray*}

We choose $A_0=\ell A$\footnote{The parameter $\ell$ forms
a useful normalizing constant in the numerical experiments
of section \ref{sec:numerics}.} and
set $\cC =\delta^2 A_0^{-2\zeta}$ in $H$ and 
$\Gamma = \sigma^2 A_0^{-2\beta}$ in $\Wl.$ Substituting 
into the update formula \eref{eq:meanu}
for $\widehat{m}_j$ and defining 
$\eta = \sigma / \delta,$ $\alpha = \zeta - \beta$,
$\Bl=(I+\eta^2 A_0^{2\alpha})^{-1}\eta^2 A_0^{2\alpha}$ in $\Wl$ 
then \eref{eq:oldB} gives a constant matrix 
\begin{eqnarray}
B=\left(
\begin{array}{cc}
\Bl & 0\\
0 & I 
\end{array}
\right).
\label{eq:B}
\end{eqnarray}
Using this we obtain the mean update formula
\begin{equation}
\hu_{j+1} = B\Psi(\hu_j) + ( I- B) y_{j+1}.
\label{eq:up2}
\end{equation}
Notice that for $\cC, \Gamma$ given as above, 
the algorithm depends only on the three parameters
$\lambda, \alpha$ and $\eta$, once the constant
of proportionality $\ell$ in $A_0$ is set. The parameter $\lambda$
measures the size of the space in which observations
are made; for fixed wavevector $k$, the parameter $\eta$ 
is a measure of the
scale of the uncertainty in observations to uncertainty
in the model; and the sign of the parameter $\alpha$
determines whether, for fixed $\eta$ and asymptotically 
for large wavevectors,
the model is trusted more ($\alpha>0$) or less ($\alpha<0$)
than the data.

In the case $\lambda=\infty$, the case of complete
observations where the whole velocity field 
is noisily observed, we again obtain \eref{eq:up2}, with 
$B=\Bl=(I+\eta^2 A_0^{2\alpha})^{-1}\eta^2 A_0^{2\alpha}$ in $H$.The roles of $\eta$ and $\alpha$ are the same as in the
finite $\lambda$ (partial observations) case.

The discussion concerning parametric
dependence with respect to varying $\eta$ shows that,
for the example of 3DVAR introduced here, and
for both $\lambda$ finite and infinite, 
variance inflation can be achieved by decreasing
the parameter $\eta.$ We will show that variance inflation 
does indeed improve the ability of the filter to track 
the signal.

\section{Stability}
\label{sec:stability}

In this section we develop conditions under which 
it is possible to prove
stability of the nonautonomous dynamical
system defined by the mean update equation \eref{eq:meanu}. 
By stability we here mean that, when the noise
perturbing the observations is ${\mathcal O}(\epsilon)$, 
the mean update
will converge to an ${\mathcal O}(\epsilon)$ neighbourhood 
of the true signal, even if initially it is an ${\mathcal O}(1)$
distance from the true signal.
In subsection \ref{ssec:m1} we study the case of partial
observations; subsection \ref{ssec:m2} contains the
(easier) result for the case of complete observations.
The third subsection \ref{ssec:s2} shows how
our results can be applied to the specific
instance of the 3DVAR algorithm introduced
in subsection \ref{ssec:3DVar}, for any $\alpha \in {\mathbb R}$,
provided $\eta$, which is a measure of uncertainty
in the data to uncertainty in the model, is sufficiently small:
this, then, is a result concerning variance inflation.
In the final subsection \ref{ssec:spde} 
we also consider 3DVAR in the case of
frequent observations and large $\eta$, deriving an
SPDE and a PDE, both of
which may be used to study filter stability.

For simplicity, we will assume a ``truth'' which is on
the global attractor, as in \cite{hayden2011discrete}. This
is not necessary, but streamlines the presentation
as it gives an automatic uniform in time bound in $H^1$. 
Recall that $\|\cdot\|$ denotes the norm in $H^1$,
and $|\cdot|$ the norm in $H$; similarly we lift
$\|\cdot\|$ to denote the induced operator norm
$H^1 \to H^1$.

\subsection{Main Result: Partial Observations}
\label{ssec:m1}

In this case we will see that it is crucial that the 
observation space $\Wl$ is sufficiently large, i.e. 
that a sufficiently large number of modes are observed. 
This, combined with the contractivity in the high modes 
encapsulated in Proposition \ref{prop:3} 
from \cite{hayden2011discrete},
can be used to ensure stability if combined with variance
inflation.  

We study filters of the form given in \eref{eq:meanu}
and make the following assumption on the observations $\{y_j\}.$ 
We note that this assumption is incompatible with the Gaussian
assumptions used to derive the filter and we return to 
this point when describing numerical results 
in section \ref{sec:numerics}.

\begin{assumption}
\label{a:1}
Consider a sequence $u_j=u(jh)$ where $u(t)$ is a solution
of \eref{eq:nse} lying on the global attractor $\cA$. 
Then, for some $\lambda \in (\lambda_1,\infty)$,
$$y_j=\Pl u_j+\xi_j$$
for some sequence $\xi_j$ satisfying $\sup_{j \ge 1}\|\xi_j\|
\le \epsilon.$
\end{assumption}

We make the following assumption about the family $\{\Be\}$,
and assumed dependence on a parameter $\eta \in {\mathbb R}^+.$
Recall that the inverse of $\eta$ 
quantifies the amount of variance inflation. 

\begin{assumption}
\label{a:2}
The family of positive operators $\{\Be(\eta)\colon H^1 \to H^1\}_{j \ge 1}$ 
commute with $A$, satisfy $\sup_{j \ge 1}\|\Be(\eta)\| \le 1$, and 
$\sup_{j \ge 1}\|I-\Be(\eta)\| \le b$ for some $b \in {\mathbb R}^+,$
uniformly with respect to $\eta$.
Furthermore, $\bigl(I-B_j(\eta)\bigr)\Ql \equiv 0$ 
and there is, for all $\lambda>\lambda_1$, constant 
$c=c(\lambda)>0$ such that $\sup_{j \ge 1}\|\Pl \Be(\eta)\| \le c\eta^2.$
\end{assumption}

We now study the asymptotic behaviour of the filter
under these assumptions. 

\begin{theorem}
\label{t:m}
Let Assumptions \ref{a:1} and \ref{a:2} hold, choose
any $\hu_0 \in \mathbb{B}_{H^1}\bigl(u(0),r\bigr)$ and
let $(\lambda^*,t^*)$ be as given in Proposition \ref{prop:3}.
Assume that $\lambda>\lambda^*$.  Then for any $h \in (0,t^*]$
there is $\eta$ sufficiently 
small so that the sequence
$\{\hu_j\}_{j \ge 0}$ given by \eref{eq:meanu} satisfies,
for some $a \in (0,1)$,
$$\|\hu_j-u_j\| \le a^j r+2b\epsilon\sum_{i=0}^{j-1}
a^i.$$ 
Hence
$$\limsup_{j \to \infty}\|\hu_j-u_j\| \le \frac{2b}{1-a}\epsilon.$$ 
\end{theorem}

\begin{proof} 
Assumption \ref{a:2} shows that
$y_{j+1}=\Pl\Psi(u_j)+\xi_{j+1}.$ 
Recall that in \eref{eq:meanu} $y_{j+1}$
has been extended to an element of $H$, by defining it to
be zero in $\Wc$, and we do the same with $\xi_{j+1}.$
Substituting the resulting expression for $y_{j+1}$ in
\eref{eq:meanu} we obtain
$$\hu_{j+1}=\Be \Psi(\hu_j)+(I-\Be)\Pl\Psi(u_j)+(I-\Be)\xi_{j+1}
$$
but since $(I-\Be)\Ql \equiv 0$ by assumption we have
\begin{equation}
\label{eq:need}
\hu_{j+1}=\Be \Psi(\hu_j)+(I-\Be)\Psi(u_j)+(I-\Be)\xi_{j+1}.
\end{equation}
Note also that
$$u_{j+1}=\Be \Psi(u_j)+(I-\Be)\Psi(u_j).$$
Subtracting gives the basic equation for error propagation,
namely
\begin{equation}
\hu_{j+1}-u_{j+1}=\Be\bigl(\Psi(\hu_j)-\Psi(u_j)\bigr)
+(I-\Be)\xi_{j+1}.
\label{eq:error}
\end{equation}

Since $\lambda>\lambda^{\star}$ the second
item in Proposition \ref{prop:3} holds.
Fix $a \in (\gamma,1)$ where $\gamma$ is defined in
Proposition \ref{prop:3}.  Assume,
for the purposes of induction, that 
$$\|\hu_j-u_j\| \le a^j r+2b\epsilon\sum_{i=0}^{j-1}
a^i.$$ 
Define $R=2r$ noting that
the inductive hypothesis implies that, for $\epsilon$
sufficiently small, 
$\|\hu_j-u_j\| \le r+2b(1-a)^{-1}\epsilon \le R.$
Applying $\Pl$ to \eref{eq:error} and using \eref{eq:1} gives
\begin{eqnarray*}
\begin{array}{cc}
\|\Pl(\hu_{j+1}-u_{j+1})\| &\le
\|\Pl\Be\|\|\bigl(\Psi(\hu_j)-\Psi(u_j)\bigr)\|
+\|\Pl(I-\Be)\|\epsilon\\
&\le c(\lambda) \eta^2 \exp(\beta h/2)\|\hu_j-u_j\|+b\epsilon.
\end{array}
\end{eqnarray*}
Applying $\Ql$ to \eref{eq:error} and using \eref{eq:2}
gives\footnote{The term $b\epsilon$ on the right-hand side of
the final identity can here be set to zero because 
$(I-\Be)\Ql \equiv 0$; however in the
analogous proof of Theorem \ref{t:mz} it is present and
so we retain it for that reason.}
\begin{eqnarray*}
\begin{array}{cc}
\|\Ql(\hu_{j+1}-u_{j+1})\| &\le
\|\Be\|\|\Ql\bigl(\Psi(\hu_j)-\Psi(u_j)\bigr)\|
+\|\Ql(I-\Be)\|\epsilon\\
&\le \gamma\|\hu_j-u_j\|+b\epsilon.
\end{array}
\end{eqnarray*}
Now note that, for any $w \in H^1$,
$\|w\| =\bigl(\|\Pl w\|^2+\|\Ql w\|^2\bigr)^{\frac12}
\le \|\Pl w\|+\|\Ql w\|.$
Thus, by adding the two previous inequalities, we find
that
$$\|\hu_{j+1}-u_{j+1}\| \le 
\bigl(c(\lambda) \eta^2 \exp(\beta h/2)+\gamma\bigr)
\|\hu_j-u_j\|+2b\epsilon.$$
Since $\gamma\in (0,1)$ and $a \in (\gamma,1)$, we may
choose $\eta$ sufficiently small so that 
$$\|\hu_{j+1}-u_{j+1}\| \le a\|\hu_j-u_j\|+2b\epsilon.$$
and the inductive hypothesis holds with $j \mapsto j+1$.
Taking $j \to \infty$ gives the desired result concerning
the limsup.
\end{proof}

\begin{remark}
\label{rem:n}

Note that the proof exploits the fact that
$B_j\Psi(\cdot)$ induces a contraction within 
a finite ball in $H^1$. This contraction
is established by means of the contractivity of $B_j$
in $\Wl$, via variance inflation,
and the squeezing property of $\Psi(\cdot)$ in $\Wc$,
for large enough observation space,
from Proposition \ref{prop:3}. 

There are two important conclusions from this theorem. The
first is that, even though the solution is only
observed in the low modes, there is sufficient
contraction in the high modes to obtain an error
in the entire estimated state which is of the same order
of magnitude as the error in the (low mode only) observations.
The second is that this phenomenon occurs even when
the initial estimate suffers from an ${\mathcal O}(1)$ error.

\label{rem:1}
\end{remark}

\subsection{Main Result: Complete Observations}
\label{ssec:m2}

Here we study filters of the form given in \eref{eq:meanu}
with observations given by \eref{eqn:Observation2}.
In this situation  the whole velocity field is
observed and so, intuitively, it should be
no harder to obtain stability than in the partially
observed case. The proof is in fact almost identical
to the case of partial observations, and so we omit
the details. We observe that, although there is no parameter
$\lambda$ in the problem statement itself, 
it is introduced in the proof: as in the previous
subsection, see Remark \ref{rem:n},
 the key to stability is to obtain
contraction in $\Wc$ using the squeezing property of
the Navier-Stokes equation, and contraction in
$\Wl$ using the properties of the filter to control
unstable modes. 

We make the following assumptions: 

\begin{assumption}
\label{a:1z}
Consider a sequence $u_j=u(jh)$ where $u(t)$ is a solution
of \eref{eq:nse} lying on the global attractor $\cA$. 
Then
$$y_j=u_j+\xi_j$$
for some sequence $\xi_j$ satisfying $\sup_{j \ge 1}\|\xi_j\|
\le \epsilon.$
\end{assumption}

\begin{assumption}
\label{a:2z}
The family of positive operators $\{\Be(\eta)\colon H^1 \to H^1\}_{j \ge 1}$ 
commute with $A$, satisfy $\sup_{j \ge 1}\|\Be(\eta)\| \le 1$, and 
$\sup_{j \ge 1}\|I-\Be(\eta)\| \le b$ for some 
$b \in {\mathbb R}^+,$ uniformly with respect to $\eta.$
Furthermore, for all $\lambda>\lambda_1$,  there is a constant 
$c=c(\lambda)>0$ such that 
$\sup_{j \ge 1}\|\Pl \Be(\eta)\| \le c\eta^2.$
\end{assumption}

We now study the asymptotic behaviour of the filter
under these assumptions. 

\begin{theorem}
\label{t:mz}
Let Assumptions \ref{a:1z} and \ref{a:2z} hold and choose
any $\hu_0 \in \mathbb{B}_{H^1}\bigl(u(0),r\bigr).$ 
Then for any $h \in (0,t^*]$, with $t^*$ given in 
Proposition \ref{prop:3}, there is $\eta$ sufficiently 
small so that the sequence
$\{\hu_j\}_{j \ge 0}$ given by \eref{eq:meanu} satisfies, for
some $a \in (0,1)$,
$$\|\hu_j-u_j\| \le a^j r+2b\epsilon\sum_{i=0}^{j-1}
a^i.$$ 
Hence
$$\limsup_{j \to \infty}\|\hu_j-u_j\| \le \frac{2b}{1-a}\epsilon.$$ 
\end{theorem}

\begin{proof} The proof is nearly identical to that of
Theorem \ref{t:m}. Differences arise only because we have
not assumed that $(I-\Be)\Ql \equiv 0.$ This fact
arises in two places in Theorem \ref{t:m}.
The first is where we obtain \eref{eq:need}.
However in this case we directly obtain \eref{eq:need} 
since the whole velocity field is observed. The second
place it arises is already dealt with in the footnote
appearing in the proof of Theorem \ref{t:m}
when estimating the contraction
properties in $\Wc$; there we indicate that the proof is
already adjusted to allow for the situation required here. 
\end{proof}

\begin{remark}
\label{r:2}

If $\sup_{j \ge 1}\|B_j(\eta)\|<c\eta^2$ 
then the proof may be simplified considerably
as it is not necessary to split the space into two parts,
$\Wl$ and $\Wc$. Instead the contraction of $B_j$ can
be used to control any expansion in $\Psi(\cdot)$, provided
$\eta$ is sufficiently small.

We observe that the key conclusion of the theorem is
the stabilization of the algorithm  when started at
distances of ${\mathcal O}(1)$; the asymptotic bound,
although of ${\mathcal O}(\epsilon)$, has constant
$\frac{2b}{1-a}$ which may exceed $1$
and so the bound may appear worse than that obtained
by simply using the
observations to estimate the signal. In practice,
however, we will show that the algorithm gives estimates
of the state which improve upon the observations.

\end{remark}

\subsection{Example of Main Result: 3DVAR}
\label{ssec:s2}

We demonstrate that the 3DVAR algorithm from 
subsection \ref{ssec:3DVar} satisfies Assumptions
\ref{a:2} and \ref{a:2z} in the partially and
completely observed cases respectively, and hence
that the resulting filters will locate the true signal,
provided $\eta$ is sufficiently small.  In the next subsection we also
consider the limit of frequent observations and $\eta$
large, in which case (S)PDEs can be derived to study filter
stability.

Satisfaction of Assumptions \ref{a:2} and \ref{a:2z}
follows from the properties of 
$$\Bl=(I+\eta^2A_0^{2\alpha})^{-1}\eta^2 A_0^{2\alpha}, \quad
I-\Bl=(I+\eta^2A_0^{2\alpha})^{-1}.$$
Note that the eigenvalues of $\Bl$ are
$$\frac{\eta^2\bigl(4\ell\pi^2|k|^2\bigr)^{2\alpha}}{1+\eta^2\bigl(4\ell\pi^2|k|^2\bigr)^{2\alpha}},$$
if $A_0=\ell A.$
Clearly the spectral radius of $\Bl$ is less than or equal
to one on $\Wl$ or $H$, independently of the sign of
$\alpha.$  The difference is just that $|k|^2<\lambda/\lambda_1$
in the former, and $|k|$ is unbounded in the latter.

First we consider the partially observed situation.
We note that $B_j \equiv B$ and is given by \eref{eq:B}: 
\begin{eqnarray*}
B=\left(
\begin{array}{cc}
(I+\eta^2A_0^{2\alpha})^{-1}\eta^2 A_0^{2\alpha} & 0\\
0 & I 
\end{array}
\right)
\end{eqnarray*}
so that the Kalman gain-like matrix $I-B$ is given by
\begin{eqnarray*}
I-B=\left(
\begin{array}{cc}
(I+\eta^2A_0^{2\alpha})^{-1} & 0\\
0 & 0 
\end{array}
\right).
\end{eqnarray*}
From this it is clear that
$(I-B)\Ql \equiv 0.$ Furthermore, since the spectral
radius of $\Bl$ does not exceed one, the same is true of $B$.
Hence for the operator norms from $H^1$ into itself
we have $\|B\| \le 1.$ Similarly,
if $\alpha<0$ then $b:=\|I-B\|=1$, whilst if $\alpha \ge 0$
then $b=\Bigl(1+\eta^2(\ell \lambda_1)^{2\alpha}\Bigr)^{-1}<1.$ 
Thus Theorem \ref{t:m} applies.

In the fully observed case we simply have $B_j \equiv B$
where $B=B_0(\eta)$ defined above on $H$. 
Again $\|B\| \le 1$ and
if $\alpha<0$ then $\|I-B\|=b=1$, whilst if $\alpha \ge 0$
then $b=\Bigl(1+\eta^2(\ell\lambda_1)^{2\alpha}\Bigr)^{-1}<1.$
Thus Theorem \ref{t:mz} applies. Note (see
Remark \ref{r:2}), that if $\alpha<0$ then
the proof of that theorem could be
simplified considerably because $\|B\|<1$ and in fact
$\sup_{j \ge 1}\|B\|<c\eta^2.$ 

\begin{remark}
\label{r:3} 

Recall that Theorem \ref{t:mz} gives the
asymptotic bound $C\epsilon$ on the state
estimate where $C=\frac{2b}{1-a}.$
When $\alpha < 0$ and $b=1$ we have $C>1;$ 
in this case the asymptotic bound exceeds that obtained
by simply employing the observations.
In the case $\alpha>0$, we have $b<1$ and 
it is possible that $C<1$, meaning that the bound
may be of direct value to the practitioner. However,
regardless of the value of $C$, we emphasize that
the value of Theorem \ref{t:mz} is the stabilization
from ${\mathcal O}(1)$ initial error, and not the value of
the constant $C$ in the asymptotic bound. Our numerics
will show that, in practice, the error in the state estimator
often falls well below the average error committed
by simply using the data as an estimator.  
\end{remark}

\subsection{(S)PDE limit}
\label{ssec:spde}

Theorems \ref{t:m} and \ref{t:mz} 
are valid for sufficiently
{\em small} $\eta$ and hence exploit variance inflation. 
In this limit the observations are given
large weight at low wavevectors and the contraction resulting
from this acts to stabilize the system.
There is an interesting family of (S)PDEs for the mean
which can be derived in the case of {\em large} $\eta$, 
by considering the limit of frequent observations. 
We describe these (S)PDEs in the case of the 3DVAR
method  from subsection \ref{ssec:3DVar}. For
simplicity we consider the fully observed case; partial
observations can be handled similarly. 

To this end we assume that $h \ll 1$, $r \in (0,1]$, 
and that $\sigma^2=\sigma_0^2/h^r$ and $\delta^2=\omega\sigma_0^2 h^{1-r}.$
Thus $\eta^{-2}=\omega h.$
We then obtain
\begin{eqnarray*}
\begin{array}{ll}
\hu_{j+1}&=(I+\eta^2 A_0^{2\alpha})^{-1}\eta^2 A_0^{2\alpha} \Psi(\hu_j)
+(I+\eta^2 A_0^{2\alpha})^{-1}y_{j+1}\\
&=(I+\eta^{-2} A_0^{-2\alpha})^{-1}\Psi(\hu_j)
+(I+\eta^{-2} A_0^{-2\alpha})^{-1}\eta^{-2} A_0^{-2\alpha} y_{j+1}\\
&=(I+\omega h A_0^{-2\alpha})^{-1}\Psi(\hu_j)+(I+\omega h A_0^{-2\alpha})^{-1}
\omega h A_0^{-2\alpha}y_{j+1}.
\end{array}
\end{eqnarray*}
If we define the sequence $\{z_j\}_{j \in \Z^+}$ by 
\begin{equation}
\label{eq:z}
z_{j+1}=z_j+hy_{j+1},\quad z_0=0
\end{equation}
then the preceding expression for $\hu_{j+1}$ can be
rearranged and expanded formally in powers of $h$
to give
$$\Bigl(\frac{\hu_{j+1}-\hu_{j}}{h}\Bigr)=
\Bigl(\frac{\Psi(\hu_j)-\hu_{j}}{h}\Bigr)-\omega  A_0^{-2\alpha}\Psi(\hu_j)+\omega  A_0^{-2\alpha}\Bigl(\frac{z_{j+1}-z_{j}}{h}\Bigr)+{\cal O}(h).$$
Note also that formal expansion in $h$ gives
$$\Psi(u;h)=u-h\bigl(\nu Au +\cB(u,u)-f\bigr)+{\mathcal O}(h^2).$$
Thus substituting in the previous expression and taking
the limit $h \to 0$ we obtain the continuous time filter
\begin{equation}
\frac{\rd \hu}{\rd t} + \nu A\hu + \cB(\hu, \hu) +
\omega A_0^{-2\alpha}\Bigl(\hu-\frac{\rd z}{\rd t}\Bigr)= f,
\quad \hu(0)=\hu_0.
\label{eq:ctf}
\end{equation}
Here the observation $y=\frac{\rd z}{\rd t}$ enters as a forcing
to the underlying Navier-Stokes equation. The term
proportional to $\hu-\frac{\rd z}{\rd t}$ acts to drive the
solution towards the observation, and may compete with
destabilization
induced by the Navier-Stokes
forcing $\nu A\hu + \cB(\hu, \hu)-f$. 
As such it is a clear continuous time analogue of
the filter update equation (\ref{eq:meanu}).

As in the discrete case (\ref{eq:meanu}), we
can study the stability of the continuous time filter
by expressing the noise in terms of an underlying
signal $u$ which the filter aims to uncover, and
study the difference $\hu-u.$ We assume that $u$ itself
solves the Navier-Stokes equation (\ref{eq:nse}).
We now express the observation signal $z$ in terms of
the truth $u$ in order to facilitate study of filter
stability.
If, instead of making the assumption of bounded noise
as in Assumption \ref{a:1z}, we assume a Gaussian noise
consistent with derivation of the filter, then we obtain 
$$\Bigl(\frac{z_{j+1}-z_j}{h}\Bigr)=y_{j+1}=u_{j+1}+ h^{\frac{1-r}{2}} \frac{\sigma_0}{\sqrt h} A_0^{-\beta}\Delta w_{j+1}$$
where $\{\Delta w_j\}_{j \ge 1}$ is an i.i.d. sequence
and $\Delta w_{1} \sim N(0,I)$ has the distribution of
a white noise in $H$.
If $r=1$ then 
we have an Euler-Maruyama discretization of the 
SDE (stochastic differential equation)
\begin{equation}
\label{eq:sdec}
\frac{\rd z}{\rd t}=u+ 
\sigma_0
A_0^{-\beta}\frac{\rd W}{\rd t},\quad z(0)=0, 
\end{equation}
where $W$ is a Brownian motion in time with covariance the identity
in $H$.  If $r \in (0,1)$ the limit is simply the ODE
\begin{equation}
\label{eq:odec}
\frac{\rd z}{\rd t}=u, \quad z(0)=0. 
\end{equation}

If $r=1$ then substituting 
the expression for $z$ from (\ref{eq:sdec})
into (\ref{eq:ctf}) we obtain the following SDE: 
\begin{equation}
\frac{\rd \hu}{\rd t} + \nu A\hu + \cB(\hu, \hu) +
\omega A_0^{-2\alpha}(\hu-u)= f+\omega \sigma_0
A_0^{-2\alpha-\beta}\frac{\rd W}{\rd t}, 
\quad \hu(0)=\hu_0.
\label{eq:nse2}
\end{equation}
On the other hand, if $r \in (0,1),$ then substituting 
the expression for $z$ from (\ref{eq:odec})
into (\ref{eq:ctf}) we obtain the following ODE: 
\begin{equation}
\frac{\rd \hu}{\rd t} + \nu A\hu + \cB(\hu, \hu) +
\omega A_0^{-2\alpha}(\hu-u)= f, 
\quad \hu(0)=\hu_0,
\label{eq:nse3}
\end{equation}
Equations (\ref{eq:nse2}) and (\ref{eq:nse3}) can be used
to study filter stability in this frequent observations
limit.  The equations are equivalent to an 
SPDE and PDE of Navier-Stokes
type, driven by noise in the former case, 
and with an additional damping term
driving the solution towards the signal $u$ in both
cases. We expect that for large enough $\omega$, which
is a form of variance inflation in this high frequency
data context, the solution will indeed stabilize
towards the signal $u$, although the nature of white noise
forcing means that excursions from the signal in the former 
case will occur infinitely often.  
These excursions will then be 
stabilized by the signal.  
It would be interesting to analyze the properties
of the SPDE and PDE by using the theory of
nonautonomous and random dynamical systems,
and the ergodicity theory developed 
in \cite{hairer2006ergodicity}.

\section{Numerical Results}
\label{sec:numerics}


In this section we describe a number of numerical
results designed to illustrate the range of filter
stability phenomena studied in the previous sections. 
We start, in subsection \ref{ssec:prob}, by
describing two useful bounds on the error committed
by filters; we will use these guides in the
subsequent numerics. Subsection \ref{ssec:results} describes
the common setup for all the subsequent
numerical results shown. Subsection \ref{ssec:complete}
describes these results in 
the case of complete observations in discrete time, whilst
Subsection \ref{ssec:partial} extends to the case of partial
observations, also in discrete time.
Subsection \ref{ssec:continuous} studies filter
stability in the case of continuous time observations, 
using the (S)PDEs derived at the
end of section \ref{sec:stability}.

Our theoretical results have been derived under
Assumptions \ref{a:1} and \ref{a:1z} on the errors.
These are incompatible
with the assumption, underpinning derivation of the
approximate filters, that the observational noise sequence
is Gaussian. This is because i.i.d Gaussian sequences
will not have finite supremum. In order to test the
robustness  of our theory we will conduct numerical
experiments with Gaussian noise sequences. 

\subsection{Useful Error Bounds}
\label{ssec:prob}

We describe two useful bounds on the error
which help to guide and evaluate the numerical
simulations. To derive these bounds we assume that
the observational noise sequence $\xi_j$ is i.i.d
with $\E \xi_j=0$ and $\E \xi_j \otimes \xi_j = \Gamma$.  
Then 
$$\E |\xi_j|^2={\rm tr} (\Gamma) = \sum_k g_k$$ 
where $\{g_k\}$ are the eigenvalues of the operator 
$\Gamma.$ 
(This operator must be trace class if the Gaussian
measure $N(0,\Gamma)$, used to derive the approximate filters,
is to defined on $H$).

\begin{itemize}

\item The lower bound is derived from 
(\ref{eq:error}). Using the assumed independence of
the sequence we see that 
\begin{equation}
\E |\hu_{j+1}-u_{j+1} |^2 \ge \E |(I-B_j)\xi_{j+1} |^2=
{\rm tr}\Bigl((I-B_j)\Gamma(I-B_j)^*\Bigr)
\label{eq:lowerbd}
\end{equation}

\item The upper bound on the filter error is found by
noting that a trivial filter is obtained by 
simply using the observation sequence as the
filter mean; this corresponds to setting
$B_j \equiv 0$ in (\ref{eq:meanu}). For this filter
we obtain 
\begin{equation}
\E |\hu_{j+1}-u_{j+1} |^2 \le \E |\xi_{j+1} |^2=
{\rm tr}\bigl(\Gamma\bigr)
\label{eq:upperbd}
\end{equation}
in the case of complete observations, and
\begin{equation}
\E |\hu_{j+1}-u_{j+1} |^2 \le \E |\xi_{j+1} |^2+|\Ql u_{j+1}|^2=
{\rm tr}\bigl(\Gamma\bigr)+|\Ql u_{j+1}|^2
\label{eq:upperbd2}
\end{equation}
in the case of incomplete observations. 

\end{itemize}

Although the lower bound (\ref{eq:lowerbd})
does not hold {\em pathwise}, only on average,
it provides a useful guide for our pathwise experiments. 
The upper bounds (\ref{eq:upperbd}) and (\ref{eq:upperbd2})
do not apply to any numerical
experiment conducted with non-zero $B_j$, but also serve
as a useful guide to those experiments: it is clearly
undesirable to greatly exceed the error committed by simply
trusting the data. We will hence plot the lower
and upper bounds as useful comparitors for
the actual error incurred in our numerical experiments below.
We note that, for the 3DVAR example from subsection
\ref{ssec:3DVar} with complete observations, 
the upper and lower bounds coincide 
in the limit $\eta \to 0$ as then $B \to 0$.
For partial observations they differ by the second term
in the upper bound.

\subsection{Experimental Setup}
\label{ssec:results}

For all the results shown we 
choose a box side of length $L=2$. 
The forcing in Eq. (\ref{eq:nse}) is taken to be $f=\nabla^{\perp}\psi$,
where $\psi=\cos(\pi k \cdot x)$ and $\nabla^{\perp}=J\nabla$ with $J$
the canonical skew-symmetric matrix, and $k=(5,5)$.  
The method used to approximate the forward model (\ref{eq:nse})
is a modification of a fourth-order Runge-Kutta method,  
ETD4RK \cite{cox2002exponential}, 
in which the Stokes semi-group is computed exactly 
by working in the incompressible Fourier
basis  $\{\psi_{k}(x)\}_{k \in {\mathbb Z}^2\backslash\{0\}}$,
and Duhamel's principle (variation of constants formula) is 
used to incorporate the nonlinear term.
Spatially, a Galerkin spectral method \cite{hesthaven2007spectral} 
is used, in the same basis, 
and the convolutions arising from products in the nonlinear 
term are computed via FFTs.
We use a double-sized domain in each dimension, buffered with
zeros, resulting in $64^2$ grid-point FFTs, and only half the 
modes in each direction are retained when transforming back 
into spectral space in order to prevent aliasing, which is avoided
as long as fewer than 2/3 of the modes are retained.

The dimension of the attractor is determined by the 
viscosity parameter $\nu$. For the particular forcing used
there is an explicit steady state for all $\nu>0$
and for $\nu \geq 0.035$ this solution is stable (see
\cite{majda2006non}, Chapter 2 for details). 
As $\nu$ decrease the flow becomes increasing complex
and the regime $\nu \leq 0.016$ corresponds to 
strongly chaotic dynamics with attributes of turbulent 
scalings in the spectrum.  
We focus subsequent studies of the filter on a  
mildly turbulent ($\nu = 0.01$) parametric regime.
For this small viscosity parameter, we use a time-step 
of $\delta t = 0.005$.

The data is generated by computing a true signal
solving \eref{eq:nse} at the desired value of $\nu$, and
then adding Gaussian random noise to it at each observation
time. Such noise does not satisfy Assumption \ref{a:1z}, since
the supremum of the norm of the noise sequence is not finite, and so 
this setting provides a severe test beyond what is 
predicted by the theory; nonetheless, it should be noted that 
Gaussian random variables only obtain arbitrarily large values 
arbitrarily rarely. 
 
All experiments are conducted using the 3DVAR setup and
it is useful to reread the end of subsection \ref{ssec:3DVar}
in order to interpret the parameters $\alpha$ and $\eta.$ 
We consider both the choices $\alpha=\pm 1$ for 3DVAR, noting
that in the case $\alpha=-1$ the operator $B$ has norm strictly
less than one and so we expect the algorithm to be more robust
in this case (see Remark \ref{r:2} for discussion of this fact). 
For 
all experiments we set $\ell=\lambda_1^{-1}$
which ensures that the action of $A_0^{2\alpha}$, and hence
$B$, on the first eigenfunction is independent of the value
of $\alpha$; this is a useful normalization when comparing
computations with $\alpha=1$ and $\alpha=-1.$

In the discrete time experiments
we set the observational noise to white
noise $\Gamma= \eps^2 I$ (i.e. $\beta=0$ in section \ref{ssec:3DVar}).  
Here $\eps = 0.04$, which gives
a standard deviation of approximately 10\% of the maximum 
standard deviation of the turbulent dynamics.  
Since we are computing in a truncated finite-dimensional basis 
the eigenvalues are summable; the situation 
can be considered as an approximation of an operator
whose eigenvalues decay rapidly outside the basis in which
we compute.

\subsection{Complete Observations; Discrete Time}
\label{ssec:complete}

We start by considering discrete and complete observations and
illustrate Theorem \ref{t:mz}, and in particular the role of 
the parameter $\eta.$
The experiments presented employ a
large observation increment of $h = 0.5 = 100 \delta t$.
For $\alpha=1$ we find that 
when $\eta=\sigma$ (Fig. \ref{a1.1}) the estimator 
stabilizes from an initial ${\mathcal O}(1)$ error
and then remains stable. The upper and lower bounds are
satisfied (the upper bound after an initial rapid transient), 
and even the high modes, which are slaved to the low modes, 
synchronize to the true signal. 
For $\eta=10\sigma$ (Fig. \ref{a1.10})
the estimator fails to satisfy the upper bound, 
but remains stable over a long time horizon; there
is now significant error in the $k=(7,7)$ mode,
in contrast to the situation with smaller $\eta$ shown 
in Fig. \ref{a1.1}.  
Finally, when $\eta=100\sigma$ (Fig. \ref{a1.100}), the estimator really
diverges from the signal, although still remains
bounded. 
 
When $\alpha=-1$ the lower and upper bounds are 
almost indistinguishable and, for all values of
$\eta$ examined, the error either exceeds or fluctuates
around the upper bound; see Figures \ref{am1.1}, \ref{am1.10}
and \ref{am1.100} where $\eta=\sigma, 10\sigma$ and $100\sigma$
respectively. It is not until $\eta=100\sigma$ (Fig. 
\ref{am1.100}) that the estimator really loses the signal.  
Notice also that the high modes of the estimator 
always follow the noisy observations and this could be 
undesirable.  For both $\eta=100\sigma$ and $10\sigma$, the $\alpha=-1$ estimator 
performs better than the one for $\alpha=1$ in terms of overall error, 
illustrating the robustness alluded to in Remark \ref{r:2}
since for $\alpha<0$ we have $\|B\|<1.$  
However, an appropriately 
tuned $\alpha=1$ filter has the potential to perform 
remarkably well, both in terms of
overall error and individual error of all modes
(see Fig. \ref{a1.1}, in contrast to 
Fig. \ref{am1.1}).  In particular,
this filter has an expected error substantially smaller than the 
upper bound, which does not happen 
for the case of $\alpha=-1$ when complete observations
are assimilated.

\begin{figure*}
\includegraphics[width=.45\textwidth]{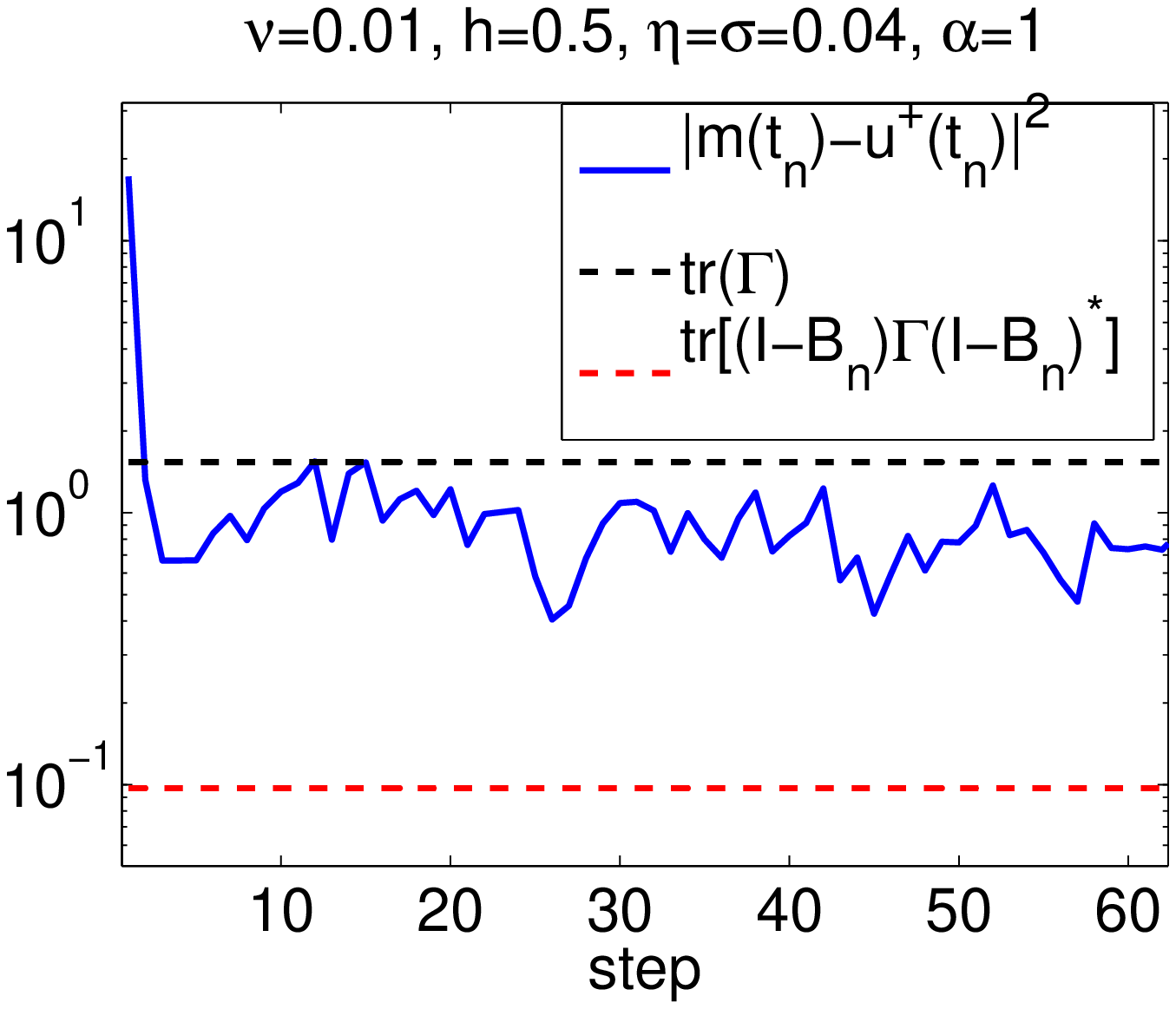}
\includegraphics[width=.45\textwidth]{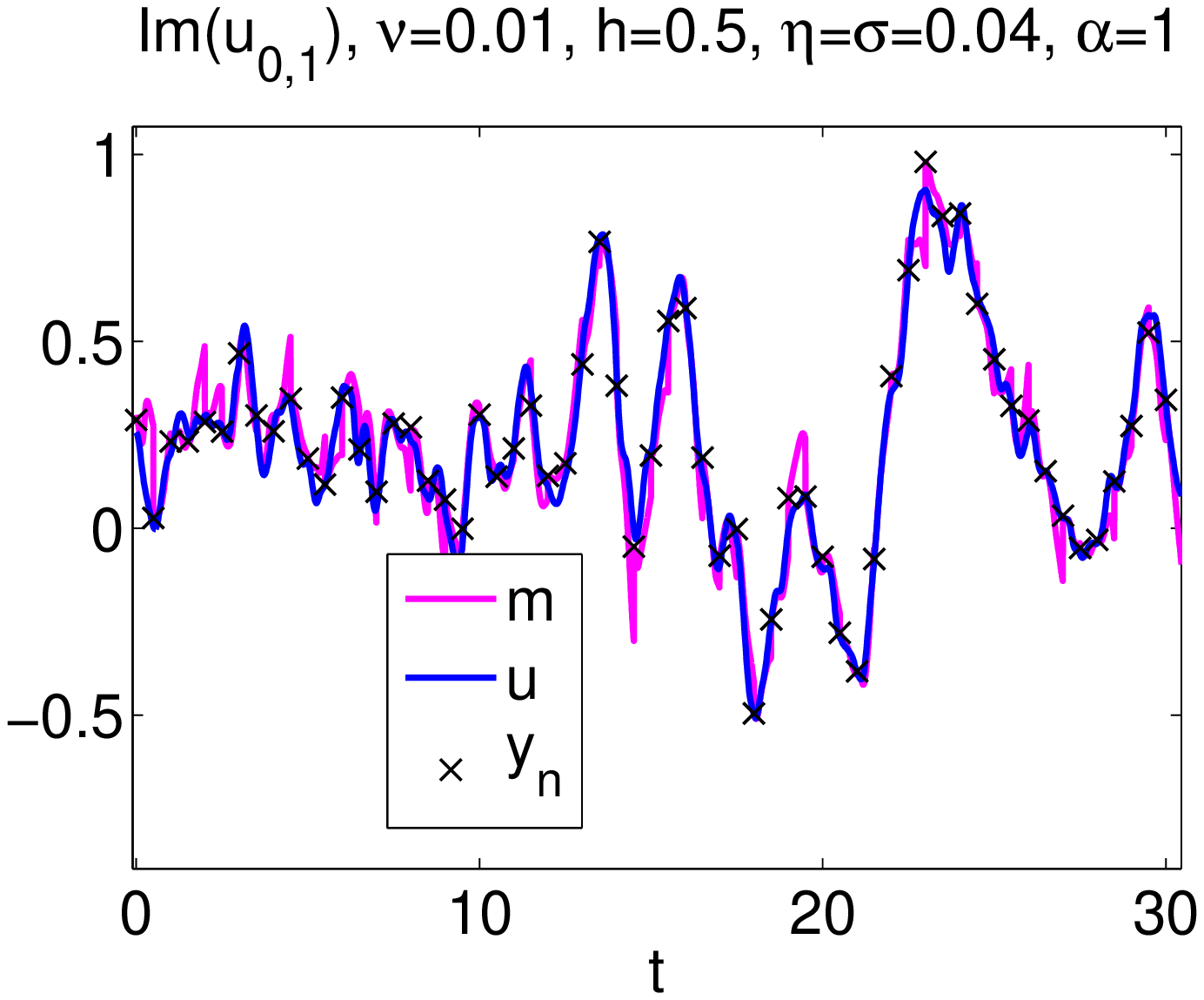}
\includegraphics[width=.45\textwidth]{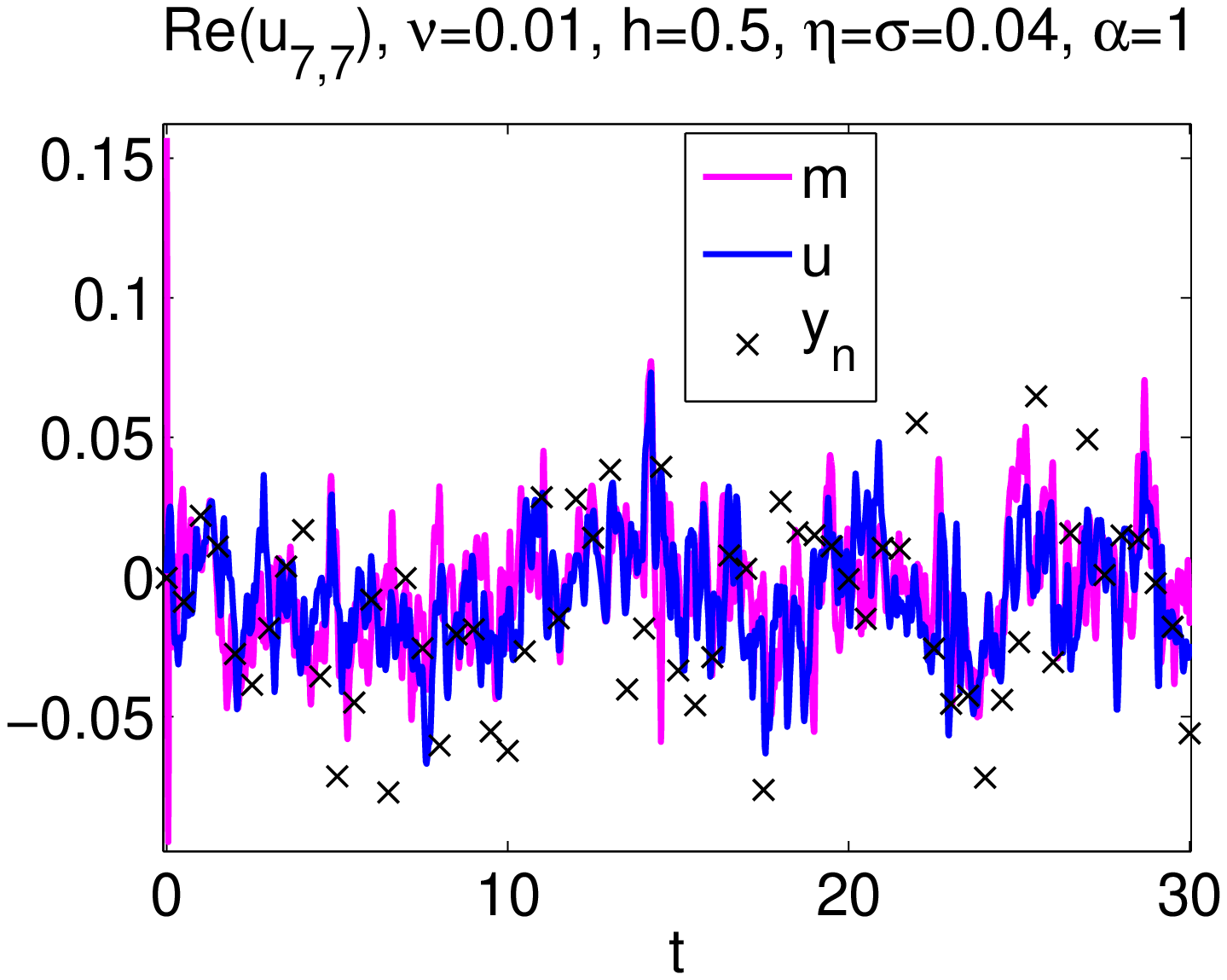}
\includegraphics[width=.45\textwidth]{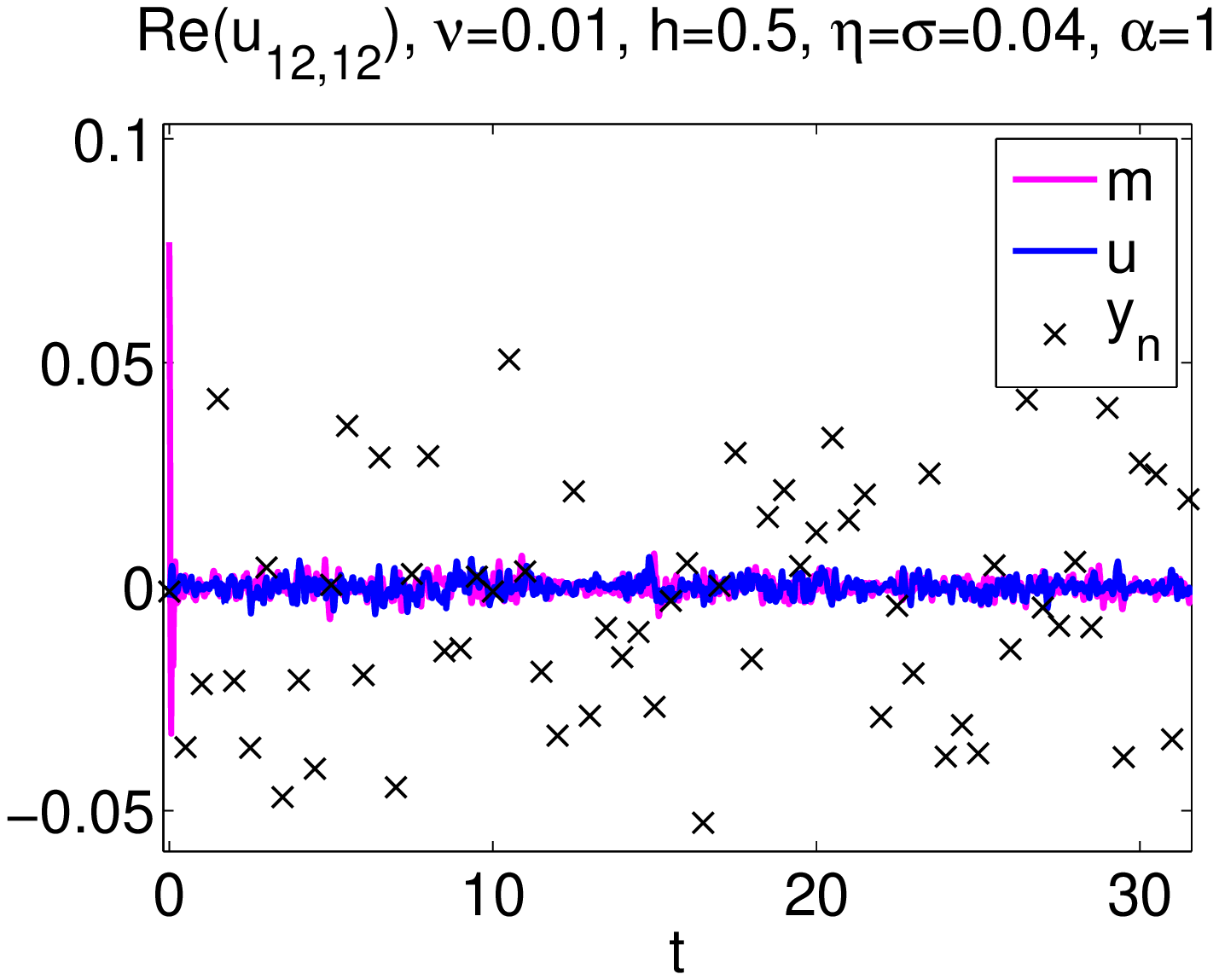}
\caption{Example of a stable trajectory for 3DVAR with $\nu=0.01,
  h=0.5, \eta=\sigma=0.04, \alpha=1$.  
The top left plot shows the norm-squared error between the
estimated mean, $m(t_n)=\hat{m}_n$, and the signal, $u(t_n)$, 
in comparison to the preferred upper
bound (i.e. the total observation error ${\rm tr} (\Gamma) = \Xi$) 
and the lower bound 
${\rm tr} [(I-B_n) \Gamma (I-B_n)]$.  The other three plots show the 
estimator, $m(t)$, together with the signal, $u(t)$, and the
observations, $y_n$ for a few individual modes.}
\label{a1.1}
\end{figure*}

\begin{figure*}
\includegraphics[width=.45\textwidth]{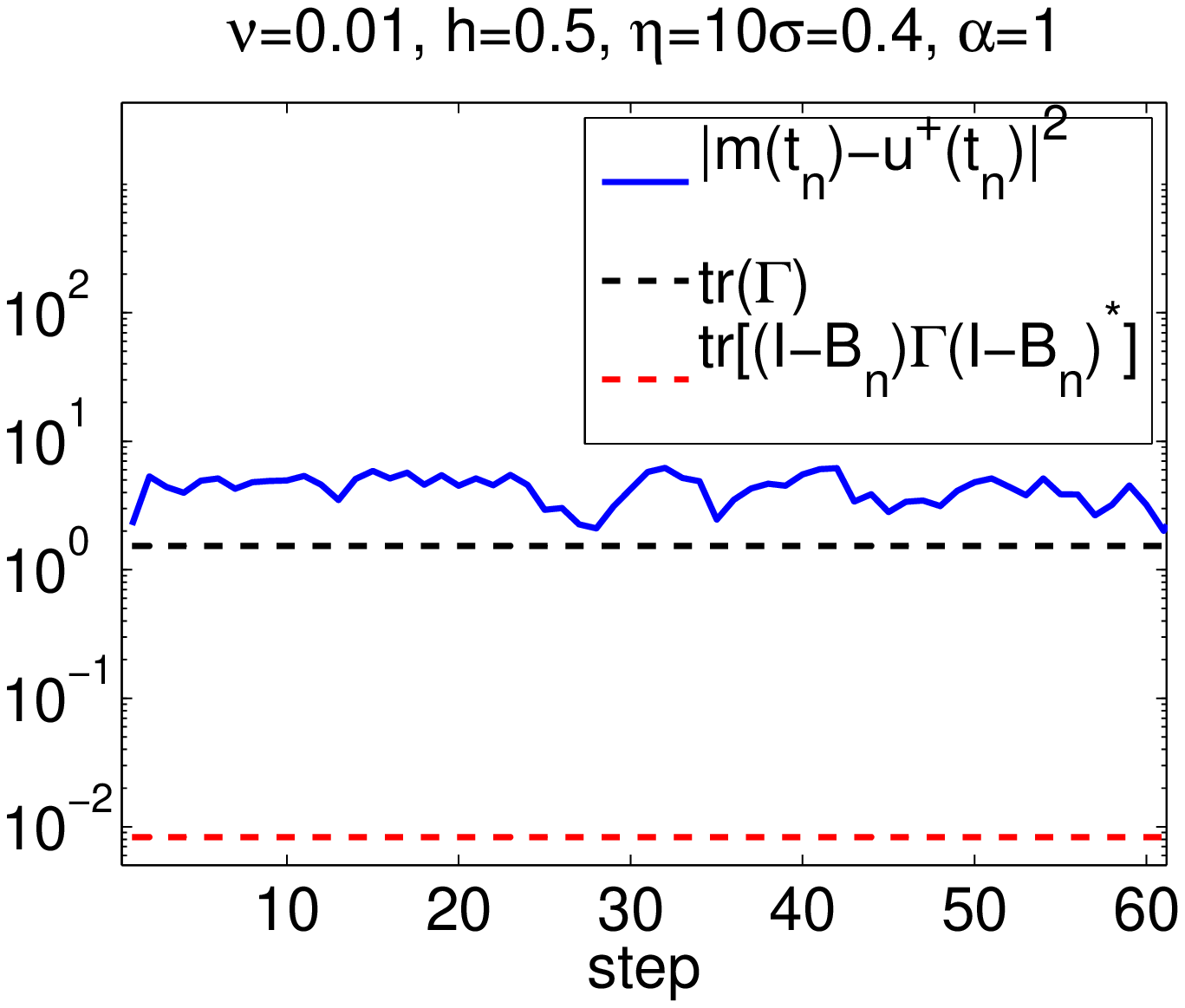}
\includegraphics[width=.45\textwidth]{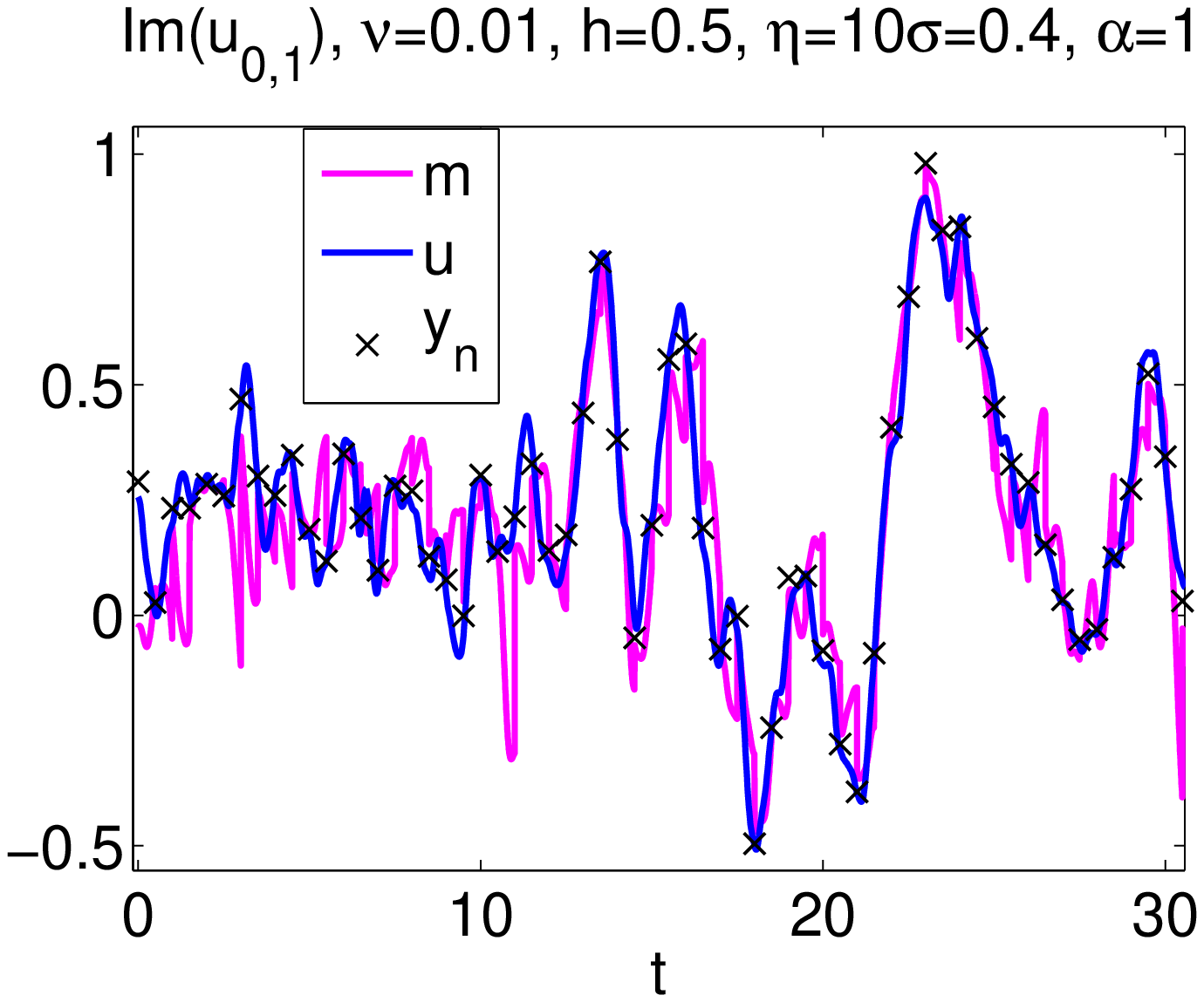}
\includegraphics[width=.45\textwidth]{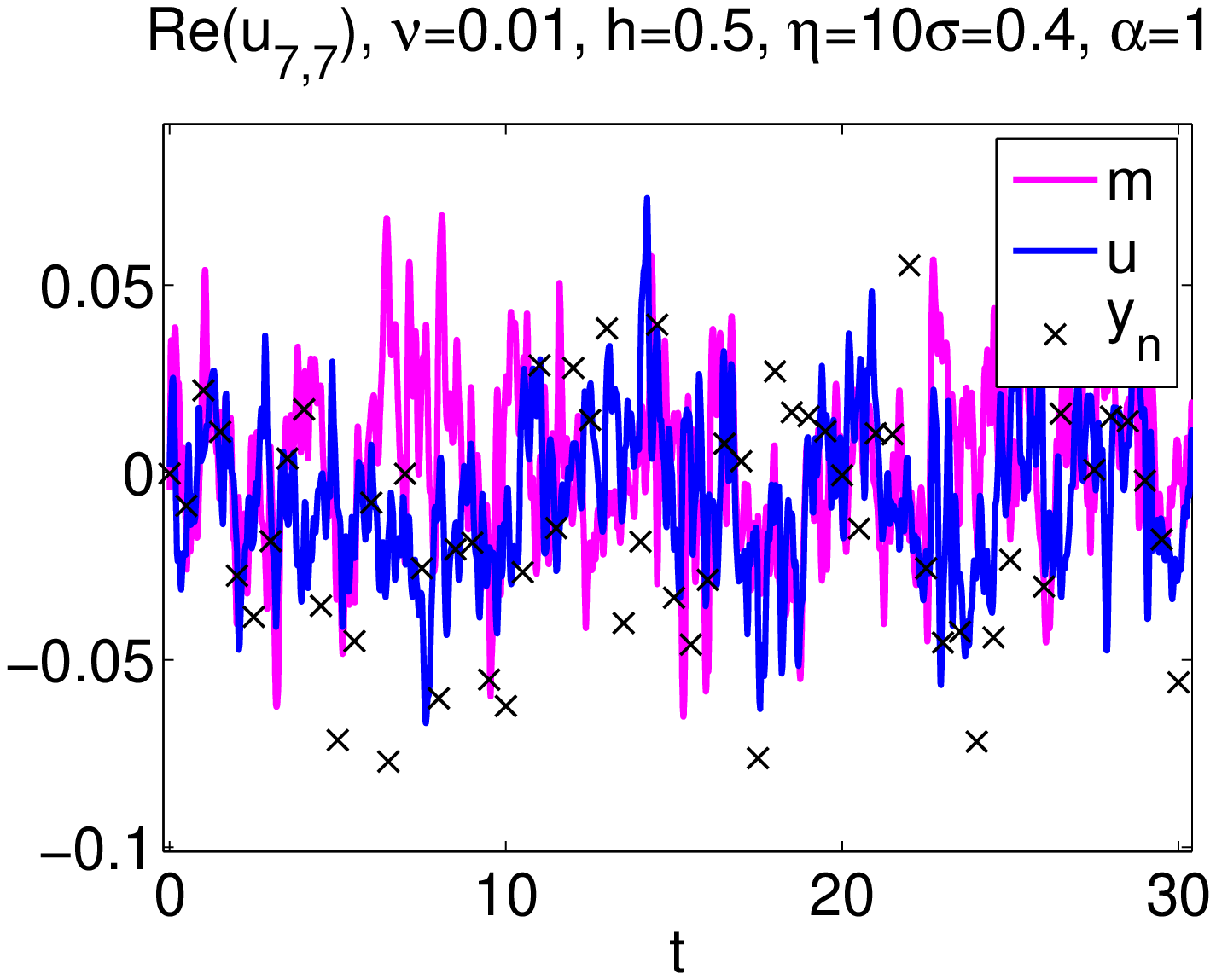}
\includegraphics[width=.45\textwidth]{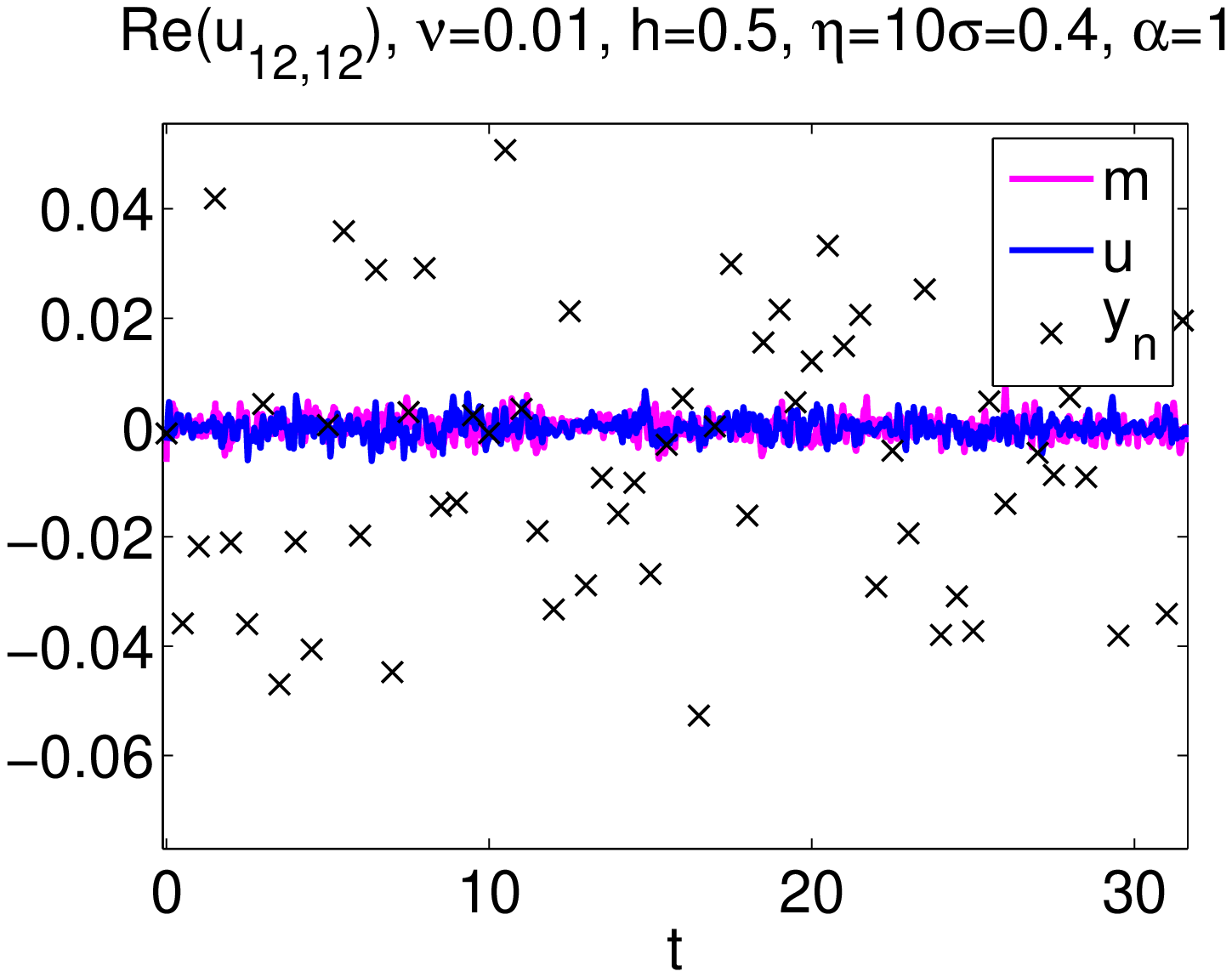}
\caption{Example of a destabilized trajectory for 3DVAR with 
the same parameters as in Fig. \ref{a1.1} except the larger
value of $\eta=10\sigma=0.4$.  Panels are the same.}
\label{a1.10}
\end{figure*}

\begin{figure*}
\includegraphics[width=.45\textwidth]{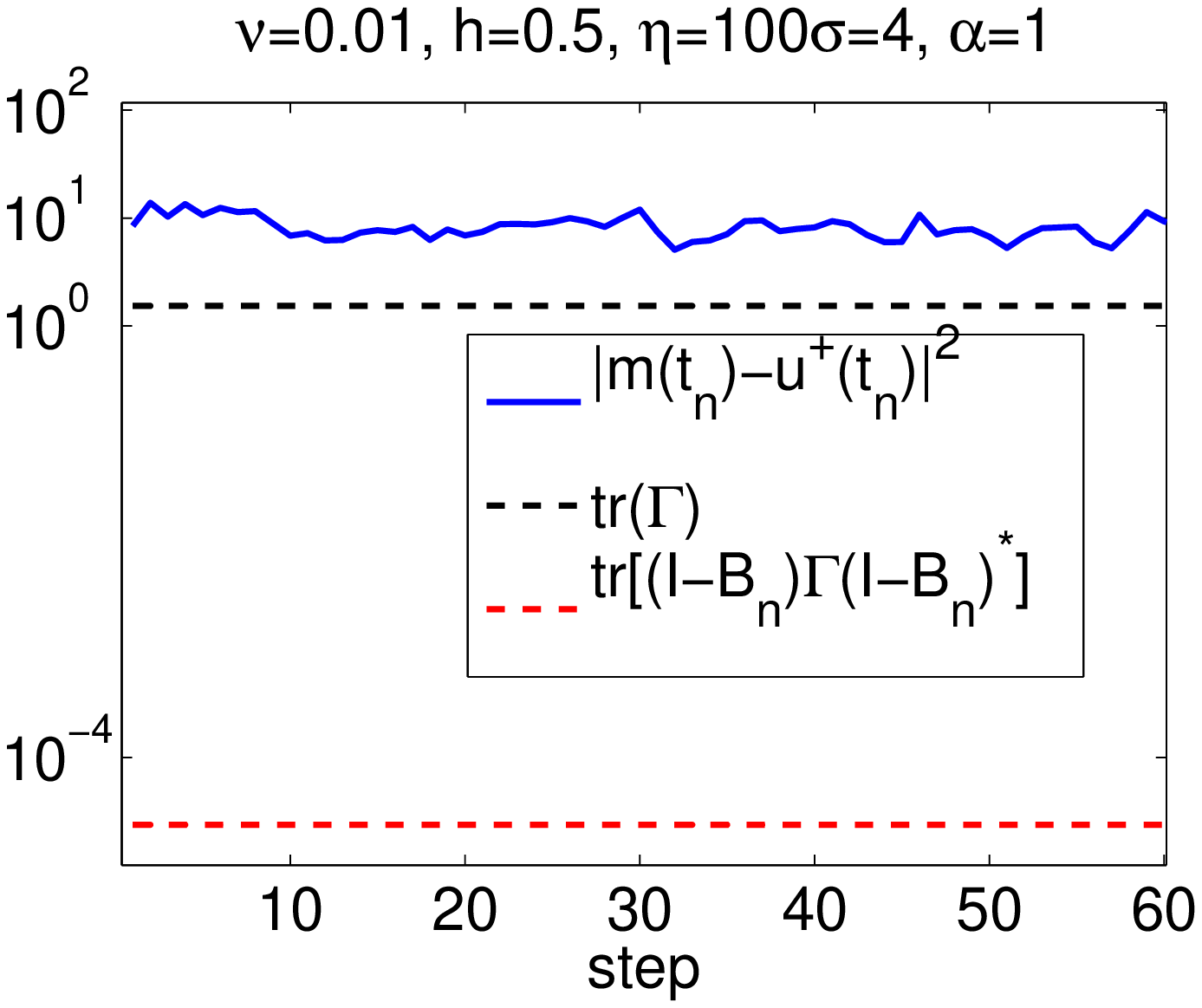}
\includegraphics[width=.45\textwidth]{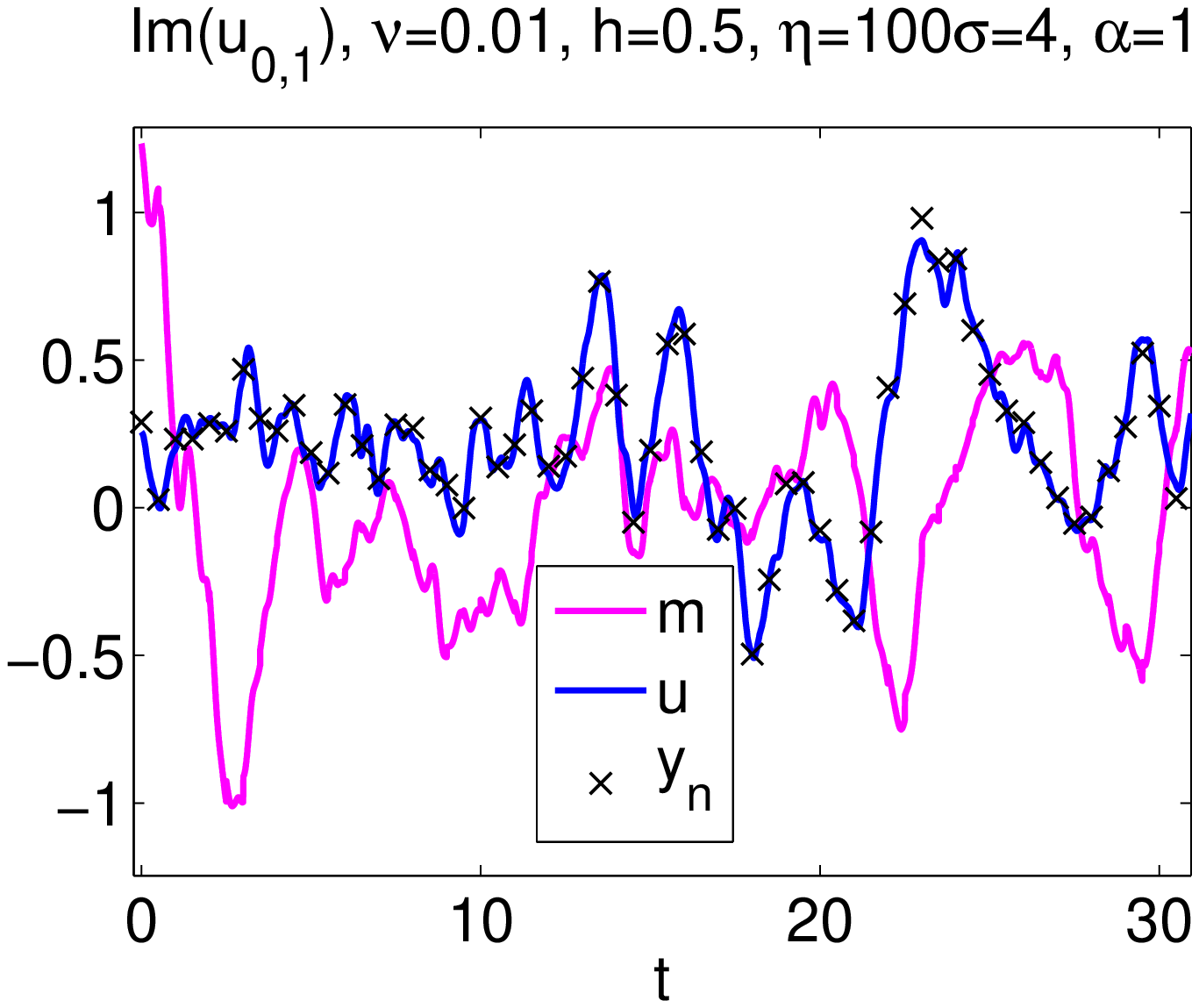}
\includegraphics[width=.45\textwidth]{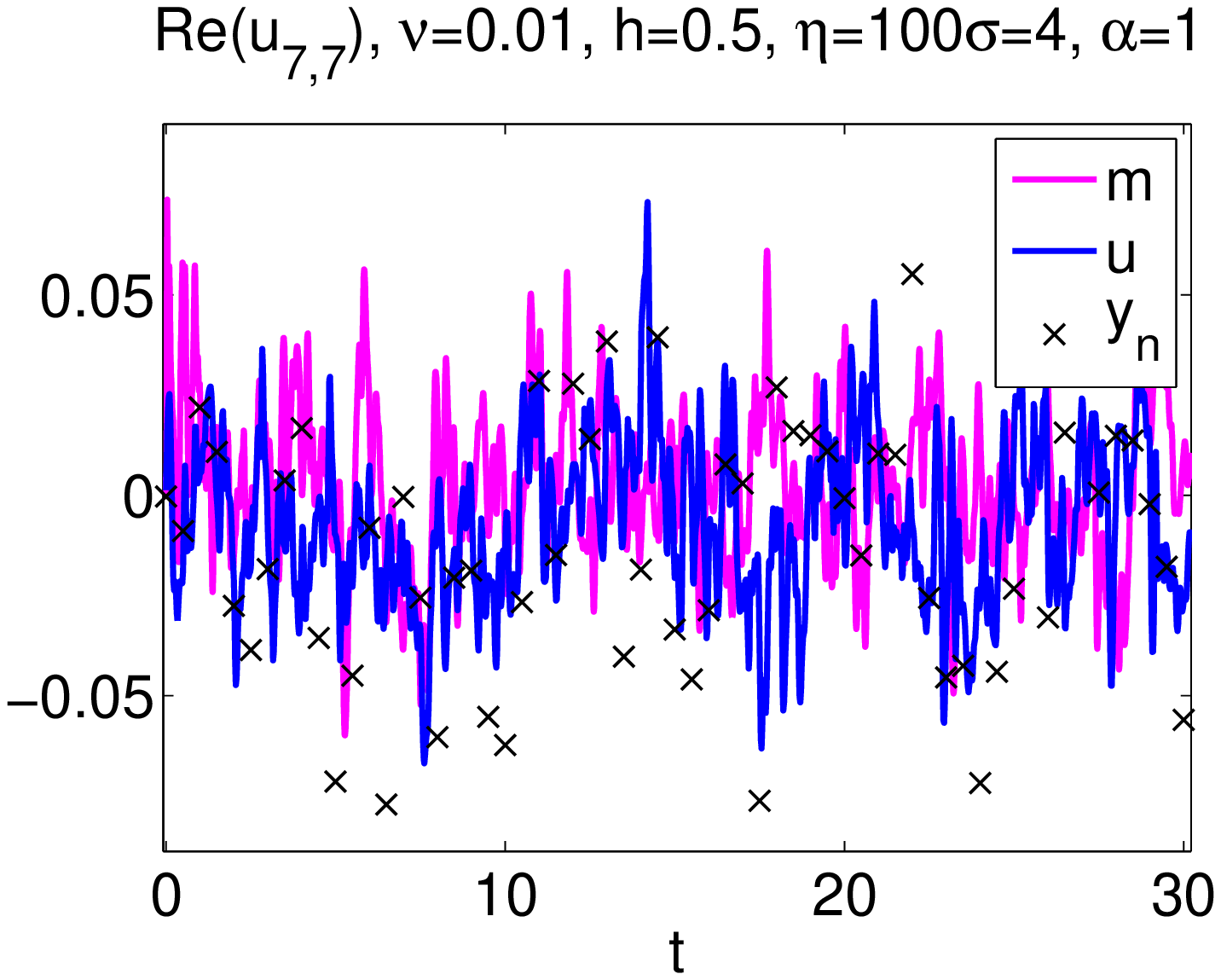}
\includegraphics[width=.45\textwidth]{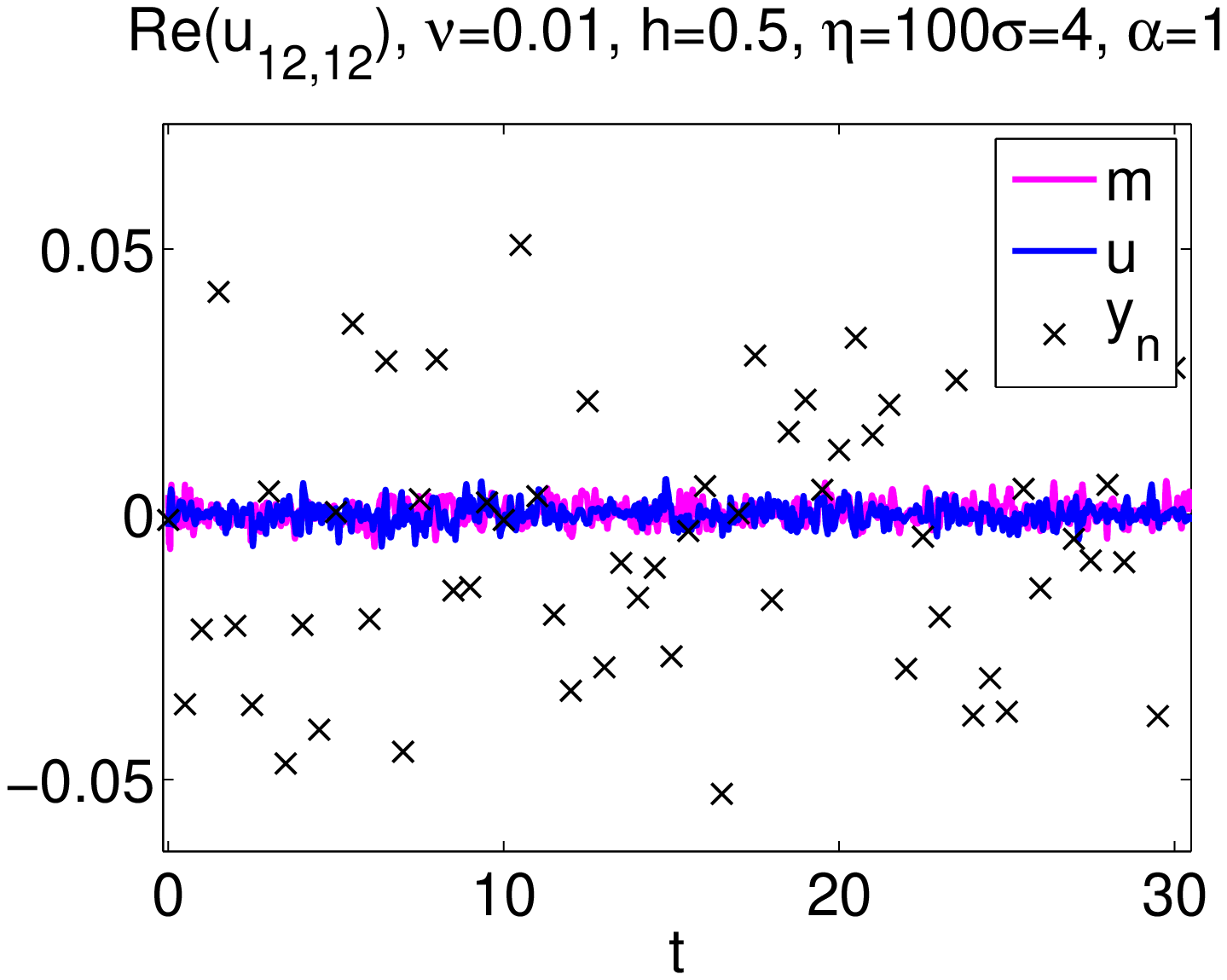}
\caption{Example of a destabilized trajectory for 3DVAR with 
the same parameters as in Fig. \ref{a1.1} except the larger
value of $\eta=100\sigma=4$.  Panels are the same.}
\label{a1.100}
\end{figure*}

\begin{figure*}
\includegraphics[width=.45\textwidth]{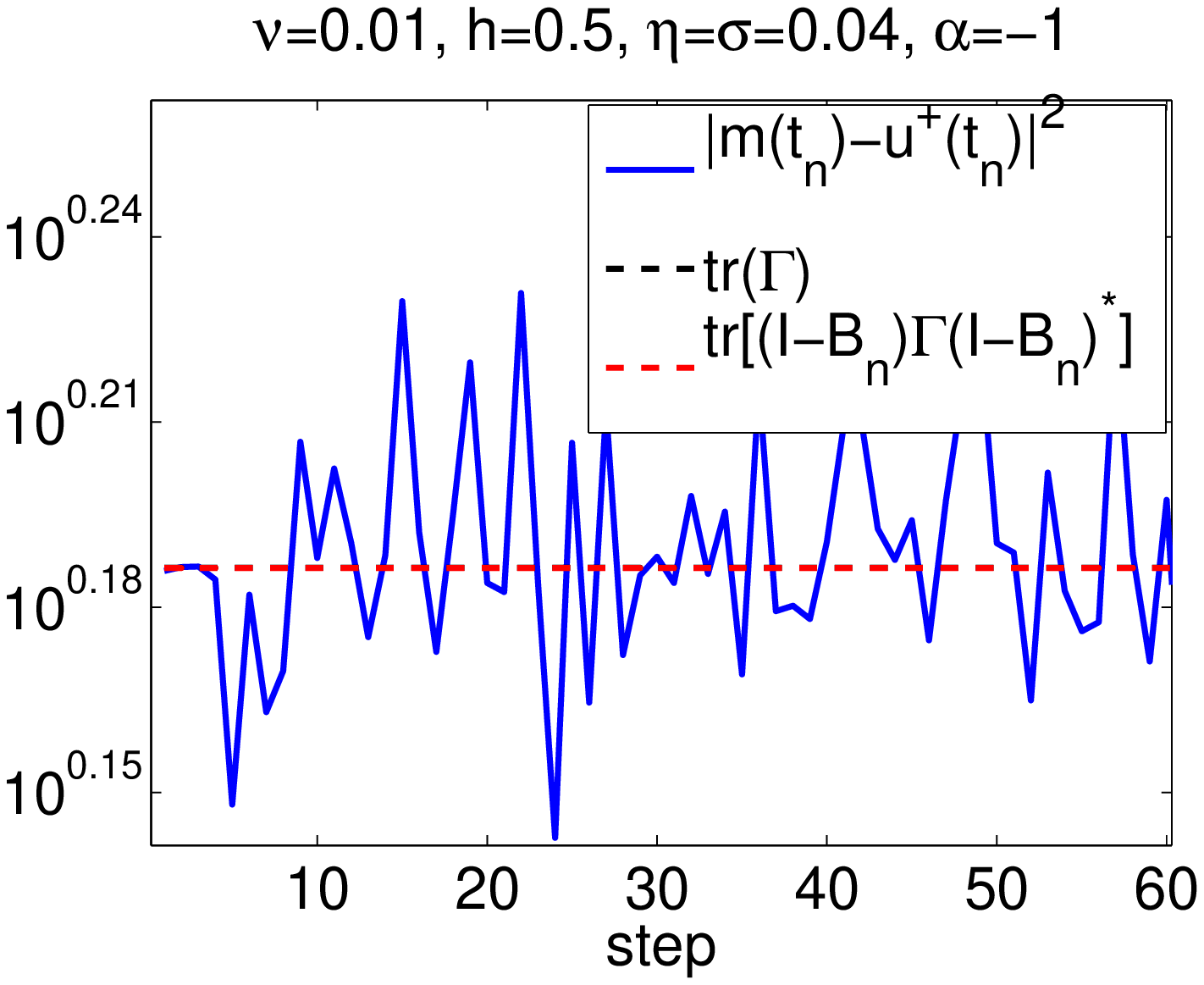}
\includegraphics[width=.45\textwidth]{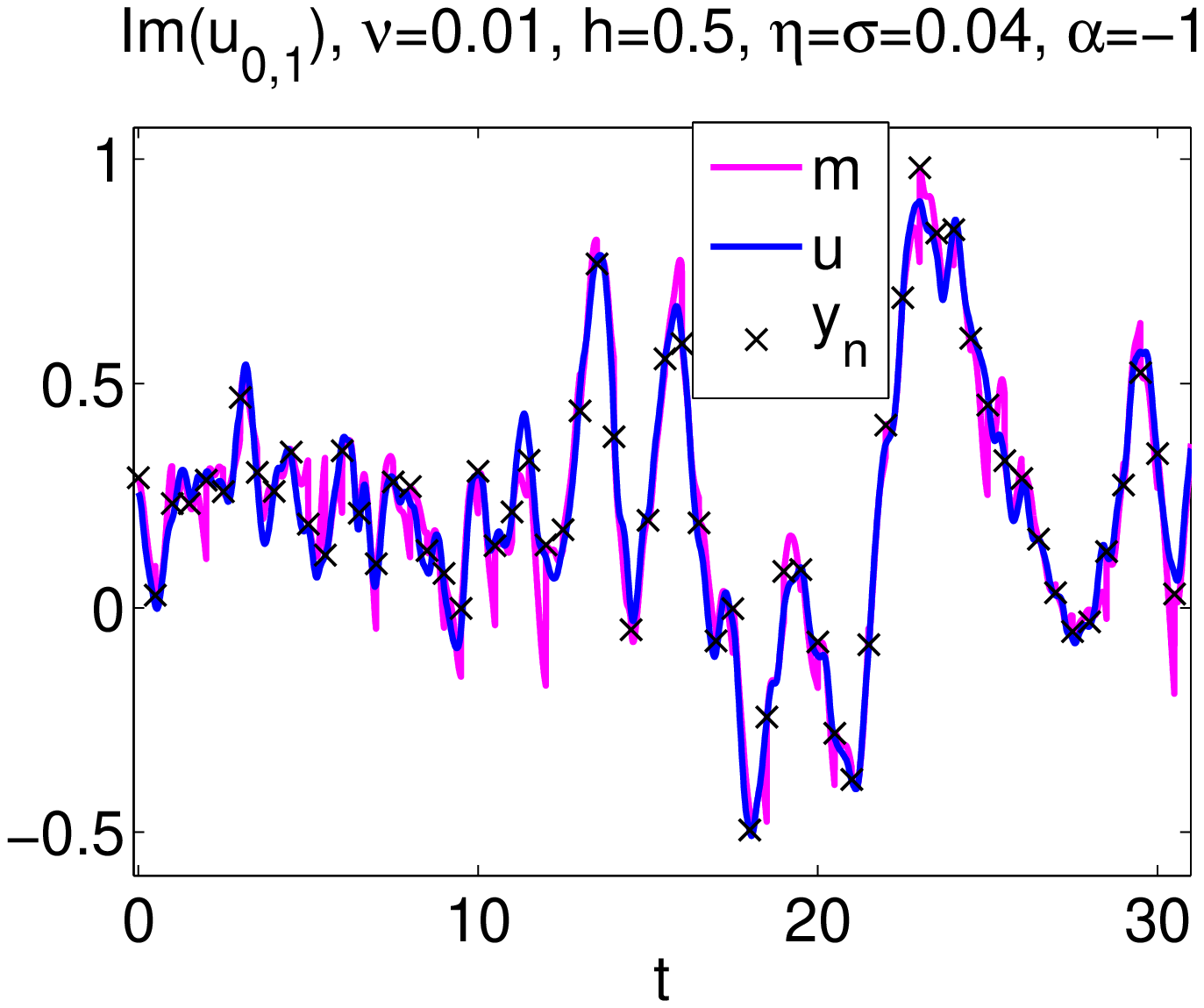}
\includegraphics[width=.45\textwidth]{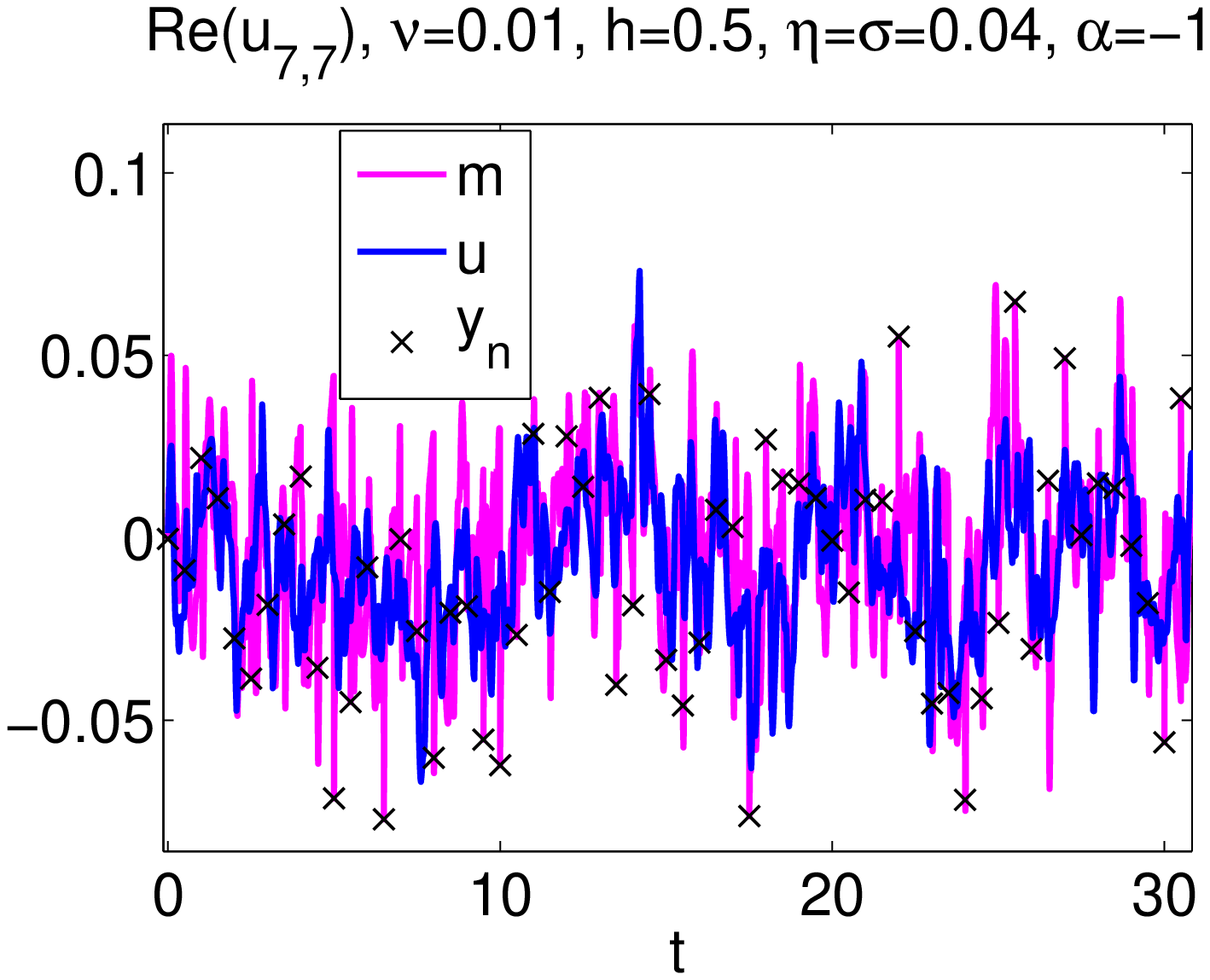}
\includegraphics[width=.45\textwidth]{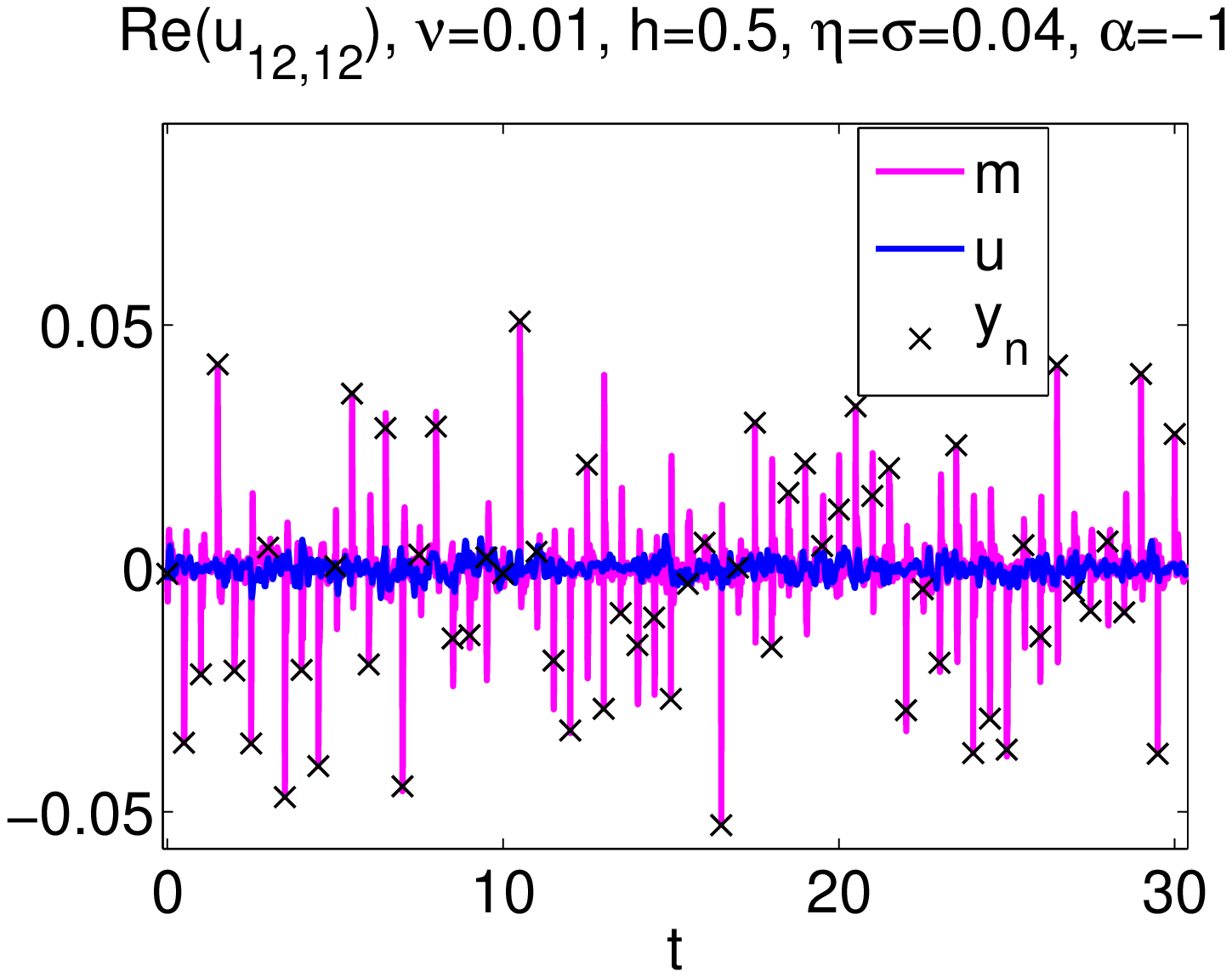}
\caption{Example of a stable trajectory for 3DVAR with 
the same parameters as in Fig. \ref{a1.1} except with
$\alpha=-1$.  Panels are the same.}
\label{am1.1}
\end{figure*}

\begin{figure*}
\includegraphics[width=.45\textwidth]{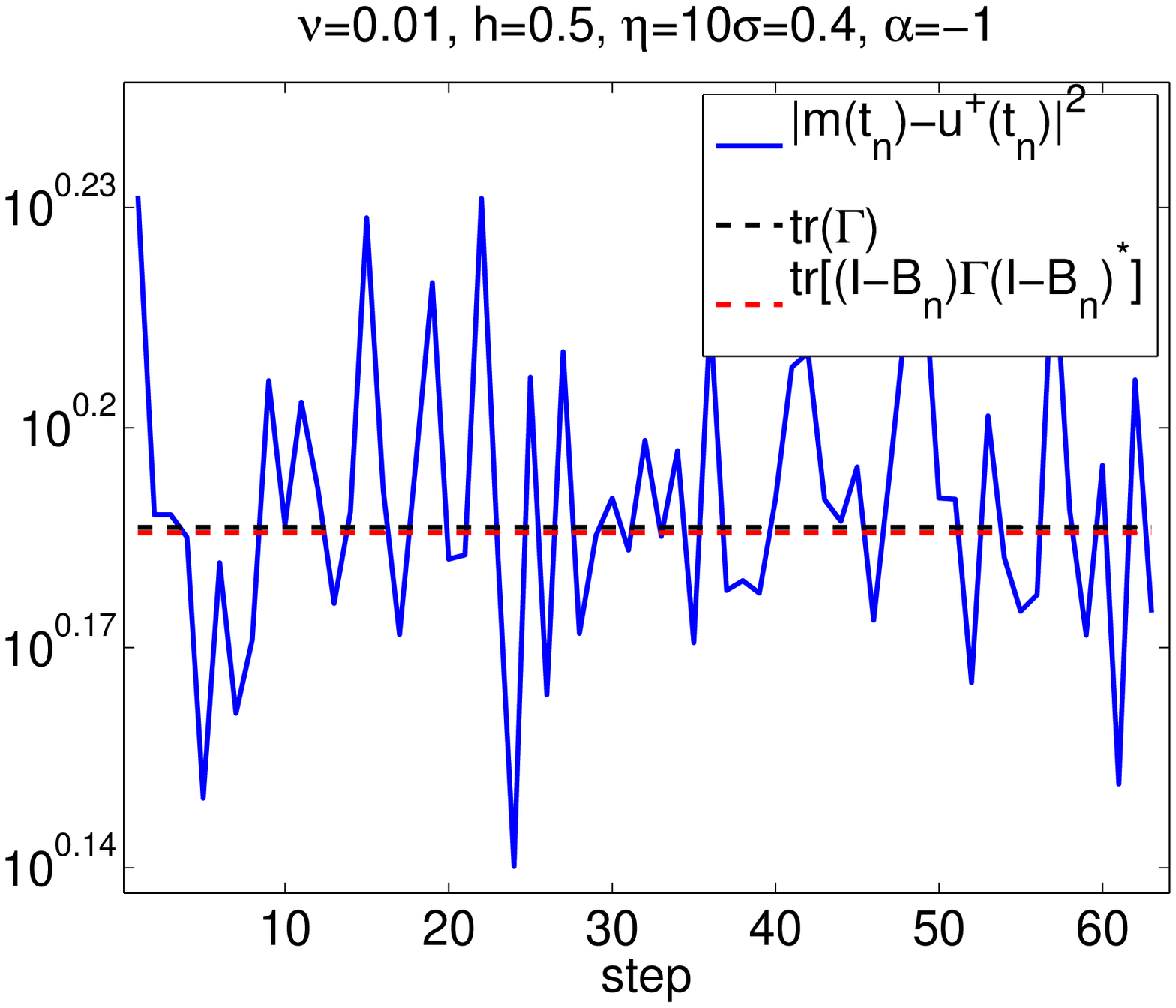}
\includegraphics[width=.45\textwidth]{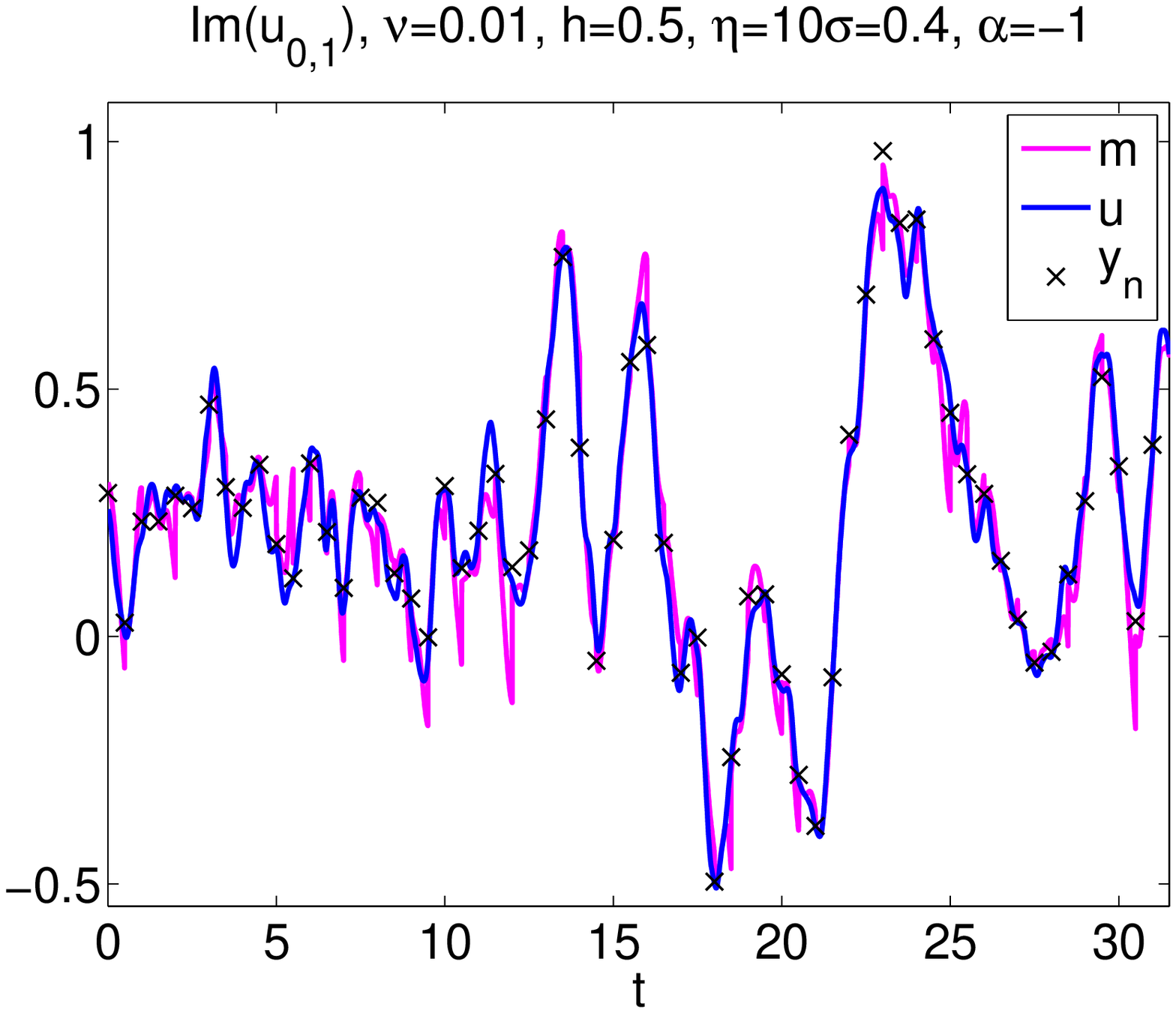}
\includegraphics[width=.45\textwidth]{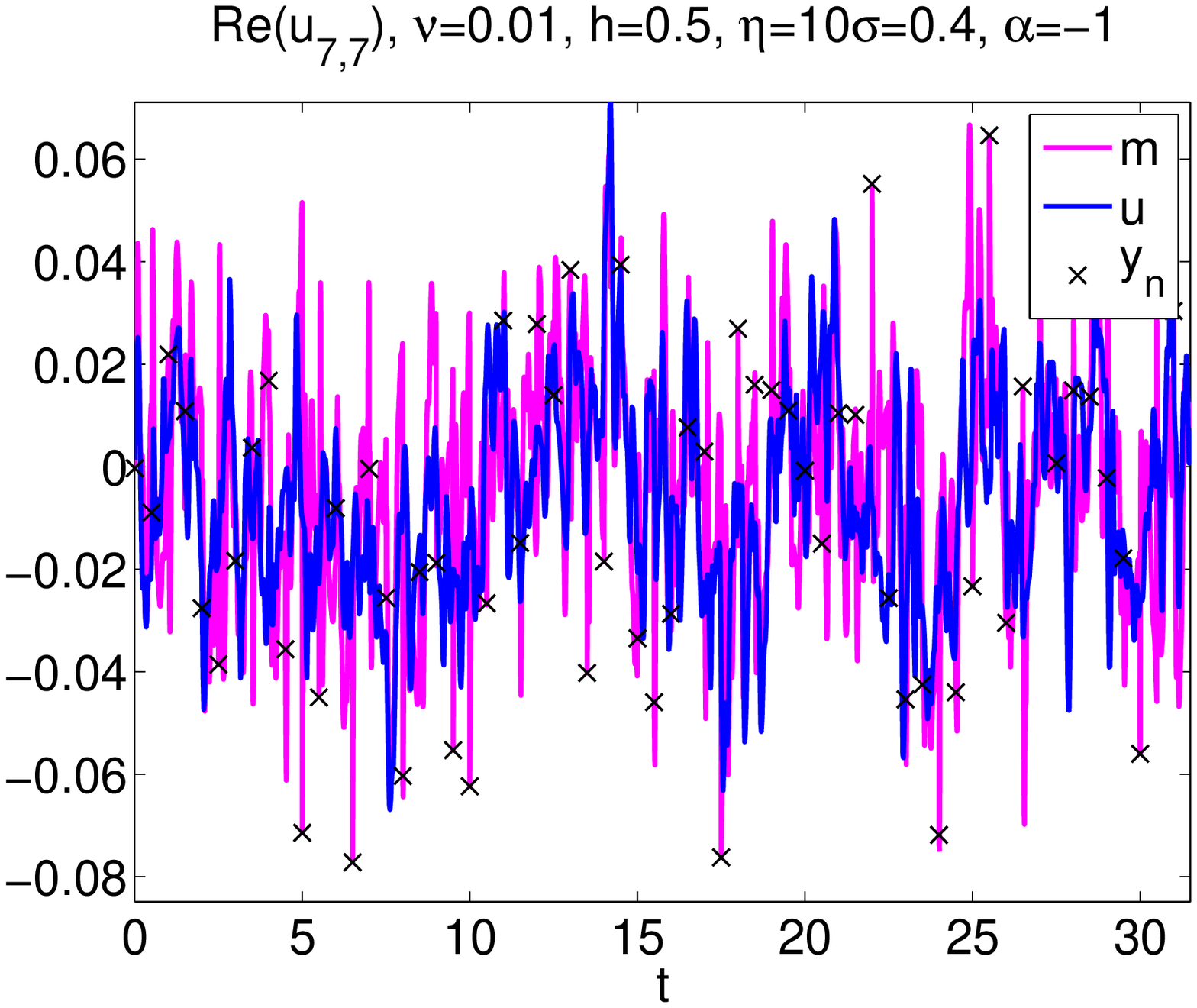}
\includegraphics[width=.45\textwidth]{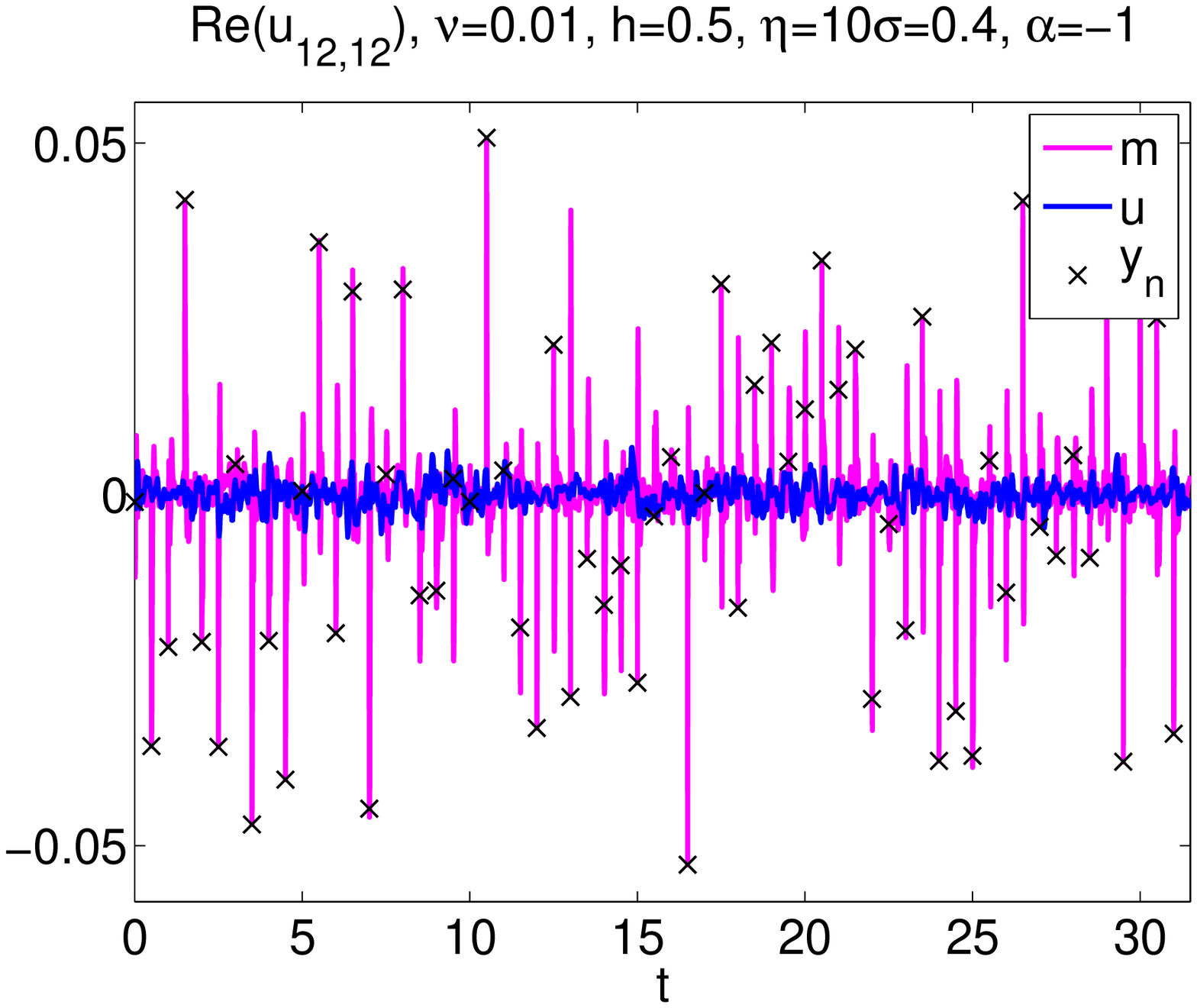}
\caption{Example of a stable trajectory for 3DVAR with 
the same parameters as in Fig. \ref{a1.10} except with
value of $\alpha=-1$.  Panels are the same.}
\label{am1.10}
\end{figure*}

\begin{figure*}
\includegraphics[width=.45\textwidth]{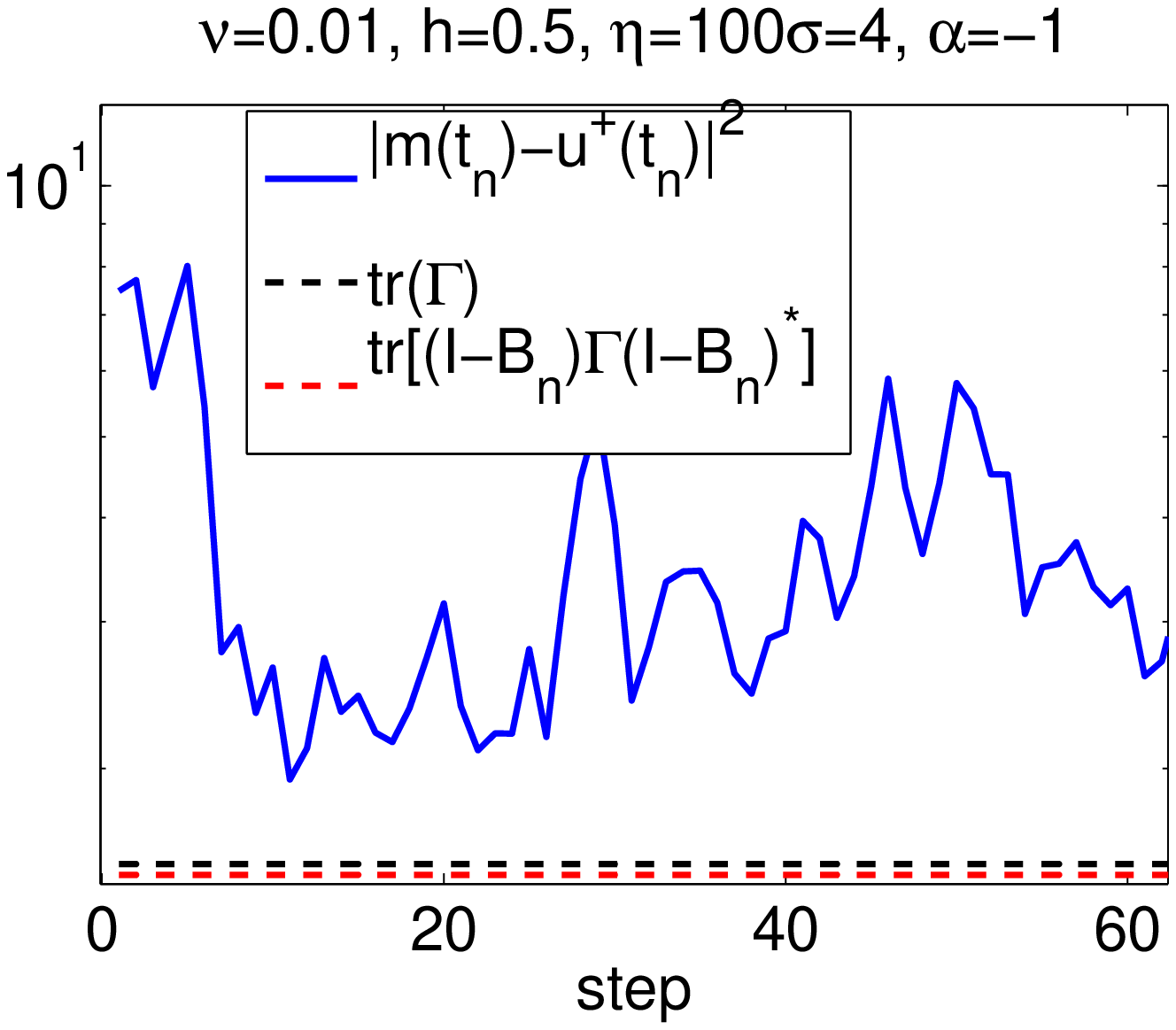}
\includegraphics[width=.45\textwidth]{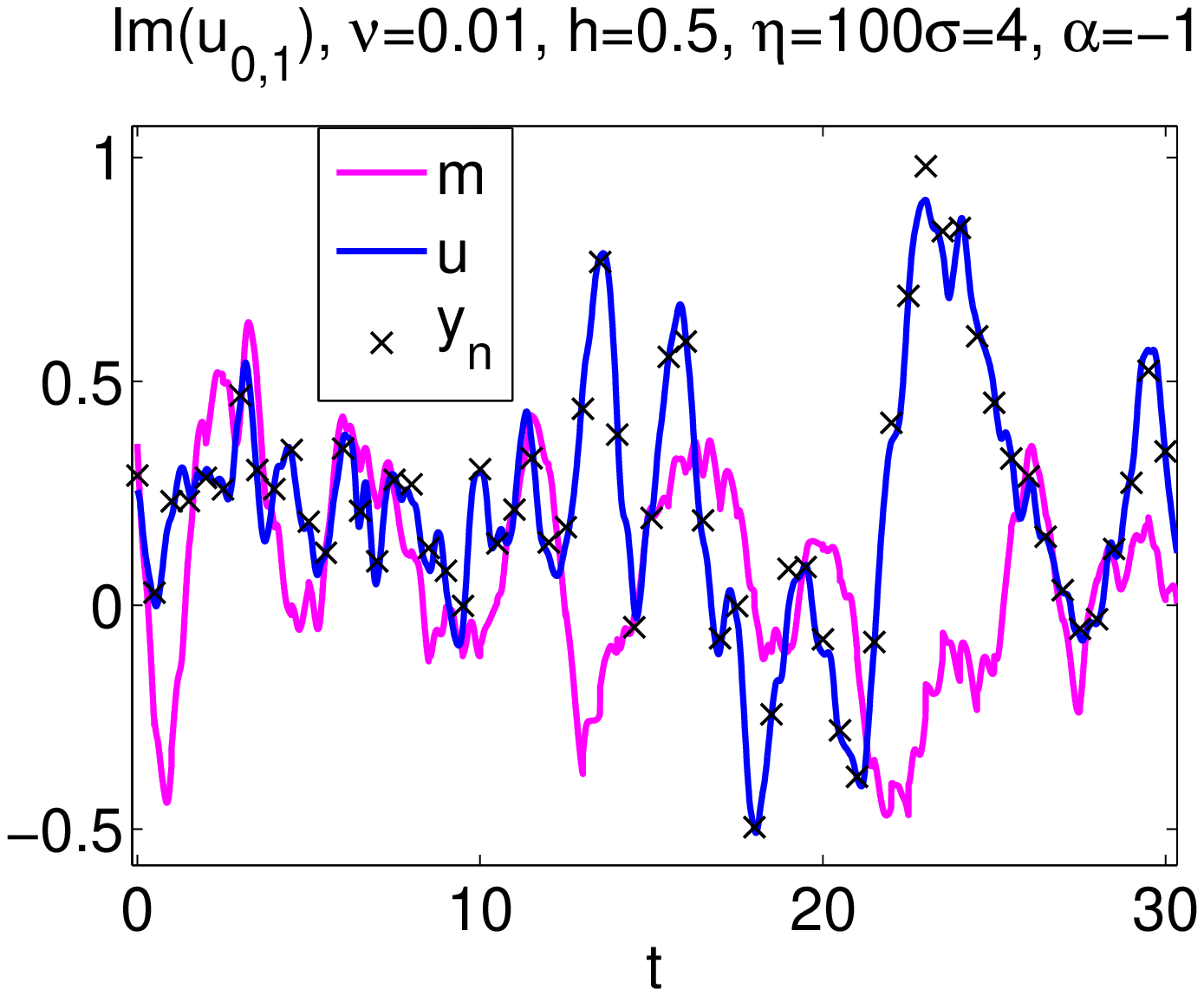}
\includegraphics[width=.45\textwidth]{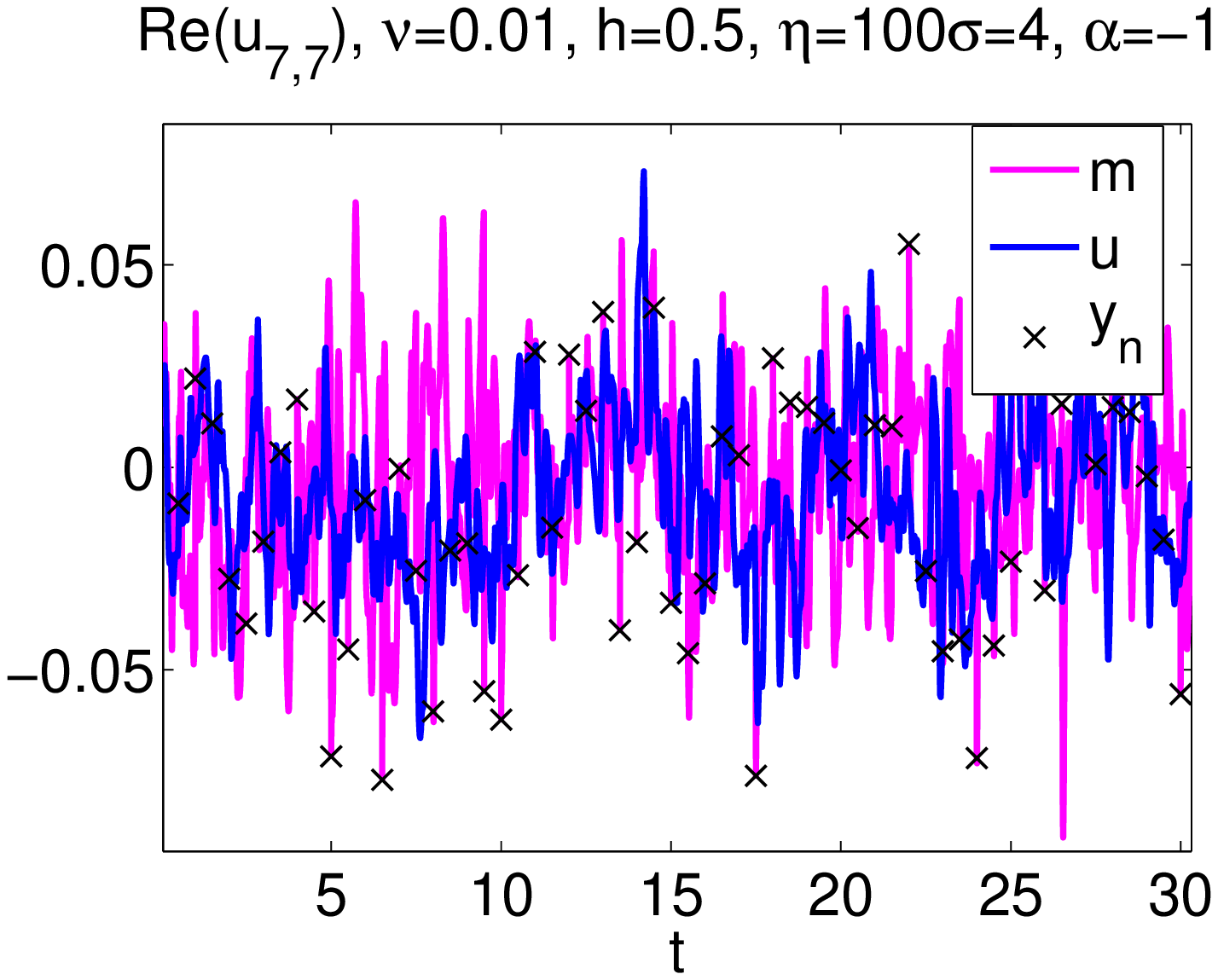}
\includegraphics[width=.45\textwidth]{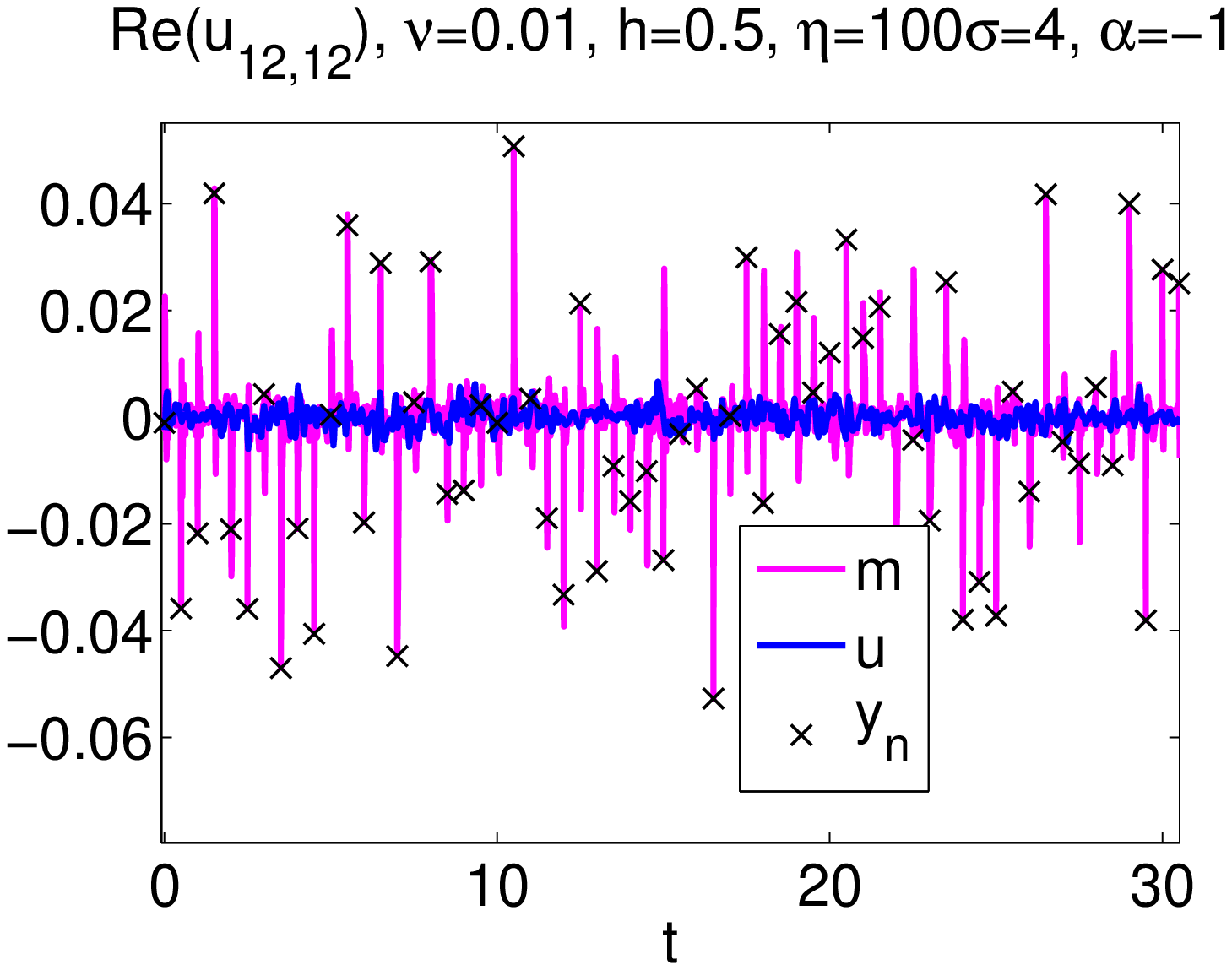}
\caption{Example of a destabilized trajectory for 3DVAR with 
the same parameters as in Fig. \ref{a1.100} except with
value of $\alpha=-1$.  Panels are the same.}
\label{am1.100}
\end{figure*}

\subsection{Partial Observations; Discrete Time}
\label{ssec:partial}

We now proceed to examine the case of partial observations,
illustrating Theorem \ref{t:m}.
Note that the forced mode has magnitude $|k_f|^2 =50$,
so ensuring that it is observed requires 
that $\lambda > 50\lambda_1$.
When enough modes are retained, for example when 
$\lambda=100\lambda_1$ in our setting, 
the results for the $\alpha=1$ case remain roughly 
the same and are not shown.  However, in the case $\alpha=-1$, 
in which the observations are trusted more than
the model at high wavevectors, 
the results are greatly improved by ignoring the observations of the
high-frequencies.  See Fig. \ref{am1.10.p10}.  
This improvement, and indeed the improvement beyond setting $B=0$
for both cases $\alpha=\pm 1$ disappears as $\lambda$ is decreased.
In particular, when $\lambda=25 \lambda_1$ the error is never very 
much smaller than the upper bound.  
This is due to the fact that the
dynamics of the low wavevectors tend to be unpredictable and
they contain very little useful information for the assimilation.
Then, for much smaller $\lambda=4 \lambda_1$,
once enough unstable modes are left unobserved,
there is no convergence.  
The order of magnitude of the error in the asymptotic 
regime as a function
of $\eta$ remains roughly consistent as $\lambda$ is decreased
until the estimator no longer converges.  
For small $h$ (high-frequency in time observations)
and complete observations,
the estimator can be slow to converge to the 
asymptotic regime.
In this case, the number of iterations until convergence, for a sufficiently 
small $\eta$, becomes significantly larger as $\lambda$ is decreased
(again until the estimator fails to converge at all).

Given $\lambda \approx k_\lambda^2\lambda_1$, we expect that
for $\eta$ sufficiently small the contribution of the model
to the filter will be negligible for all 
$k$ with $|k|<k_\lambda$ for $\alpha=1$.  Hence
the estimators for both $\alpha=\pm 1$ will behave similarly. 
An example of this is shown in Fig. \ref{a1am1.7.p01}
where $\eta=0.01\sigma=0.0004$ 
and $\lambda=49 \lambda_1$ 
in Fig. \ref{a1am1.7.p01}.  In both cases, the estimator is
essentially utilizing all the available observations. 
There are enough observations to draw the 
higher wavevectors of the estimator
closer to the truth than if we just set the population of 
those modes to zero.  
In contrast, as mentioned above, when $\lambda=25 \lambda_1$,
there are not enough observations even when they are all used, 
and the error is roughly the same as the upper bound 
as $\eta \rightarrow 0$ (not shown).

\begin{figure*}
\includegraphics[width=.45\textwidth]{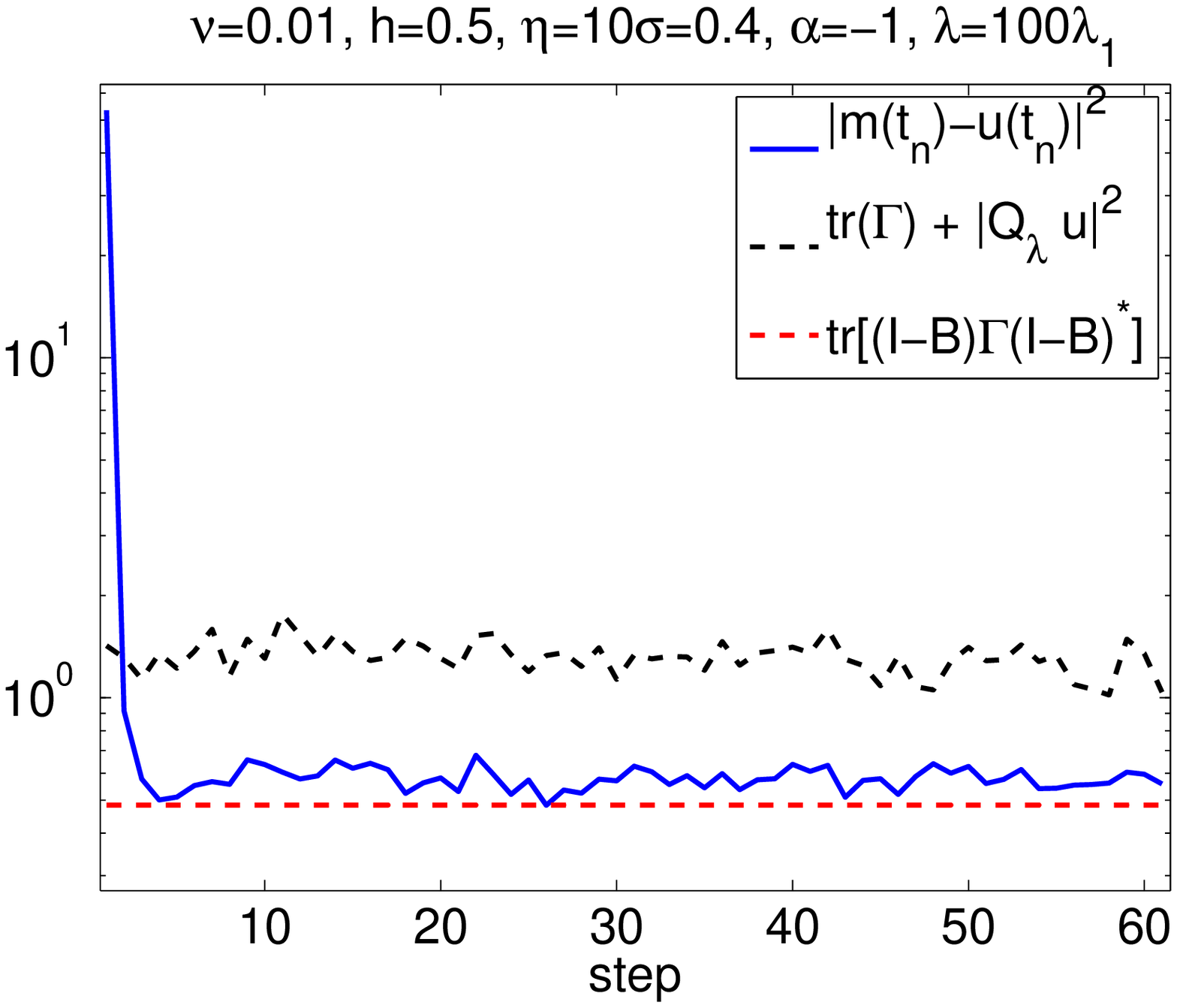}
\includegraphics[width=.45\textwidth]{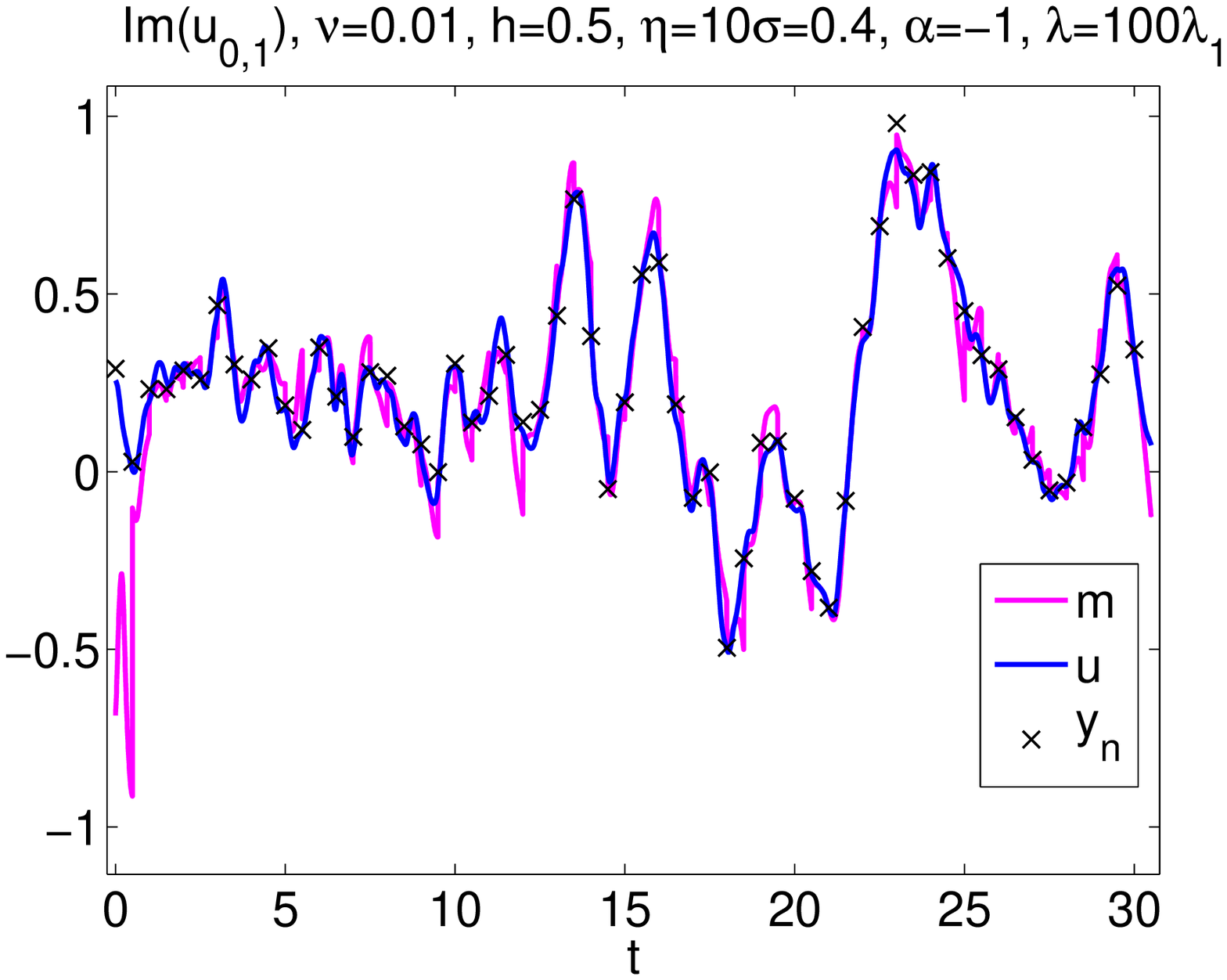}
\includegraphics[width=.45\textwidth]{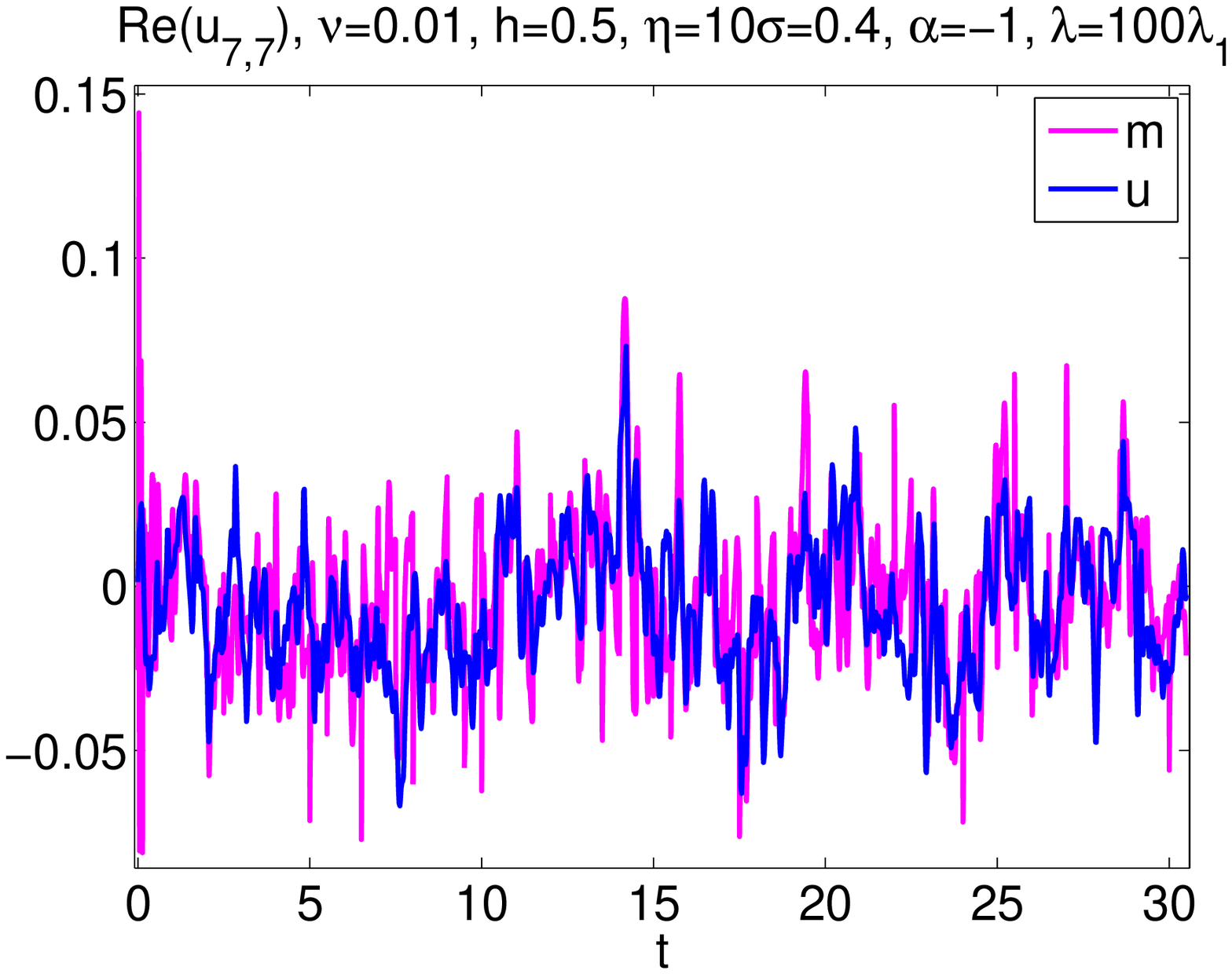}
\includegraphics[width=.45\textwidth]{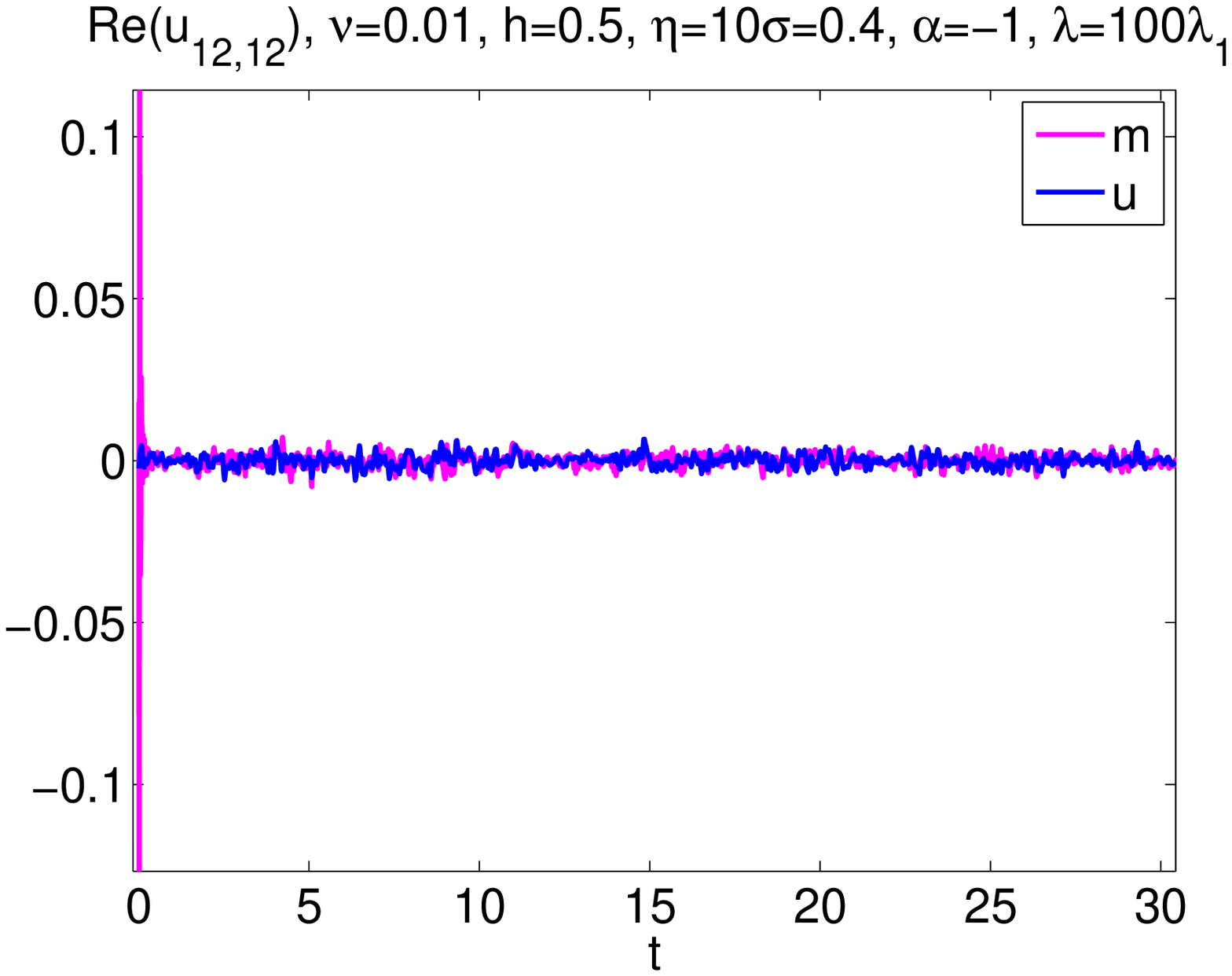}
\caption{Example of an improved estimator for partial observations
  with $\lambda=100 \lambda_1$ and otherwise the  
same parameters as in Fig. \ref{am1.10}. Panels are the same.}
\label{am1.10.p10}
\end{figure*}


\begin{figure*}
\includegraphics[width=.45\textwidth]{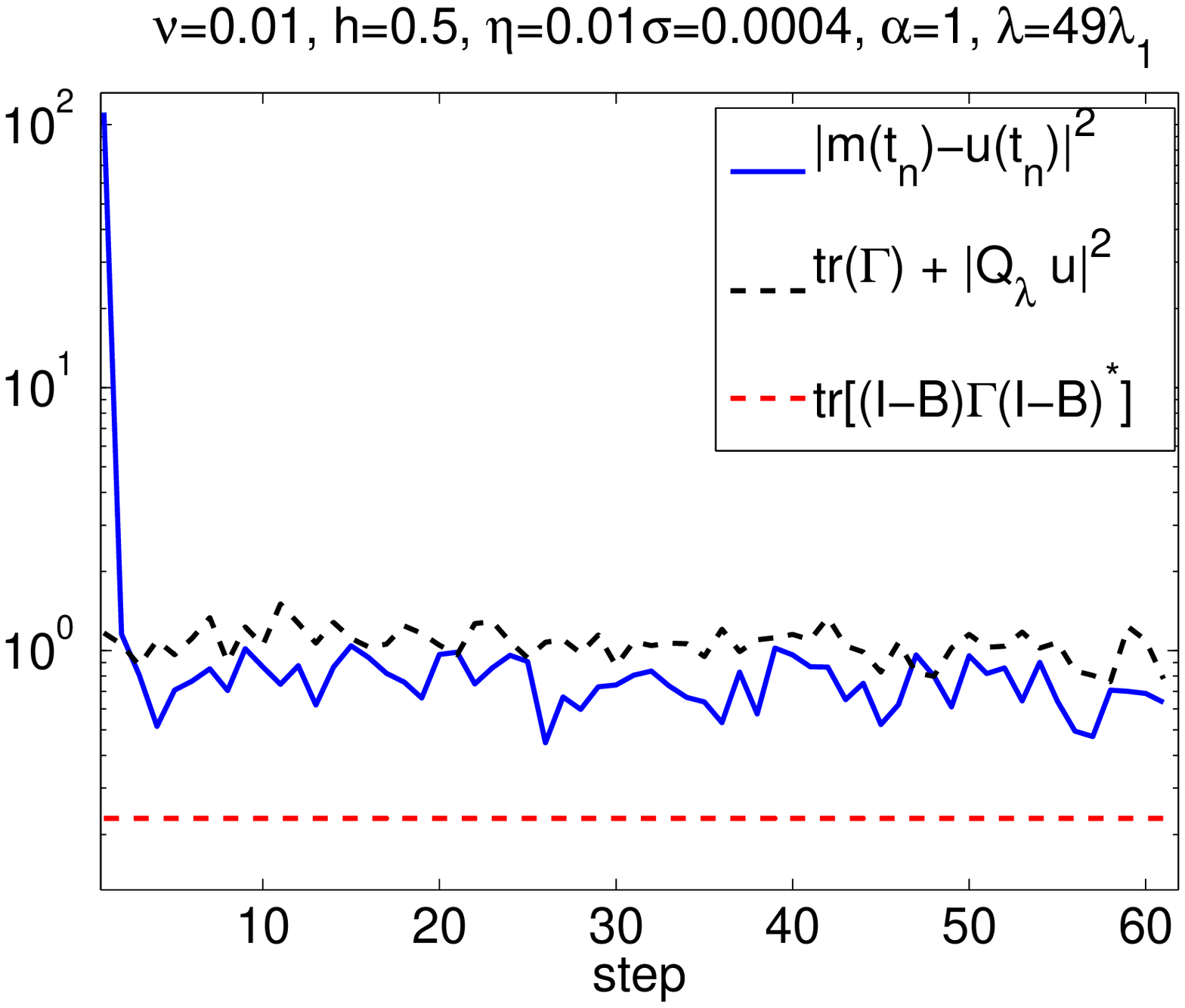}
\includegraphics[width=.45\textwidth]{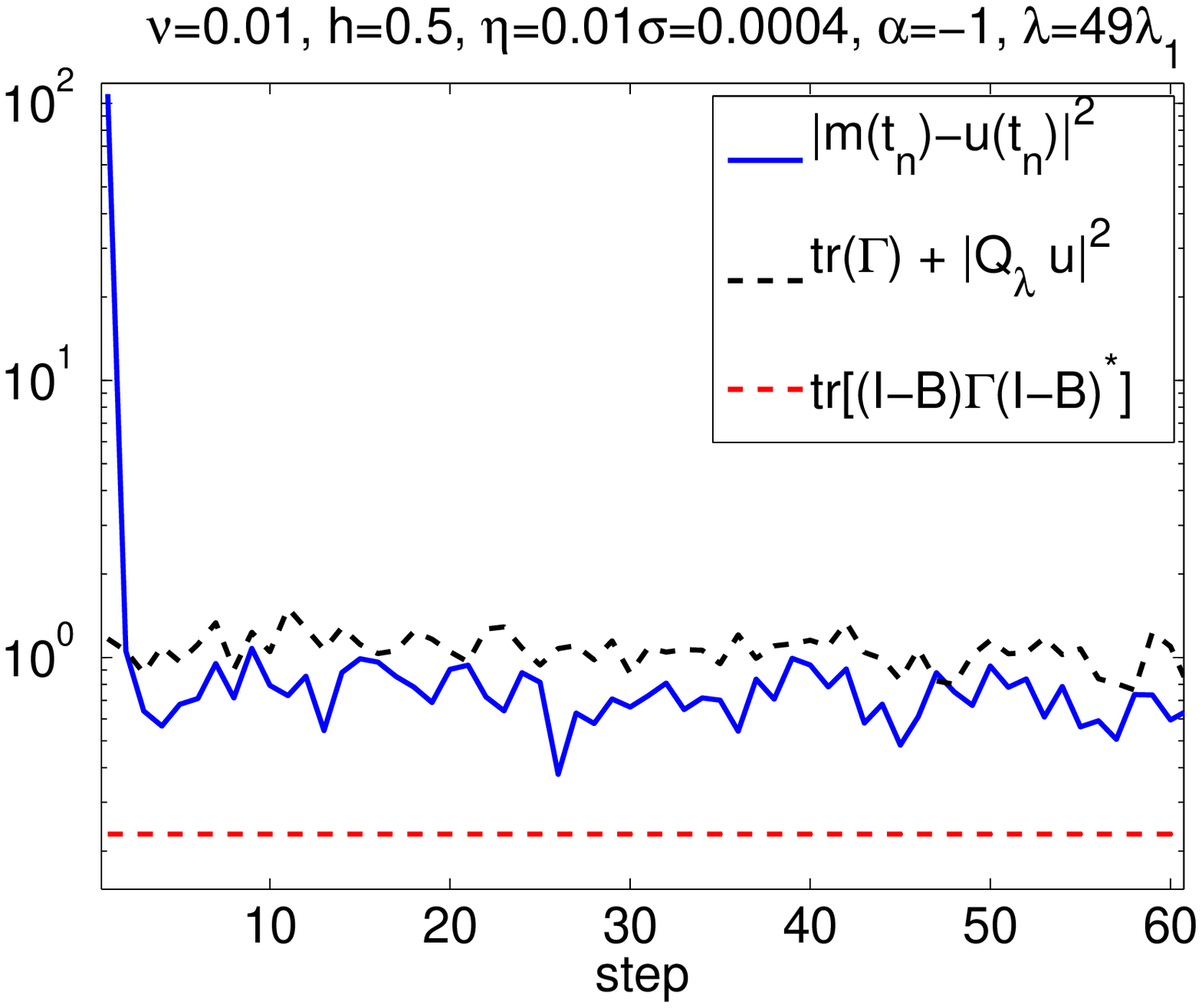}
\includegraphics[width=.45\textwidth]{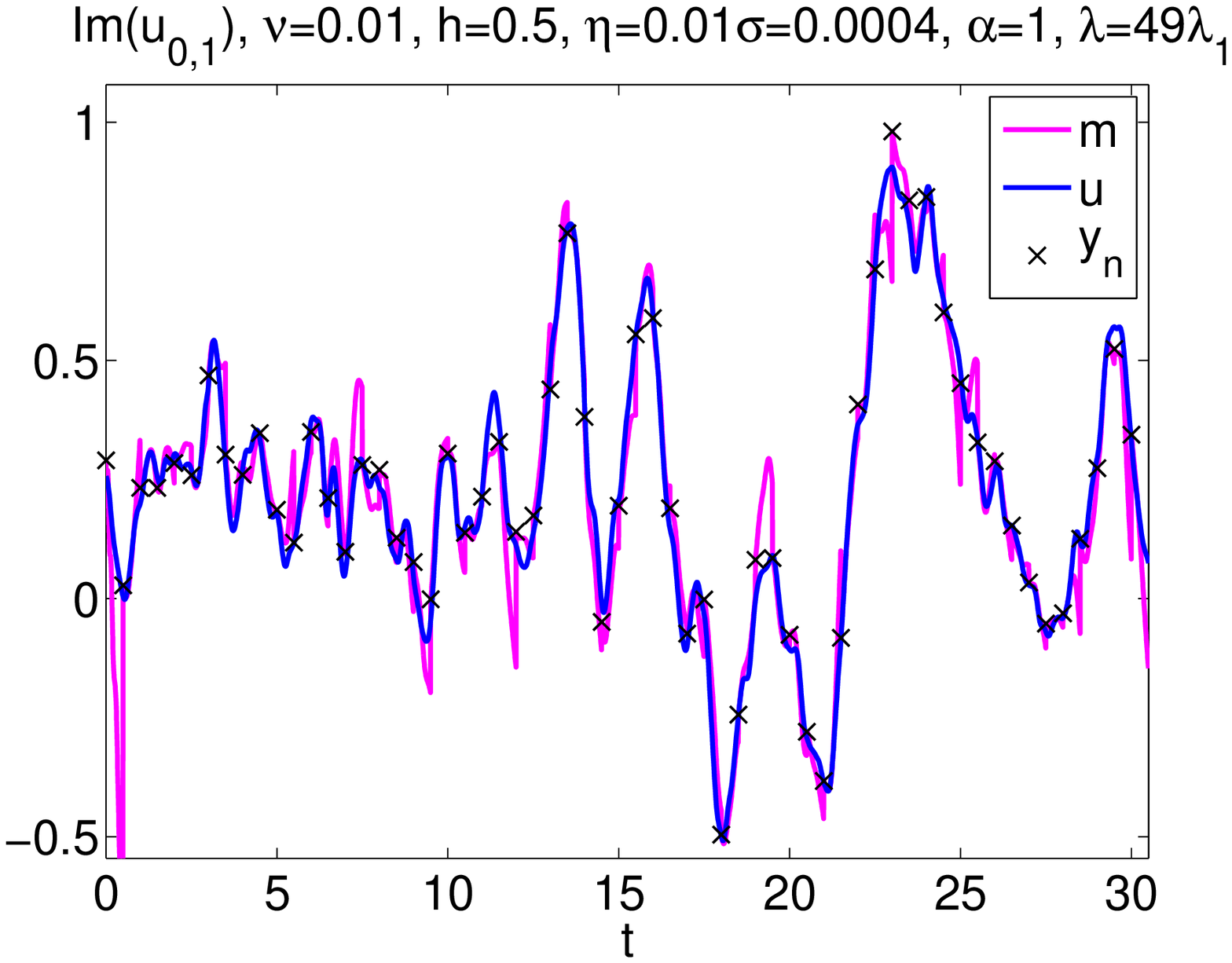}
\includegraphics[width=.45\textwidth]{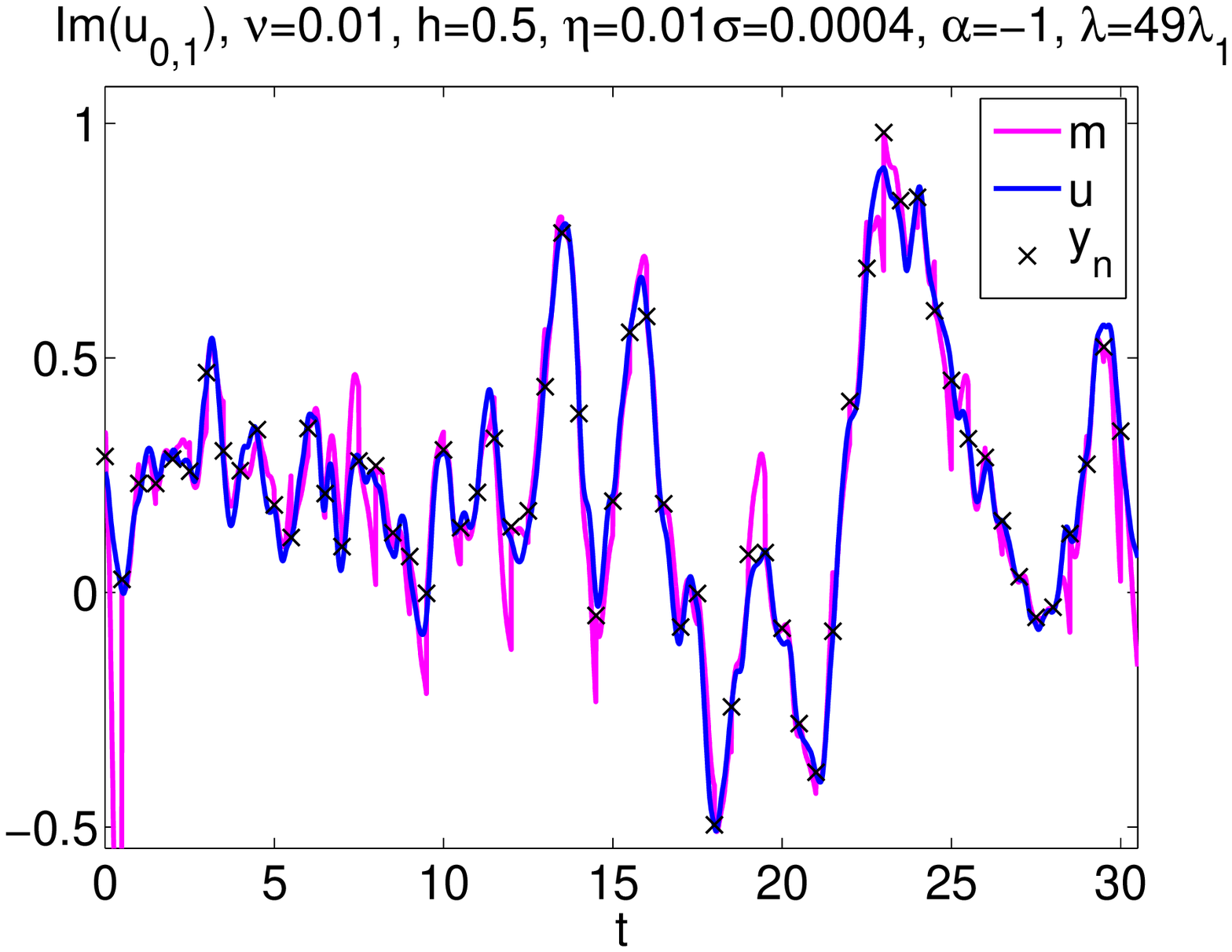}
\caption{Examples of estimators for partial observations 
with $\lambda=49 \lambda_1^2$ and $\eta=0.1\sigma=0.004$, otherwise the  
same parameters as in Figs. \ref{a1.1} and \ref{am1.1}. Left panels
are for $\alpha=1$ and right panels are for $\alpha=-1$.}
\label{a1am1.7.p01}
\end{figure*}

\subsection{Continuous Observations}
\label{ssec:continuous}

Finally, we explore the case of continuous 
observations using the SPDE and PDE derived in
subsection \ref{ssec:spde}. We let $\alpha=1/2$ throughout.
We invoke a split-step scheme to solve Eqns. \ref{eq:nse2}
and \ref{eq:nse3} in which we compose solutions of (\ref{eq:nse})
and the Ornstein-Uhlenbeck process
\begin{equation}
\frac{\rd \hu}{\rd t} + 
\omega A_0^{-2\alpha}(\hu-u)= \omega \sigma_0
A_0^{-2\alpha-\beta}\frac{\rd W}{\rd t}, 
\quad \hu(0)=\hu_0,
\label{eq:nse222}
\end{equation}
at each step, when $r=1$, and setting $\sigma_0=0$ 
in (\ref{eq:nse222}) when $r<1$.
We begin by considering the $r=1$ case in which we recover
the SPDE (\ref{eq:nse2}). Notice that the parameter $\omega$
sets a time-scale for relaxation towards the true signal,
and  $\sigma_0$ sets a scale for the size of flucuations
about the true signal. The parameter $\beta$
rescales the fluctuation size at different wavevectors
with respect to the relaxation time.
First we consider setting $\beta=0.$
In Fig. \ref{a1c5.t} we show numerical experiments 
with $\omega=100$ and $\sigma_0=0.05.$ 
We see that the noise level
on top of the signal in the low modes is almost $O(1)$, and
that the high modes do not synchronize at all; the total error 
remains $O(1)$ although trends in the signal are followed.  
On the other hand, for the smaller value of $\sigma_0=0.005$,
still with $\omega=100$,
the noise level on the signal in the low modes is moderate,
the high modes synchronize sufficiently well, 
and the total error is small. See Fig. \ref{a1c05.t}.

Now we consider the case $\beta=1$. 
Again we take $\omega=100$ and
$\sigma_0=0.05$ and $0.005$
in Figures \ref{a1dc5.t}
and \ref{a1dc05.t}, respectively.  
The synchronization is stronger than
that observed for $\beta=0$ in each case. This is because 
the noise decays more rapidly for large wavevectors when
$\beta$ is increased, as can be observed in the relatively 
smooth trajectories of the high modes of the estimator.

For the case when $r<1$ and we recover a PDE, the values of
$\sigma_0$ and $\beta$ are irrelevant.  The value of $\omega$
is the critical parameter in this case. 
For values of $\omega$ of $O(100)$ the 
convergence is exponentially
fast to machine precision.  For values of $\omega$ of 
$O(1)$ the estimator does not exhibit stable
behviour.  For intermediate values, 
the estimator may approach the signal and remain bounded and 
still an $O(1)$ distance away(see the case $\omega=10$ 
in Fig. \ref{d10.t}), or else 
it may come close to synchronizing (see the case $\omega=30$
in Fig. \ref{d30.t}).

\begin{figure*}
\includegraphics[width=1\textwidth]{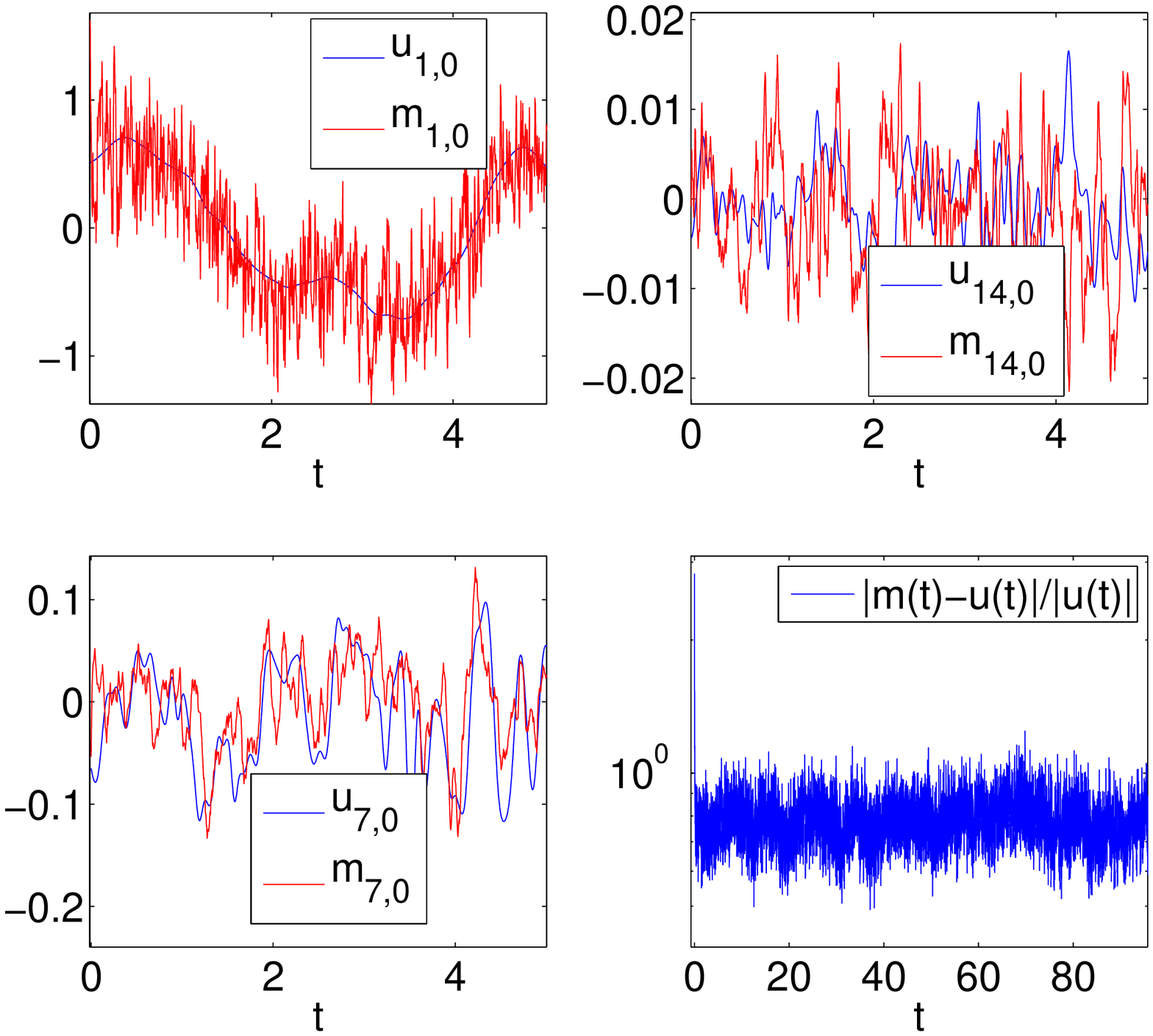}
\caption{Trajectories of various modes of the estimator $\hu$ 
and the signal $u$ are depicted above for 
$\beta=0$ and $\sigma_0=0.05$, along with the total relative 
error in the $l^2$ norm, $|\hu-u|/|u|$.}
\label{a1c5.t}
\end{figure*}


\begin{figure*}
\includegraphics[width=1\textwidth]{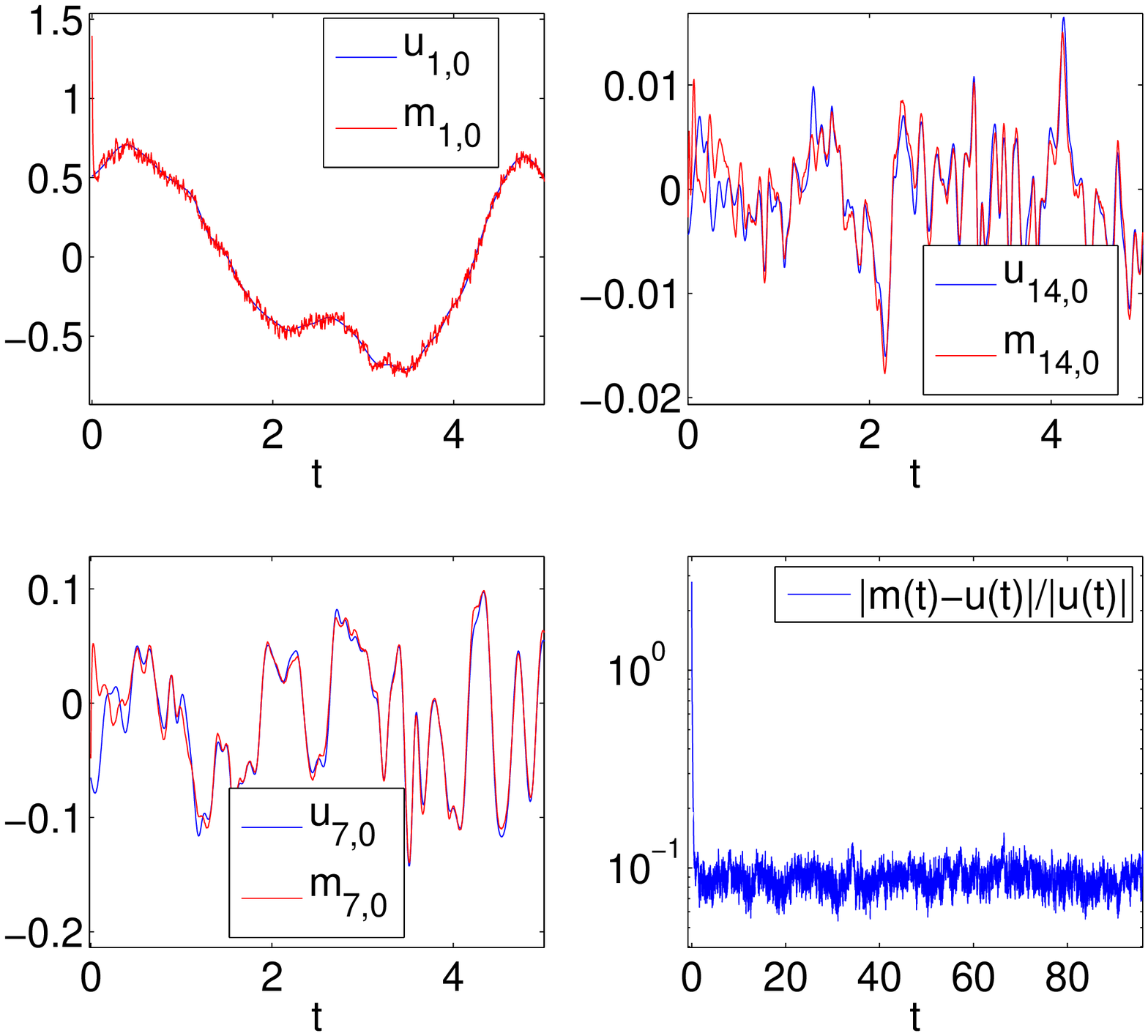}
\caption{Trajectories of various modes of the estimator $\hu$ 
and the signal $u$ are depicted above for $\beta=0$ and 
$\sigma_0=0.005$, along with the relative error.}
\label{a1c05.t}
\end{figure*}


\begin{figure*}
\includegraphics[width=1\textwidth]{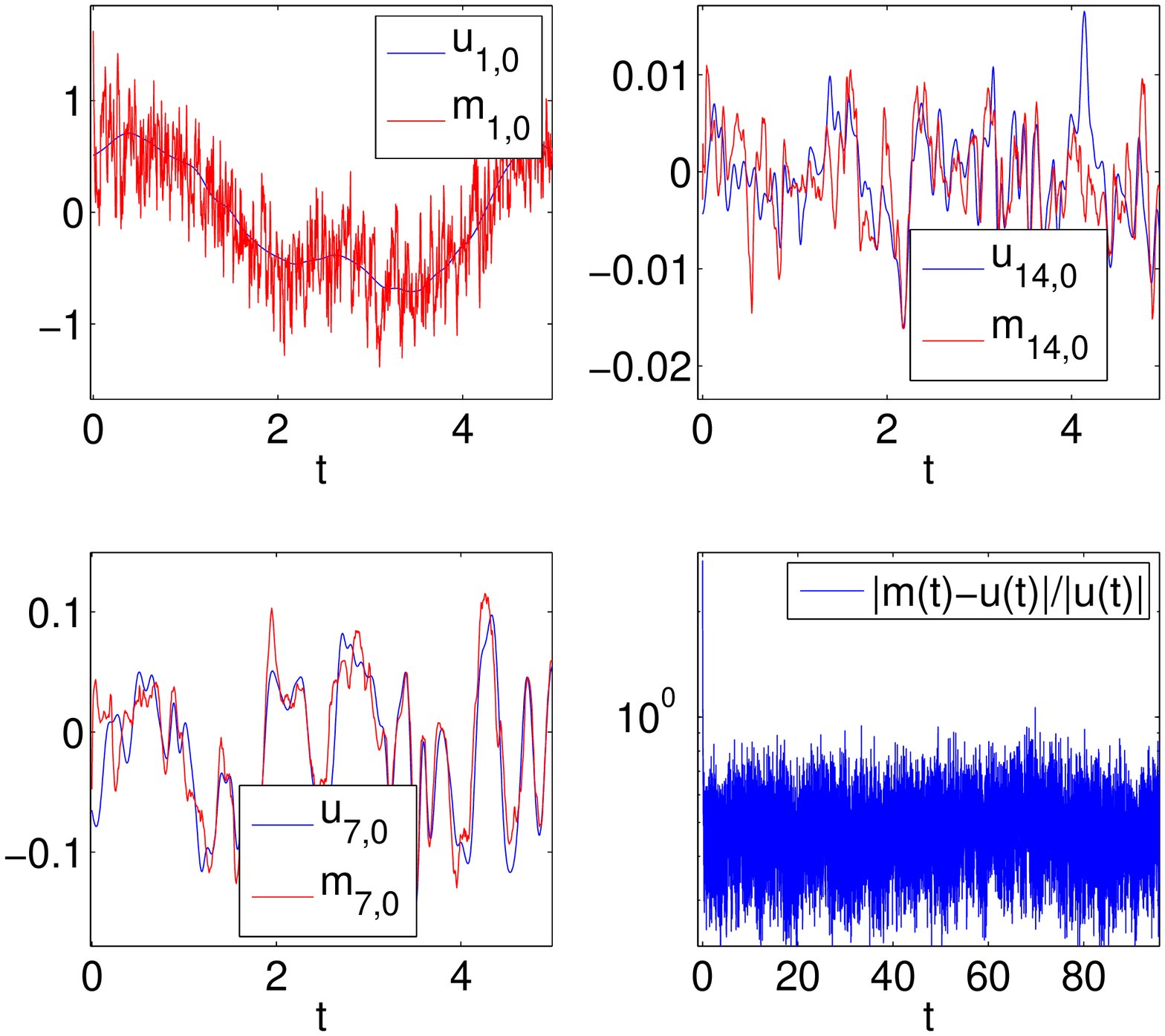}
\caption{Trajectories of various modes of the estimator $\hu$ 
and the signal $u$ are depicted above for 
$\beta=1$ and $\sigma_0=0.05$, along with the relative error.}
\label{a1dc5.t}
\end{figure*}


\begin{figure*}
\includegraphics[width=1\textwidth]{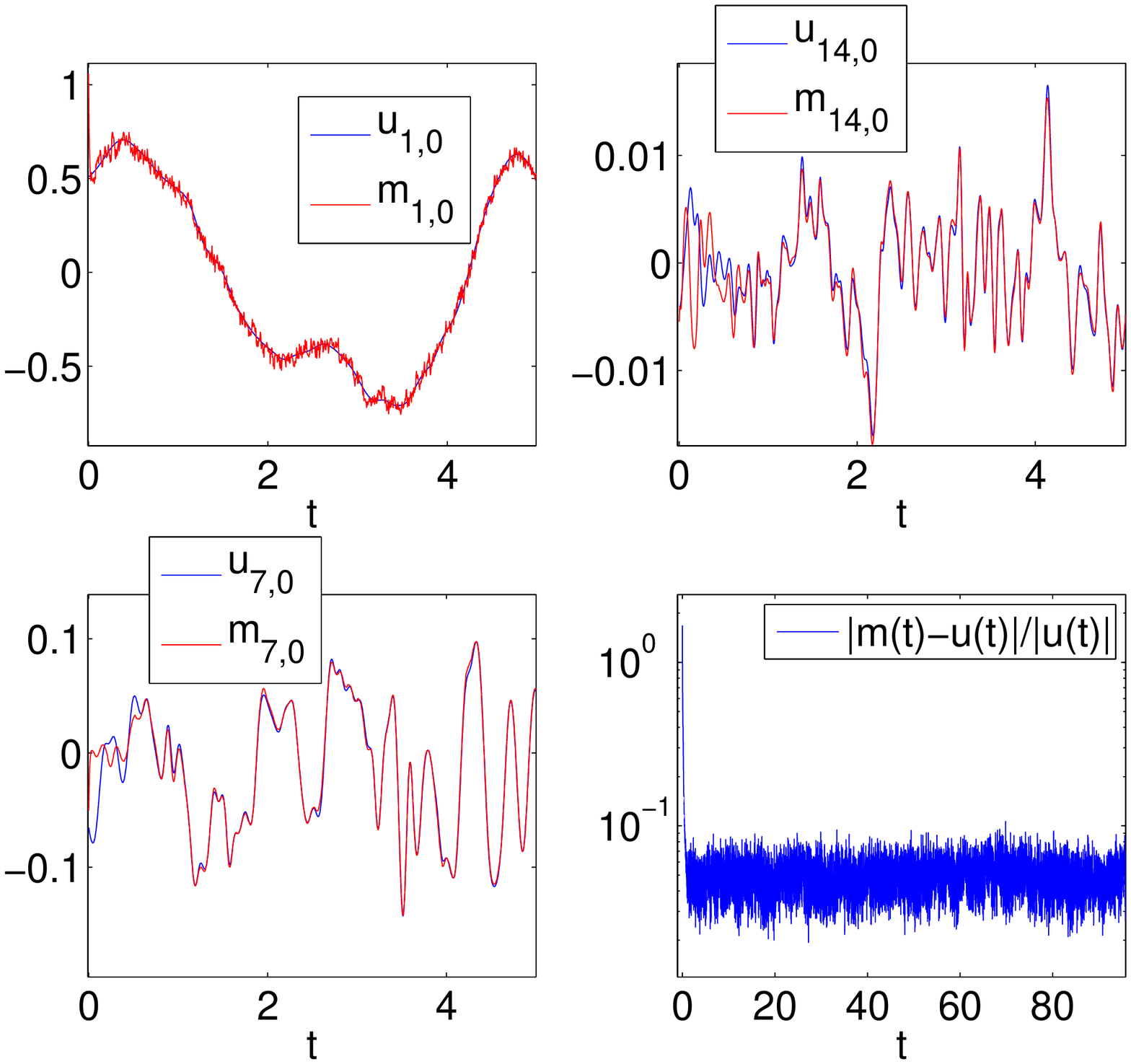}
\caption{Trajectories of various modes of the estimator $\hu$ 
and the signal $u$ are depicted above for $\beta=1$ and 
$\sigma_0=0.005$, along with the relative error.}
\label{a1dc05.t}
\end{figure*}


\begin{figure*}
\includegraphics[width=1\textwidth]{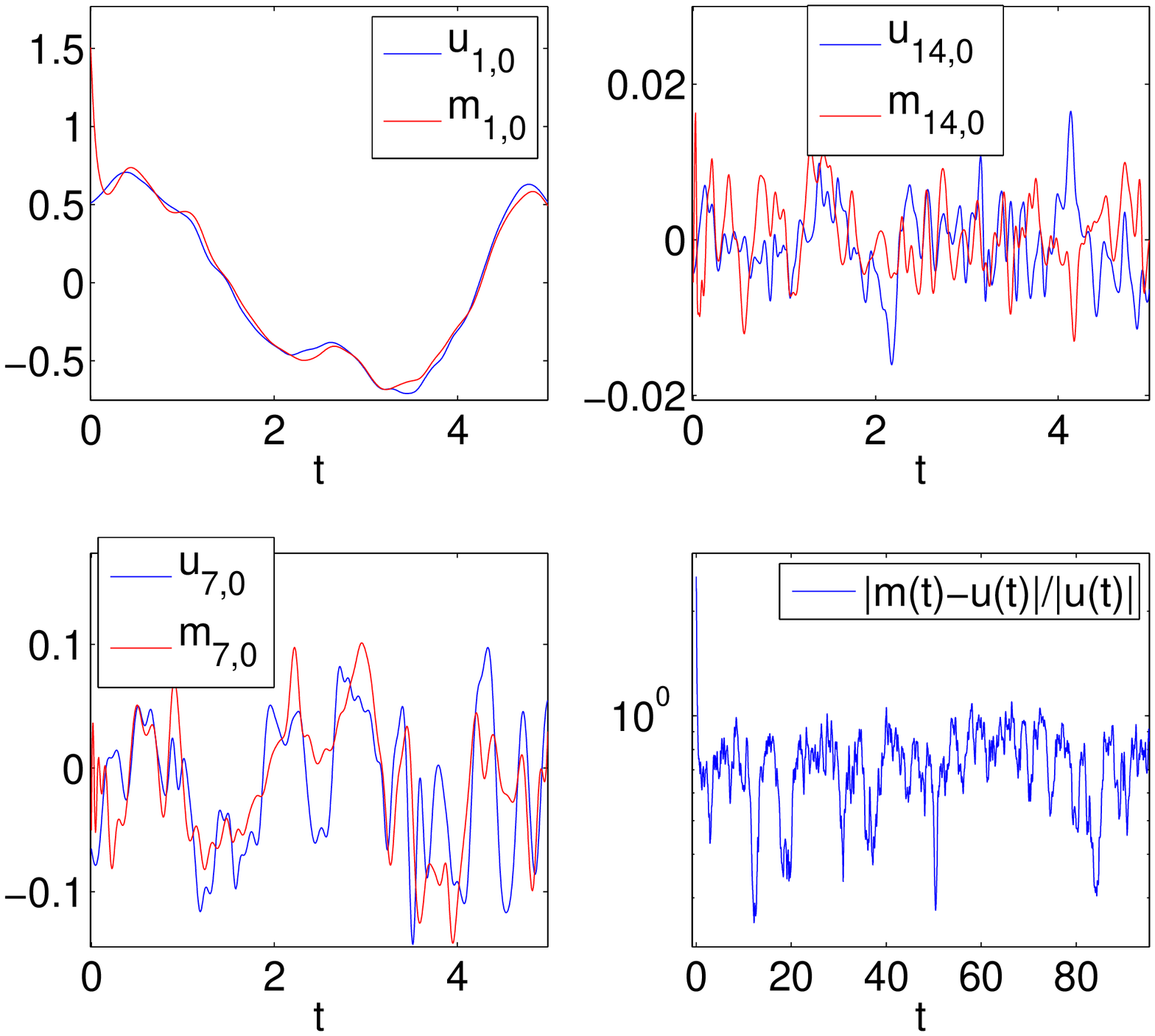}
\caption{Trajectories of various modes of the estimator $\hu$ 
and the signal $u$ are depicted above for 
$r<1$ and $\omega=10$, along with the relative error.}
\label{d10.t}
\end{figure*}


\begin{figure*}
\includegraphics[width=1\textwidth]{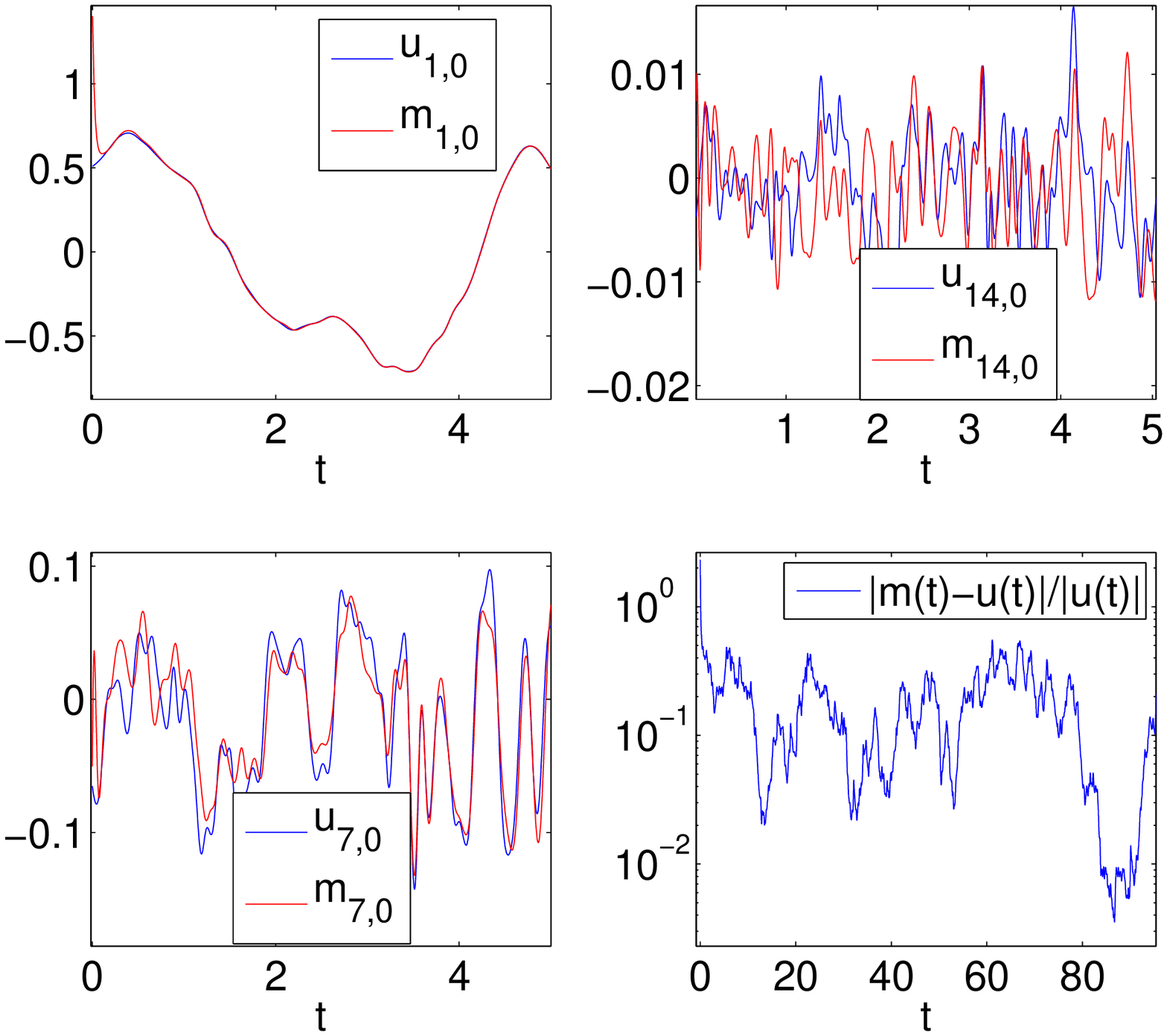}
\caption{Trajectories of various modes of the estimator $\hu$ 
and the signal $u$ are depicted above for 
$r<1$ and $\omega=30$, along with the relative error.}
\label{d30.t}
\end{figure*}


\section{Conclusion}
\label{sec:conclusions}

This paper contains three main components:

\begin{itemize}

\item we show that the filtering problem for the
Navier-Stokes equation may be formulated as a
well-posed inverse problem in function space,
and demonstrated how various approximate Gaussian
filters can be derived (Theorems \ref{thm:z} and
\ref{t:app},
and Corollary \ref{cor:z});

\item we prove Theorems \ref{t:m} and \ref{t:mz},
which establish filter stability 
for small enough observational noise;

\item we derive an SPDE (\ref{eq:nse2}) and a PDE
(\ref{eq:nse3}), which may be used as the basis to
study filter stability in the case of frequent observations
and large observational noise.

\end{itemize}
Numerical results are used to illustrate, and extend the
validity of, the theory.  

\noindent There are a number of natural future directions
which stem from this work:

\begin{itemize}

\item to develop analogous filter stability theorems
for more sophisticated filters, such as the extended
and ensemble-based methods, when applied to the
Navier-Stokes equation;

\item to rigorously derive, and study the properties of,
the SPDE (\ref{eq:nse2}) which characterizes
filter stability in the case of frequent observations
and large observational noise;

\item to study model-data mismatch by looking at filter
stability for data generated by forward models which
differ from those used to construct the filter; 

\item to study the effect of filtering in the presence of
model error, by similar methods, to understand how this
may be used to overcome problems arising from model-data mismatch.

\end{itemize}

\vspace{0.1in}

\noindent{\bf Acknowledgements}
{AMS would like to thank the following institutions for
financial support: EPSRC, ERC and ONR; KJHL was supported
by EPSRC and ONR; and 
CEAB, KFL, DSM and MRS were supported EPSRC, through the
MASDOC Graduate Training Centre at Warwick University. 
The authors also thank The Mathematics Institute and Centre for Scientific
Computing at Warwick University for supplying valuable computation
time.  Finally, the authors
thank Masoumeh Dashti for valuable input.}


\bibliographystyle{plain}

\bibliography{fsbib}

\end{document}